\documentclass{amsart}

\usepackage[margin=1.5in]{geometry}

\usepackage{graphicx}
\usepackage[centertags]{amsmath}
\usepackage{amsfonts}
\usepackage{amssymb}
\usepackage{amsthm}
\usepackage{newlfont}
\usepackage{hyperref}
\usepackage{caption}

\allowdisplaybreaks

\newtheorem{thm}{Theorem}[section]
\newtheorem*{thm*}{Theorem}
\newtheorem{cor}[thm]{Corollary}
\newtheorem{lem}[thm]{Lemma}
\newtheorem{prop}[thm]{Proposition}

\theoremstyle{definition}
\newtheorem{defn}[thm]{Definition}

\theoremstyle{remark}
\newtheorem*{rem}{Remark}

\DeclareSymbolFont{bbold}{U}{bbold}{m}{n}
\DeclareSymbolFontAlphabet{\mathbbold}{bbold}

\newcommand{\T}[1]{\bar{#1}}

\newcommand{\rZ}{\mathcal{Z}}
\newcommand{\bdim}{\mathrm{L}}

\newcommand{\N}{\mathbb{N}}
\newcommand{\Z}{\mathbb{Z}}
\newcommand{\R}{\mathbb{R}}

\newcommand{\cH}{\mathcal{H}}
\newcommand{\bOne}{\mathbbold{1}}
\newcommand{\acts}{\curvearrowright}
\newcommand{\Stab}{\mathrm{Stab}}
\newcommand{\Sub}{\mathrm{Sub}}

\newcommand{\dom}{\mathrm{dom}}
\newcommand{\Rect}{\mathrm{Rec}}

\renewcommand{\:}{\,:\,}
\newcommand{\res}{\restriction}

\makeatletter
\def\Ddots{\mathinner{\mkern1mu\raise\p@
\vbox{\kern7\p@\hbox{.}}\mkern2mu
\raise4\p@\hbox{.}\mkern2mu\raise7\p@\hbox{.}\mkern1mu}}
\makeatother

\begin{document}

\title[Locally nilpotent groups and hyperfinite equivalence relations]{Locally nilpotent groups and hyperfinite \\ equivalence relations}

\author{Scott Schneider and Brandon Seward}
\address{Department of Mathematics, University of Michigan, 530 Church Street, Ann Arbor, MI 48109, U.S.A.}
\email{sschnei@umich.edu, b.m.seward@gmail.com}
\thanks{The second author was supported by the National Science Foundation Graduate Student Research Fellowship under Grant No. DGE 0718128.}

\begin{abstract}
A long standing open problem in the theory of hyperfinite equivalence relations asks if the orbit equivalence relation generated by a Borel action of a countable amenable group is hyperfinite. In this paper we show that this question has a positive answer when the acting group is locally nilpotent. This extends previous results obtained by Gao--Jackson for abelian groups and by Jackson--Kechris--Louveau for finitely generated nilpotent-by-finite groups. Our proof is based on a mixture of coarse geometric properties of locally nilpotent groups together with an adaptation of the Gao--Jackson machinery.
\end{abstract}
\maketitle


\section{Introduction} \label{SEC INTRO}

This paper is a contribution to the project of determining which countable groups have the property that all of their orbit equivalence relations are hyperfinite. An equivalence relation $E$ on the Polish space $X$ is \emph{Borel} if it is a Borel subset of $X\times X$, and \emph{finite} if each $E$-class is finite. $E$ is \emph{hyperfinite} if it is the union of an increasing sequence of finite Borel equivalence relations. It is well-known (cf.\ \cite{Weiss84}, \cite{SS88}, \cite{DJK}) that the Borel equivalence relation $E$ on $X$ is hyperfinite if and only if $E$ arises as the orbit equivalence relation of some Borel action of $\Z$ on $X$.

In the measure-theoretic context, Ornstein and Weiss \cite{OW80} proved that if the countable amenable group $G$ acts in a Borel fashion on the standard Borel space $X$, then the resulting orbit equivalence relation $E^X_G$ is \emph{$\mu$-a.e.\ hyperfinite} for any Borel probability measure $\mu$ on $X$. This just means that for any such measure $\mu$ there is an invariant $\mu$-conull subset $Y\subseteq X$ such that the restriction of $E^X_G$ to $Y$ is hyperfinite. Taking this result as a guide, Weiss \cite{Weiss84} asked whether the orbit equivalence relation arising from a Borel action of a countable amenable group is hyperfinite in the purely Borel setting. He proved that this is the case for Borel actions of $\Z$ \cite{Weiss84}, and since then significant attention has been focused on the project of extending this result to the widest possible class of groups. By a result from the folklore of the field that was first recorded in \cite{JKL}, if every orbit equivalence relation arising from a Borel action of $G$ is hyperfinite, then $G$ must be amenable.

In unpublished work shortly after \cite{Weiss84}, Weiss generalized his result on $\Z$-actions to show that Borel actions of the groups $\Z^n$ give rise to hyperfinite orbit equivalence relations. In 1988, Slaman and Steel \cite{SS88} independently proved Weiss's result for $\Z$-actions by constructing \emph{Borel marker sets}, thus introducing a technique that would go on to play a central role in hyperfiniteness arguments. Then in 2002, Jackson, Kechris, and Louveau generalized all previous results on the problem by proving that finitely-generated groups of polynomial growth have hyperfinite orbit equivalence relations \cite{JKL}.  By Gromov's Theorem \cite{G}, these are exactly the finitely-generated nilpotent-by-finite groups. They also showed that orbit equivalence relations arising from Borel actions of a certain class of uncountable locally compact Polish groups are Borel reducible to hyperfinite equivalence relations (see \cite[1.16]{JKL}).

In an important breakthrough in 2007, Gao and Jackson \cite{GJ} proved that orbit equivalence relations arising from Borel actions of countable abelian groups are hyperfinite, thus eliminating for the first time the hypothesis of finite generation. Their proof extended the use of Borel marker sets initiated by Slaman-Steel, made essential use of the geometry of $\Z^n$, introduced the notion of \emph{anti-coherent}, or \emph{orthogonal} pairs of equivalence relations, and involved an intricate multi-scale inductive construction. Since every countable group is the increasing union of finitely generated subgroups, their result may be interpreted as solving a special case of the widely known \emph{union problem} which asks whether the union of an increasing sequence of hyperfinite equivalence relations is again hyperfinite.

In this paper we will show how to extend the techniques of \cite{GJ} in order to remove the assumption of finite generation for the class of nilpotent groups. In particular we show how finitely-generated nilpotent groups can be viewed locally as having ``roughly-rectangular" geometry similar to that of finitely-generated abelian groups, and we transfer the theory of orthogonal equivalence relations developed in \cite{GJ} to the new setting. Since all our constructions are local in nature, we actually obtain hyperfiniteness of orbit equivalence relations arising from Borel actions of countable locally nilpotent groups. A group is \emph{locally nilpotent} if every finitely generated subgroup is nilpotent.

\begin{thm}
If the countable locally nilpotent group $G$ acts in a Borel fashion on the standard Borel space $X$, then the induced orbit equivalence relation is hyperfinite.
\end{thm}

Proving this theorem for non-free actions requires addressing a new difficulty that is not present for abelian groups or finitely-generated nilpotent-by-finite groups. Namely, the conjugacy relation on the space of subgroups of a countable nilpotent group is in general non-smooth. This means that if a Borel action of a nilpotent group has non-trivial stabilizers, one cannot necessarily choose a distinguished stabilizer from each orbit in a Borel manner. Overcoming this problem leads to significant structural differences between our proof for nilpotent groups and the Gao--Jackson proof for abelian groups when the action is not free.

The question of whether orbit equivalence relations of countable amenable groups are necessarily hyperfinite has received attention not only in the measure and Borel settings, but also in the topological setting. Here it takes the form of asking whether for every minimal continuous action of an amenable group on a Cantor space there is a continuous action of $\Z$ with the same orbits. This question remains open, and currently the best positive result was obtained for actions of $\Z^n$ by Giordano, Matui, Putnam, and Skau \cite{GMPS}. Motivated by this perspective, we attempt to obtain our hyperfiniteness result in a topologically favorable way, although we are unable to supply a positive answer for nilpotent groups to the topological version of Weiss's question. The following theorem extends a similar result by Gao and Jackson \cite{GJ}. Here $E_0$ is the equivalence relation of eventual agreement of binary sequences.

\begin{thm}
If the countable locally nilpotent group $G$ acts freely and continuously on the zero-dimensional Polish space $X$, then the induced orbit equivalence relation is continuously embeddable into $E_0$.
\end{thm}

The rest of this paper is organized as follows.  In Section \ref{SEC PRELIM} we establish notation and collect together some basic facts about Borel equivalence relations, group actions, and locally nilpotent groups.  In Section \ref{SEC MARKER} we prove a Borel marker lemma similar to \cite[Lemma 2.1]{GJ}, and we discuss the notion of a \emph{$G$-clopen} relation and its connection to continuous reductions. In Section \ref{SEC GEOM} we discuss certain aspects of the geometry of abelian groups and introduce a device called a \emph{chart} that will enable us to transfer this geometry to nilpotent groups. In Sections \ref{SEC RECT} and \ref{SEC FACES} we carry out this transfer, introducing the notions of \emph{rough rectangle} and \emph{facial boundary} in groups admitting charts. In Section \ref{SEC ORTHO} we define \emph{orthogonal equivalence relations} and prove the main technical lemma of the paper which is used for building them. Section \ref{SEC STAB} is devoted to the conjugacy equivalence relation on subgroups and the new difficulty presented by non-free actions of general (i.e., non-finitely-generated) nilpotent groups.  Finally, in Section \ref{SEC FINAL} we prove our main results.

\section{Preliminaries} \label{SEC PRELIM}

\subsection{Descriptive set theory and equivalence relations}

A \emph{Polish space} is a separable, completely metrizable topological space. Familiar examples of Polish spaces include $\R$, any countable discrete space, Cantor space $2^\N$ of binary sequences, and Baire space $\N^\N$ of sequences of natural numbers. Cantor space and Baire space are \emph{zero-dimensional} Polish spaces, meaning that they admit a base of clopen sets.

An equivalence relation $E$ on the Polish space $X$ is \emph{Borel} (\emph{closed}, $F_\sigma$, \emph{etc}) if it is Borel (closed, $F_\sigma$, etc) as a subset of $X\times X$. $E$ is \emph{countable} if every equivalence class is countable, and \emph{finite} if every class is finite. If $E$ and $F$ are Borel equivalence relations on the standard Borel spaces $X$ and $Y$, a \emph{reduction} from $E$ to $F$ is a function $f:X\to Y$ such that for all $x,y\in X$,
$$x\mathrel{E}y \ \Leftrightarrow \ f(x)\mathrel{F}f(y).$$
We say that $E$ is \emph{Borel reducible} to $F$, written $E\leq_BF$, if there is a Borel reduction from $E$ to $F$, and that $E$ is \emph{Borel embeddable} in $F$, written $E\sqsubseteq_BF$, if there is an injective Borel reduction (i.e., an \emph{embedding}) from $E$ to $F$. The notions of \emph{continuously reducible} and \emph{continuously embeddable} are defined analogously.

Sometimes we will want to ignore topological considerations and focus solely on the Borel setting. A \emph{standard Borel space} is a measurable space $(X,\mathcal{B})$ such that $\mathcal{B}$ arises as the Borel $\sigma$-algebra of some Polish topology on $X$. The notions of Borel equivalence relation and Borel reduction can then be defined just as above in this more general setting. By a classical result of Kuratowski, any two uncountable standard Borel spaces are isomorphic, which has the effect that in the Borel setting we may always work on whatever space is most convenient.

We shall be especially concerned with equivalence relations that arise from countable group actions. Throughout this paper $G$ will always denote a countable group, and we will always view countable groups as discrete topological groups. An action $\alpha:G\times X\to X$ of $G$ on the Polish space $X$ is \emph{continuous} (\emph{Borel}) if the function $\alpha$ is continuous (Borel) on the product space $G\times X$. Since $G$ is countable, this is equivalent to the functions $x\mapsto g\cdot x$ being continuous (Borel) for every $g\in G$.  We will frequently avoid naming actions, writing $G\acts X$ for an action of $G$ on $X$ and $g\cdot x$ for the image of $(g,x)$ when no confusion can arise. All our actions will be on the left. If $H\leq G$ then we let $H$ act on $X$ by restricting the $G$-action.

There is an intimate connection between countable group actions and countable Borel equivalence relations. If the countable group $G$ acts in a Borel fashion on the Polish space $X$, then the resulting \emph{orbit equivalence relation} $E^X_G$ defined by
$$x\mathrel{E^X_G}y \ \Leftrightarrow \ (\exists g\in G)\, g\cdot x=y$$
is a countable Borel equivalence relation. If moreover $G$ acts continuously, then in fact $E^X_G$ is $F_\sigma$. Conversely, by the well-known Feldman-Moore theorem \cite{FM}, for any countable Borel equivalence relation $E$ on the standard Borel space $X$ there is a countable group $G$ and a Borel action $G\acts X$ such that $E=E^X_G$. In this sense the study of countable Borel equivalence relations amounts to the study of orbit equivalence relations of countable groups.

Given an action of $G$ on $X$, the $E^X_G$-equivalence class of $x$ is called the \emph{orbit} of $x$ and is equal to $G\cdot x=\{g\cdot x\:g\in G\}$. In general if $E$ is any equivalence relation on $X$, we write $[x]_E$ for the $E$-equivalence class of $x\in X$. An action is \emph{free} if $g\cdot x\ne x$ for all $x\in X$ and $1_G\ne g\in G$, where here and below we write $1_G$ for the identity element of $G$. Equivalently $G\acts X$ is free if all stabilizers $\Stab(x)$ are trivial, where
$$\Stab(x):=\{g\in G\:g\cdot x=x\}$$
is the \emph{stabilizer} of $x$ in $G$. If $g\cdot x=y$ then
$$\Stab(y)=g\cdot\Stab(x)\cdot g^{-1},$$
so that orbit equivalent elements have conjugate stabilizers.

Countable Borel equivalence relations have been the focus of intensive study over the past twenty-five years, and significant progress has been made in understanding their structure even though a number of fundamental problems remain unsolved. The relation $\leq_B$ of Borel reducibility defines a partial pre-order on the class of Borel equivalence relations that is often interpreted as a complexity ordering. With respect to this ordering, the simplest countable Borel equivalence relations are the \emph{smooth} ones, i.e., those that Borel reduce to the equality relation on some standard Borel space, or equivalently those admitting a Borel selector. Here a \emph{selector} for the equivalence relation $E$ on $X$ is a function $\sigma:X\to X$ whose graph is contained in $E$ and whose image is a \emph{transversal} for $E$, that is, a subset of $X$ meeting each $E$-class in exactly one point. Every finite Borel equivalence relation is smooth, and in general the structure of smooth countable Borel equivalence relations is rather trivial to understand.

The next simplest countable Borel equivalence relations are the hyperfinite ones. An equivalence relation $E$ is said to be \emph{hyperfinite} if it can be expressed as the union of an increasing sequence of finite Borel equivalence relations. Examples of hyperfinite equivalence relations include the orbit equivalence relation arising from the shift action of $\Z$ on $2^\Z$, the Vitali equivalence relation $E_v$ defined on $\R$ by $x\mathrel{E_v}y\Leftrightarrow x-y\in\mathbb Q$, and the combinatorial counterpart $E_0$ of the Vitali relation defined on Cantor space $2^\N$ by
$$x\mathrel{E_0}y \ \Leftrightarrow \ (\exists n)(\forall k\geq n)\, x(k)=y(k).$$
Every smooth countable Borel equivalence relation is hyperfinite, but each of the examples just given is non-smooth. By a deep result due to Harrington--Kechris--Louveau \cite{HKL} and generalizing earlier work of Glimm--Effros, if $E$ is any (not necessarily countable) non-smooth Borel equivalence relation, then $E_0\sqsubseteq_B E$. Hence $E_0$ is an immediate successor of the trivial smooth equivalence relations in the $\leq_B$ hierarchy.

The hyperfinite Borel equivalence relations were thoroughly investigated and completely classified up to Borel bireducibility and isomorphism by Dougherty, Jackson, and Kechris in \cite{DJK}. Here $E$ and $F$ are \emph{Borel bireducible} if $E\leq_BF$ and $F\leq_BE$, and \emph{isomorphic} if there is a bijective Borel reduction from $E$ to $F$. Dougherty, Jackson, and Kechris showed that any two non-smooth hyperfinite Borel equivalence relations are Borel bireducible (in fact bi-embeddable) with each other, and that two hyperfinite Borel equivalence relations are isomorphic if and only if they admit the same number of invariant ergodic Borel probability measures. They also proved that the class of hyperfinite Borel equivalence relations is closed under Borel reductions, sub-equivalence relations, restrictions to Borel sets, finite products, finite extensions, and countable disjoint unions. For more on hyperfinite equivalence relations see \cite{DJK}.

\subsection{Locally nilpotent groups}

Let $G$ be a countable group. The \emph{center} of $G$ is the subgroup
$$\zeta(G):=\{h\in G\: gh=hg\mbox{ for all $g\in G$}\}.$$
A subgroup $H\leq G$ is \emph{central} if $H\leq\zeta(G)$. $G$ is said to be \emph{nilpotent} if $G$ admits a \emph{central series}, i.e., a sequence
$$\{1_G\} = G_0 \lhd G_1 \lhd \cdots \lhd G_n = G$$
of subgroups, necessarily normal in $G$, such that $G_{i+1}/G_i\leq\zeta(G/G_i)$ for each $0\leq i<n$. The minimal length of a central series of a nilpotent group $G$ is called the \emph{nilpotency class} of $G$.

The \emph{upper central series} of the nilpotent group $G$, written
$$\{1_G\} = \zeta_0G \lhd \zeta_1G \lhd \cdots \lhd \zeta_nG=G,$$
is defined inductively so that $\zeta_{i+1}G$ is the pullback of $\zeta(G/\zeta_iG)$ under the canonical surjection $G\to G/\zeta_iG$. Each term $\zeta_iG$ is a characteristic subgroup of $G$, and $\zeta_1G=\zeta(G)$ is just the center of $G$. If $\{1_G\}=G_0\lhd G_1\lhd\cdots\lhd G_m=G$ is any central series in $G$, then $G_i\leq\zeta_iG$ for each $0\leq i\leq m$, so that in particular $\zeta_mG=G$. Therefore the upper central series of a nilpotent group is a central series of minimal length in $G$; in particular its length, i.e., the least $n$ such that $\zeta_nG=G$, is the nilpotency class of $G$. If $G$ is a nilpotent group of class $n$, then $G/\zeta(G)$ is a nilpotent group of class $n-1$, and therefore it is possible to prove facts about nilpotent groups by induction on nilpotency class.

A series $(G_i)$ in $G$ is \emph{cyclic} if each factor $G_{i+1}/G_i$ is cyclic. Finitely generated nilpotent groups admit central series whose factors are cyclic with prime or infinite orders (\cite[5.2.18]{R}). Moreover, the number of infinite factors in such a series is independent of the series (\cite[5.4.13]{R}). If $G$ is a finitely generated nilpotent group, the number of infinite factors in any cyclic series of $G$ is called the \emph{Hirsch length} of $G$. Hirsch length can also be understood more concretely as follows. Suppose $G$ is a finitely generated nilpotent group. Since subgroups of finitely generated nilpotent groups are finitely generated (\cite[5.2.17]{R}), each term $\zeta_iG$ of the upper central series of $G$ is itself finitely generated. Hence each factor $\zeta_{i+1}G/\zeta_iG$ is a finitely generated abelian group isomorphic to some $\Z^{m_i}\times\Gamma_i$ where $\Gamma_i$ is finite abelian and $m_i$ is the \emph{rank} of $\zeta_{i+1}G/\zeta_iG$. The Hirsch length of $G$ is then just the sum $\sum m_i$ of the ranks of the abelian factors $\zeta_{i+1}G/\zeta_iG$ of the upper central series of $G$.

The group $G$ is \emph{locally nilpotent} if all of its finitely generated subgroups are nilpotent, or equivalently if $G$ is the union of an increasing sequence of nilpotent groups. An easy way to produce a countable locally nilpotent group that is not nilpotent is to take the direct sum of a countably infinite family of nilpotent groups with unbounded nilpotency classes. Such a group will be \emph{hypercentral}, meaning that it admits a transfinite central series which exhausts the group (see \cite[12.1]{R} for details). Perhaps somewhat surprisingly, there exist locally nilpotent groups that are not hypercentral. Indeed, the class of locally nilpotent groups is large and varied. It is closed under subgroups and images and includes in addition to nilpotent groups all the solvable $p$-groups, all hypercentral groups, all groups that satisfy the normalizer condition, and all Fitting, Baer, and Gruenberg groups (see Sections 12.1 and 12.2 of \cite{R} for basic facts about these classes of groups). As an example of how far removed a locally nilpotent group can be from the class of nilpotent groups, we remark that there exists a countable locally nilpotent group $G$ such that $G$ has trivial center, no nontrivial abelian quotients, no subgroups of finite index, and no finitely generated normal subgroups (\cite[12.1.9]{R}). On the other hand, every countable locally nilpotent group is of course amenable.

\section{Marker sets and $G$-clopen relations} \label{SEC MARKER}

Marker sets have been a common ingredient in hyperfiniteness proofs ever since \cite{SS88}. Roughly speaking, they can be used to convert local constructions on a single orbit into global ones that apply in a uniform Borel manner across the entire space. The lemma below is a slight generalization of the ``Basic Clopen Marker Lemma'' in \cite{GJ}. In \cite{GJ} this lemma is presented in the special case $G = \Z^n$ and $Z = X$. The extension presented here is rather straightforward, but as a convenience to the reader we include a proof. We call the set $Y$ constructed in the lemma a \emph{marker set}.

\begin{lem} \label{LEM MARKER}
Let $G \acts X$ be a Borel action of $G$ on the Polish space $X$. Let $1_G \in F \subseteq G$ be finite and symmetric, and let $Z \subseteq X$ be any Borel subset of $X$. Assume that $F \cap \Stab(z) = \{1_G\}$ for every $z \in Z$. Then there exists a Borel set $Y \subseteq Z$ such that
\begin{enumerate}
 \item[\rm (i)] if $y, y' \in Y$ are distinct, then there is no $g \in F$ for which $g \cdot y = y'$; and
 \item[\rm (ii)] for every $z \in Z$ there exists $y \in Y$ and $g \in F$ such that $g \cdot y = z$.
\end{enumerate}
Furthermore, if $X$ is zero-dimensional, $Z$ is clopen, and the action of $G$ is continuous, then $Y$ can be chosen to be clopen as well.
\end{lem}

\begin{proof}
We will assume that $X$ is zero-dimensional, $Z$ is clopen, and the action of $G$ is continuous. The proof in the Borel case follows by ignoring the topology.

Consider a point $z \in Z$. Since $F \cap \Stab(z) = \{1_G\}$, we have that $g \cdot z \neq  z$ for every $1_G \neq g \in F$. Since the action is continuous and $F$ is finite, we can find a clopen neighborhood $U$ of $z$ with $g \cdot U \cap U = \varnothing$ for all $1_G \neq g \in F$. Therefore there is a countable base $\{U_i \: i \geq 0\}$ for the relative topology on $Z$ consisting of clopen subsets of $Z$ with the property that $g \cdot U_i \cap U_i = \varnothing$ for every $1_G \neq g \in F$ and $i \geq 0$. Now, inductively define the Borel sets $Y_i \subseteq Z$ by setting $Y_0 = U_0$ and
$$Y_{i+1} = Y_i \cup \left( U_{i+1} \setminus \bigcup_{g\in F} g \cdot Y_i \right).$$
We will show that $Y = \bigcup_i Y_i$ is a clopen set satisfying conditions (i) and (ii).

For (ii), let $z \in Z$ be arbitrary and let $i$ be least such that $z \in U_i$. Then by the definition of $Y_i$, either $z \in Y_i$ or $i > 0$ and $z = g \cdot y$ for some $g \in F$ and $y \in Y_{i-1}$. In either case the result follows since $1_G \in F$. For (i), we show by induction on $i$ that if $y, y' \in Y_i$ are distinct, then there is no $g \in F$ such that $g \cdot y = y'$. If $i = 0$, this follows immediately from the fact that $g \cdot U_0 \cap U_0 = \varnothing$ for all $1 \ne g \in F$. Now fix $i \geq 0$, and suppose that $y$ and $y'$ are distinct elements of $Y_{i+1}$. We may assume that $y$ and $y'$ do not both belong to $U_{i+1}$, and by the inductive hypothesis we may assume that $y$ and $y'$ do not both belong to $Y_i$. Hence without loss of generality we may assume that
$$y \in Y_i \setminus U_{i+1} \quad \mbox{and} \quad y' \in U_{i+1} \setminus \bigcup_{g \in F} g \cdot Y_i.$$
In particular,
$$y \in Y_i \quad \mbox{and} \quad y' \not\in \bigcup_{g \in F} g \cdot Y_i,$$
so there can be no $g \in F$ such that $y' = g \cdot y$. Since $F$ is symmetric, likewise there is no $g \in F$ such that $y = g \cdot y'$.

Finally, we show that $Y$ is clopen. By continuity of the action, any translate of any clopen set is clopen, so it is easy to verify by induction on $i$ that each $Y_i$ is clopen, which implies that $Y$ is open. Now clauses (i) and (ii) imply that
$$X \setminus Y = \left( \bigcup_{1_G \ne g \in F} g \cdot Y \right) \cup \big( X \setminus Z \big)$$
is open as well, so $Y$ is clopen.
\end{proof}

\begin{cor} \label{COR PART}
Let $G \acts X$ be a Borel action of $G$ on the Polish space $X$, and let $Y \subseteq X$ be a non-empty Borel set. If $1_G \in F \subseteq G$ is finite and symmetric with $F \cap \Stab(y) = \{1_G\}$ for every $y \in Y$, then there is a partition of $Y$ into finitely many disjoint Borel sets
$$Y = Y_1 \sqcup \cdots \sqcup Y_k$$
such that for all $1\leq i\leq k$ and distinct points $y, y' \in Y_i$, there is no $g \in F$ such that $g \cdot y = y'$. Furthermore, if $X$ is zero-dimensional, $Y$ is clopen, and the action of $G$ is continuous, then the $Y_i$'s can be chosen to be clopen as well.
\end{cor}

\begin{proof}
First apply Lemma \ref{LEM MARKER} with $Z = Y$ to obtain a Borel (clopen) set $Y_1 \subseteq Y$ with the property that $F \cdot Y_1 \supseteq Y$ and $g \cdot Y_1 \cap Y_1 = \varnothing$ for every $1_G \neq g \in F$. Then inductively, given disjoint nonempty Borel (clopen) subsets $Y_1, \ldots, Y_i$ of $Y$ such that $Y \setminus (Y_1 \sqcup \cdots \sqcup Y_i) \ne \varnothing$, apply Lemma \ref{LEM MARKER} with
$$Z = Y \setminus \big( Y_1 \sqcup \cdots \sqcup Y_i \big)$$
to obtain a Borel (clopen) set $Y_{i+1} \subseteq Y \setminus (Y_1 \sqcup \cdots \sqcup Y_i)$ with the property that
$$F \cdot Y_{i+1} \supseteq Y \setminus \big( Y_1 \sqcup \cdots \sqcup Y_i \big)$$
and $g \cdot Y_{i+1} \cap Y_{i+1} = \varnothing$ for every $1_G \neq g \in F$. We claim that there is $k \leq t := |F|+1$ such that $Y = Y_1 \sqcup \cdots \sqcup Y_k$. If not, then $Y \setminus (Y_1 \sqcup \cdots \sqcup Y_t)$ is nonempty, say it contains $y$. Then we have $y \in F \cdot Y_i$ for all $1 \leq i \leq t$. As $F$ is symmetric, this implies $Y_i \cap F \cdot y \neq \varnothing$ for all $1 \leq i \leq t$, contradicting the fact that the $Y_i$ are pairwise disjoint.
\end{proof}

The following notion, introduced in \cite{GJ}, will be useful in constructing continuous reductions to $E_0$.

\begin{defn} \label{DEFN GCLOPEN}
Let $G$ act continuously on the Polish space $X$. A relation $R \subseteq X \times X$ is \emph{$G$-clopen} if for every $g \in G$ the set $\{x \in X \: (x, g \cdot x) \in R\}$ is clopen.
\end{defn}

In Remark 3.3 of \cite{GJ}, Gao and Jackson called an equivalence relation satisfying Definition \ref{DEFN GCLOPEN} \emph{clopen} rather than $G$-clopen, and used this terminology throughout \cite{GJ}. This is in conflict with the standard notion of clopen (viewing $R$ as a subset of $X \times X$), and we adopt the new terminology in order to avoid ambiguity. In practice we will only use the notion for two types of relations: equivalence relations on $X$ and (graphs of) functions from $X$ to $X$. The first lemma below addresses the connection between clopenness and $G$-clopenness.

\begin{lem}
Let $G$ act continuously on the Polish space $X$ and let $E$ be an equivalence relation on $X$. If $E$ is contained in $E_G^X$ and is $G$-clopen, then $E$ is $F_\sigma$ (in particular, $E$ is Borel). On the other hand, if $E$ is clopen then $E$ is $G$-clopen.
\end{lem}

\begin{proof}
First suppose that $E \subseteq E_G^X$ and that $E$ is $G$-clopen. For $g \in G$ define
$$E_g = \{x \in X \: x \ E \ g \cdot x\}.$$
The assumption $E \subseteq E_G^X$ gives
$$(x, y) \in E \quad \Longleftrightarrow \quad (\exists g \in G) \ y = g \cdot x \ \wedge \ x \in E_g.$$
Since $E$ is $G$-clopen and the action is continuous, for each fixed $g \in G$ the condition on the right is closed in $X \times X$. Therefore $E$ is an $F_\sigma$ subset of $X \times X$ as claimed.

Now suppose that $E$ is a clopen subset of $X \times X$, fix $g\in G$, and define $\pi_g : X \rightarrow X \times X$ by $\pi_g(x) = (x, g \cdot x)$. Then $E_g$ is the inverse image of the clopen set $E$ under the continuous map $\pi_g$, so $E_g$ is clopen.
\end{proof}

\begin{lem} \label{LEM GHCLOPEN}
Let $G$ act freely and continuously on the Polish space $X$, let $E$ be an equivalence relation on $X$, and let $H \leq G$.
\begin{enumerate}
\item[\rm (i)] If $E$ is $G$-clopen then $E$ is also $H$-clopen.
\item[\rm (ii)] If $E$ is $H$-clopen and $E \subseteq E_H^X$ then $E$ is $G$-clopen.
\end{enumerate}
\end{lem}

\begin{proof}
Clause (i) is immediate from the definitions and does not require freeness of the action. For clause (ii), given $g \in G$ the set $\{x \in X \: x \ E \ g \cdot x\}$ is clopen if $g \in H$ and is empty if $g \not\in H$ since the action is free and $E \subseteq E_H^X$. So in either case the set is clopen.
\end{proof}

\begin{lem} \label{LEM SELECT}
Let $G$ act freely and continuously on the zero-dimensional Polish space $X$. Let $E$ be a finite $G$-clopen equivalence relation on $X$ and suppose there is a finite set $K \subseteq G$ such that $[x]_E \subseteq K \cdot x$ for all $x \in X$. Then there is a continuous and $G$-clopen selector $S$ for $E$, i.e.\ a continuous and $G$-clopen function $S : X \rightarrow X$ such that for all $x, y \in X$, we have $x \ E \ S(x)$ and
$$x \ E \ y \quad \Longleftrightarrow \quad S(x) = S(y).$$
\end{lem}

\begin{proof}
Fix a countable base of clopen sets $\{U_n \: n \in \N\}$ for the topology on $X$. For $x \in X$ let $S(x)$ be such that
$$\{S(x)\} = U_{n(x)} \cap [x]_E$$
where $n(x)$ is least with $|U_{n(x)} \cap [x]_E| = 1$. Such an $n(x)$ exists since $E$ is finite and $X$ is Hausdorff. Clearly $S$ is a selector for $E$, so it only remains to check that $S$ is continuous and $G$-clopen. We accomplish this by showing that for each $k\in K$ the set of $x\in X$ for which $S(x)=k\cdot x$ is clopen. This will suffice since the action is continuous and $\{x\in X\: S(x)=g\cdot x\}=\varnothing$ for $g\in G\setminus K$.

For $g \in G$ let $E_g$ be the clopen set $\{x \in X \: x \ E \ g \cdot x\}$, and for $x \in X$ let
$$C(x) = \{k \in K \: x \ E \ k \cdot x\}.$$
Then for any subset $K_0 \subseteq K$, the set
$$\{x \in X \: C(x) = K_0\} \ = \ \left( \bigcap_{k\in K_0} E_k \right) \setminus \left( \bigcup_{k\in K\setminus K_0} E_k \right)$$
is clopen. Hence for each $K_0 \subseteq K$ and $n \in \N$ the set
$$\{x\in X \: C(x)=K_0 \; \wedge \; |U_n\cap [x]_E|=1\}$$
is clopen, as it the intersection of $\{x \in X \: C(x) = K_0\}$ with the clopen set
$$\bigcup_{k\in K_0}\left(k^{-1}\cdot U_n\setminus\bigcup_{k\ne h\in K_0}h^{-1}\cdot U_n\right)_.$$
(Here we use the freeness of the action). From this it follows that for each $K_0\subseteq K$ and $m\in\N$ the set
$$X_{K_0,m} \ = \ \{x\in X\:C(x)=K_0\;\wedge\; n(x)=m\}$$
is clopen, and therefore for each $m\in\N$,
$$X_m \ = \ \{x\in X\: n(x)=m\} \ = \ \bigcup_{K_0\subseteq K}X_{K_0,m}$$
is clopen. Finally, this implies that for each $k\in K$ the set
$$X(k) \ = \ \{x\in X\:S(x)=k\cdot x\} \ = \ E_k\cap\left[\bigcup_{m\in\N}\left(X_m\cap k^{-1}\cdot U_m\right)\right]$$
is open. Now $X=\bigsqcup_{k\in K}X(k)$ is a finite partition of $X$ into open sets, so each $X(k)$ must in fact be clopen.
\end{proof}

Our final lemma in this section was demonstrated in \cite{GJ}, but we include a proof as a convenience to the reader. Our proof will make use of the well-known fact that given a countable group $G$ and a Borel action of $G$ on the standard Borel space $X$, there is a $G$-equivariant embedding of $X$ into $(2^\N)^G$. Specifically, let $G$ act on $(2^\N)^G$ by permuting coordinates on the left, so that $(g\cdot y)(i,h)=y(i,g^{-1}h)$ for $y\in (2^\N)^G$, $i\in\N$, and $g, h\in G$. Fix a sequence $(U_i)$ of Borel sets in $X$ that separates points and define $\phi:X\to (2^\N)^G$ by
$$\phi(x)(i,h)=1 \quad \Longleftrightarrow \quad x\in h\cdot U_i.$$
Then $\phi$ is Borel, injective, and $G$-equivariant, meaning that $\phi(g\cdot x)=g\cdot\phi(x)$ for all $x\in X$ and $g\in G$. If $X$ is a zero-dimensional Polish space, $G$ acts continuously, and we take the $U_i$ to be a clopen base, then additionally $\phi$ is continuous.

\begin{lem}[Gao--Jackson, \cite{GJ}] \label{LEM E0}
Let $G$ act freely and continuously on the zero-dimensional Polish space $X$. Let $F_1, F_2, \ldots$ be a sequence of finite $G$-clopen equivalence relations on $X$. Suppose there are finite sets $K_n \subseteq G$ such that $[x]_{F_n} \subseteq K_n \cdot x$ for all $n \geq 1$ and all $x \in X$. Then there is a continuous embedding of $F$ into $E_0$, where $F$ is defined by
$$x \ F \ y \quad \Longleftrightarrow \quad (\exists m \in \N) \ (\forall n \geq m) \ x \ F_n \ y.$$
\end{lem}

\begin{proof}
In this proof we write $2^n$ for the set of all binary sequences of length $n$. Let $F_0$ be the equality relation on $X$ and set $K_0 = \{1_G\}$. By enlarging the $K_n$'s we may suppose that the sets $K_n$, $n \in \N$, are increasing, symmetric, and exhaust $G$. Since $X$ is zero-dimensional and $G$ acts continuously, there is a $G$-equivariant continuous embedding of $X$ into $(2^\N)^G$, as described above. So without loss of generality, we may suppose that $X \subseteq (2^\N)^G$.

For each $n \in \N$, let $S_n: X \rightarrow X$ be a continuous and $G$-clopen selector for $F_n$, as given by Lemma \ref{LEM SELECT}. For $x \in X$, let $g_0(x)=1_G$ and for $n\geq 1$ define $g_n(x)$ to be the unique element of $K_nK_{n-1}\subseteq K_nK_n$ such that $S_n(x)=g_n(x)\cdot S_{n-1}(x)$. Notice that each function $g_n$ is continuous since the $S_n$'s are $G$-clopen. Next, for $x \in X\subseteq (2^\N)^G=2^{\N\times G}$ and $n \in \N$ let $\sigma_n(x)$ denote the restriction of $S_n(x)$ to the domain $\{0, 1, \ldots, n\} \times (K_n K_n)$. Note that $\sigma_n(x) \in (2^{n+1})^{K_n K_n}$ and for each $n$ the map $x\mapsto \sigma_n(x)$ is continuous. Also fix for each $n$ an integer $m(n)$ and an injection $\theta_n: (2^{n+1})^{K_n K_n} \times K_n K_n \rightarrow 2^{m(n)}$.

Now define $f: X \rightarrow 2^\N$ by
$$f(x) = \theta_0(\sigma_0(x), g_0(x)) ^\frown \theta_1(\sigma_1(x), g_1(x)) ^\frown \cdots ^\frown \theta_n(\sigma_n(x), g_n(x)) ^\frown \cdots,$$
where $\pi ^\frown \tau$ denotes the concatenation of $\pi, \tau \in 2^{< \N}$ (here $2^{< \N}$ denotes the set of all finite sequences of $0$'s and $1$'s). Then $f$ is continuous since the $\sigma_n$'s and $g_n$'s are continuous. If $F$ is defined as in the statement of the lemma then $x \ F \ y$ implies $S_n(x) = S_n(y)$ for all sufficiently large $n$. This implies that $\theta_n(\sigma_n(x), g_n(x)) = \theta_n(\sigma_n(y), g_n(y))$ for all sufficiently large $n$, so $x \ F \ y$ implies $f(x) \ E_0 \ f(y)$.

Now we check that $f(x) \ E_0 \ f(y)$ implies $x \ F \ y$ and that $f$ is injective. Suppose that $f(x) \ E_0 \ f(y)$, and choose any value of $n$ such that $\theta_m(\sigma_m(x), g_m(x)) = \theta_m(\sigma_m(y), g_m(y))$ for all $m \geq n$. Then $g_m(x) = g_m(y)$ and $S_m(x)$, $S_m(y)$ agree on the set $(2^{m+1})^{K_m K_m}$ for all $m \geq n$. We will show that this implies $S_n(x) = S_n(y)$. Indeed, let $i \in \N$ and $h \in G$, and fix $m \geq \max(n, i)$ with $K_n h \subseteq K_m$. We have $S_m(x) \in K_m \cdot x$ and $S_n(x) \in K_n \cdot x$ and thus $S_n(x) \in K_n K_m \cdot S_m(x)$. We also have
$$g_m(x) g_{m-1}(x) \cdots g_{n+1}(x) \cdot S_n(x) = S_m(x),$$
so by freeness of the action $(g_m(x) g_{m-1}(x) \cdots g_{n+1}(x))^{-1} \in K_n K_m$. Since $g_t(x) = g_t(y)$ for $t \geq n$ we find that there is a single $k \in K_n K_m$ with both $S_n(x) = k \cdot S_m(x)$ and $S_n(y) = k \cdot S_m(y)$. Now since $K_n h \subseteq K_m$ we have $k^{-1} h \in K_m K_m$ and thus
\begin{align*}
S_n(x)(i, h) = [k \cdot S_m(x)](i, h) & = S_m(x)(i, k^{-1} h)\\
& = S_m(y)(i, k^{-1} h) = [k \cdot S_m(y)](i, h) = S_n(y)(i, h).
\end{align*}
Since $i \in \N$ and $h \in G$ were arbitrary, it follows that $S_n(x) = S_n(y)$. This holds for all sufficiently large $n$, so we conclude $x \ F \ y$. Furthermore, if $f(x) = f(y)$ then we can use $n = 0$ to obtain $x = S_0(x) = S_0(y) = y$, so $f$ is injective.
\end{proof}

\section{Geometry of abelian and nilpotent groups} \label{SEC GEOM}

The arguments used by Gao--Jackson \cite{GJ} relied heavily upon the nice geometry of the groups $\Z^n$ and in particular upon geometric notions such as $n$-dimensional rectangles and their faces. Our arguments will also rely heavily upon geometric notions, though ours will be coarse geometric notions inspired by the geometry of abelian groups. In this section we discuss the relevant geometry of abelian groups and then present a definition, namely that of a \emph{chart}, which will allow us to extend a coarse approximation of this geometry to nilpotent groups.

Throughout $\Gamma$ will always denote a finite additive abelian group. We will consider additive groups of the form $\Z^\ell \times \Gamma$. We view points $\T{v} \in \R^\ell \times \Gamma$ as vectors with $\ell + 1$ coordinates, where the first $\ell$ coordinates range over real numbers and the $(\ell + 1)^\text{st}$ coordinate ranges over elements of $\Gamma$. We will always denote the coordinates of $\T{v} \in \R^\ell \times \Gamma$ by $v_i$, $1 \leq i \leq \ell + 1$. Given $\ell$ and $\Gamma$, we let $\T{0}$ and $\T{1}$ denote the vectors whose first $\ell$ coordinates are $0$ and $1$, respectively, and whose $(\ell + 1)^\text{st}$ coordinate is $1_\Gamma$. We let $\T{e}_i$, $1\leq i\leq\ell$, denote the vector with value $1$ in the $i^\text{th}$-coordinate, value $1_\Gamma$ in the $(\ell + 1)^\text{st}$-coordinate, and value $0$ in all other coordinates. (The dependence of $\T{e}_i$, $\T{0}$, and $\T{1}$ on $\ell$ and $\Gamma$ will never cause confusion). If $\T{v} = (v_1, v_2, \ldots, v_\ell, v_{\ell + 1}) \in \R^\ell \times \Gamma$ and $\lambda \in \R$ then we set
$$\lambda \cdot \T{v} = (\lambda v_1, \lambda v_2, \ldots, \lambda v_\ell, v_{\ell + 1}) \in \R^\ell \times \Gamma.$$

For a vector $\T{a} \in \R^\ell \times \Gamma$ we define
$$\Rect(\T{a}) = \{\T{b} \in \Z^\ell \times \Gamma \: -|a_i| \leq b_i \leq |a_i| \text{ for all } 1 \leq i \leq \ell\}.$$
Note that there is no restriction on the $(\ell + 1)^\text{st}$-coordinate of $\T{b} \in \Rect(\T{a})$. A \emph{rectangle} in $\Z^\ell \times \Gamma$ is any set of the form $\T{c} + \Rect(\T{a})$ for $\T{c},\T{a} \in \Z^\ell \times \Gamma$. Observe that even if one ignores $\Gamma$ this is still not quite the standard notion of rectangle since the integral portions of our rectangles must have genuine centers in $\Z^\ell$. Note that if $A = \T{c} + \Rect(\T{a})$ with $\T{c}\in\Z^\ell\times\{1_\Gamma\}$ and $\T{a} \in \N^\ell \times \{1_\Gamma\}$, then $\T{c}$ and $\T{a}$ are uniquely determined from $A$; in this case we call $\T{c}$ the \emph{center} of $A$ and $\T{a}$ the \emph{radius vector} of $A$. We write $\bdim(A)$ for the radius vector $\T{a}$ of $A$ and $\bdim_i(A)$ for its integer components $a_i$ ($1 \leq i \leq \ell$). A rectangle $A$ is \emph{centered} if it can be written in the form $A = \T{c} + \Rect(\T{a})$ with $\T{c} = \T{0}$. A \emph{translate} of $A$ is a rectangle of the form $\T{t} + A$ where $\T{t} \in \Z^\ell \times \Gamma$. Note that $\bdim(A) = \bdim(\T{t} + A)$. We write $A \sqsubseteq B$ to mean that a translate of the rectangle $A$ is contained in the rectangle $B$, or equivalently $\bdim_i(A) \leq \bdim_i(B)$ for each $1 \leq i \leq \ell$. If the rectangle $A$ is centered at $\T{c}$ and $\lambda \in \R_+$ then we let $\lambda \cdot A = \T{c} + \Rect(\lambda \cdot \bdim(A))$. Note that $\bdim_i(\lambda \cdot A) = \lfloor \lambda \cdot \bdim_i(A) \rfloor$ for $1 \leq i \leq \ell$ (where $\lfloor x \rfloor$ denotes the greatest integer less than or equal to the real number $x$). We write $-A$ for the rectangle $\{-\T{a}\:\T{a}\in A\}$, and $-\lambda\cdot A$ for $- (\lambda \cdot A)$ when $\lambda>0$. Observe that if $A$ is centered at $\T{c}$ and $\lambda > 0$ then $- \lambda \cdot A$ is centered at $-\T{c}$.

For a rectangle $A \subseteq \Z^\ell\times\Gamma$ centered at $\T{c}$ and $1 \leq i \leq \ell$ we set
$$A^i = \T{c} + \Rect \big( \bdim(A) - \bdim_i(A) \cdot \T{e}_i \big).$$
In other words, $A^i$ is obtained from $A$ by flattening $A$ in the $i^\text{th}$ coordinate direction, or more precisely $A^i$ is the rectangle centered at the same point as $A$ and with $\bdim_j(A^i) = \bdim_j(A)$ for $j \neq i$ and $\bdim_i(A^i) = 0$. Notice that $-\bdim_i(A) \cdot \T{e}_i + A^i$ and $\bdim_i(A) \cdot \T{e}_i + A^i$ are the two faces of $A$ which are perpendicular to $\T{e}_i$.

Our notation for rectangles is convenient in that it allows for streamlined proofs of our results. However we should point out that the notation has some shortcomings, and occasionally suggests statements which are not true. For example, the (element-wise) sum $A + A$ is indeed a rectangle, but is not equal to $2 \cdot A$, as the center of $A + A$ is twice the center of $A$ while the center of $2 \cdot A$ is the same as the center of $A$. Similarly, we have that $\lambda\cdot A-\eta\cdot A\subseteq (\lambda+\eta)\cdot C$ where $C$ is the centered translate of $A$ and $\lambda,\eta\in\R_+$. Two further examples are that $\lambda\cdot A+\eta\cdot A$ need not equal $(\lambda+\eta)\cdot A$ even when $A$ is centered, and $\lambda \cdot (\eta \cdot A) \neq (\lambda \eta) \cdot A$ in general, although see clauses (i) and (vii) of Lemma \ref{LEM BASIC} below. These phenomena are very minor nuisances and are dealt with by Lemma \ref{LEM BASIC}. We will write $\lambda \eta \cdot A$ for $(\lambda \eta) \cdot A$ and will make clear use of parentheses in the few rare cases where a different order of operations is desired.

\begin{lem} \label{LEM BASIC}
Let $A, B \subseteq \Z^\ell \times \Gamma$ be rectangles, and let $\epsilon$, $\delta$, $\lambda$, and $\eta$ be positive real numbers with $\delta \cdot B \sqsupseteq \Rect(\T{1})$.
\begin{enumerate}
\item[\rm (i)] $\lambda \cdot (\eta \cdot A) \subseteq \lambda \eta \cdot A$.
\item[\rm (ii)] If $\lambda\leq\eta$, then $\lambda\cdot A\subseteq\eta\cdot A$.
\item[\rm (iii)] If $A\sqsubseteq B$, then $\lambda\cdot A\sqsubseteq\lambda\cdot B$.
\item[\rm (iv)] If $2 \delta \cdot B \sqsubseteq \epsilon \cdot A$, then $\lambda \delta \cdot B \sqsubseteq \lambda \epsilon \cdot A$.
\item[\rm (v)] If $A$ is centered and $A \sqsubseteq \epsilon \cdot B$, then $\lambda \cdot B + A \subseteq (\lambda + \epsilon) \cdot B$.
\item[\rm (vi)] If $A$ is centered and $2 \delta \cdot B \sqsubseteq A$, then $(\lambda + \delta) \cdot B \subseteq \lambda \cdot B + A$.
\item[\rm (vii)] If $A$ is centered, then $\lambda\cdot A+\eta\cdot A\subseteq (\lambda+\eta)\cdot A$.
\end{enumerate}
\end{lem}

We remark that clauses (iv) and (vi) do not give the optimal estimates, but will suffice for our purposes.

\begin{proof}
Recall that $A \sqsubseteq B$ if and only if $\bdim_i(A) \leq \bdim_i(B)$ for all $1 \leq i \leq \ell$. Set $\T{a} = \bdim(A)$ and $\T{b} = \bdim(B)$.

(i) -- (iii) are immediate from the definitions.

(iv). Using $\Rect(\T{1}) \sqsubseteq\delta\cdot B$, we have $\delta b_i < \lfloor \delta b_i + 1 \rfloor \leq \lfloor 2 \delta b_i \rfloor \leq \lfloor \epsilon a_i \rfloor \leq \epsilon a_i$. Therefore $\lambda \delta b_i < \lambda \epsilon a_i$ and hence $\lfloor \lambda \delta b_i \rfloor \leq \lfloor \lambda \epsilon a_i \rfloor$.

(v). We have $\lfloor \lambda b_i \rfloor + a_i \leq \lfloor \lambda b_i \rfloor + \lfloor \epsilon b_i \rfloor \leq \lfloor (\lambda + \epsilon) b_i \rfloor$. This shows that $\lambda \cdot B + A \sqsubseteq (\lambda + \epsilon) \cdot B$. Since $A$ is centered the former must be a subset of the latter.

(vi). Using $\Rect(\T{1}) \sqsubseteq \delta \cdot B$, we have $\lfloor \lambda b_i \rfloor + a_i \geq \lfloor \lambda b_i \rfloor + \lfloor 2 \delta b_i \rfloor \geq \lfloor \lambda b_i \rfloor + \lfloor \delta b_i + 1 \rfloor \geq \lfloor (\lambda + \delta) b_i \rfloor$. So as in the proof of (v) we obtain $(\lambda + \delta) \cdot B \subseteq \lambda \cdot B + A$ since $A$ is centered.

(vii) This follows from clause (v).
\end{proof}

Clauses (i) -- (iii), (v), and (vii) of Lemma \ref{LEM BASIC} will be used frequently throughout the paper. We will refrain from explicitly citing these clauses since they are quite intuitive (the reader will likely not even notice that they are needed) and, given how frequently we use them, it would be overly repetitive to cite them. Clause (iv) will also be used frequently without mention, but we will discuss this more at a later time (specifically after Definitions \ref{DEFN ROUGH} and \ref{DEFN RECTEQ}). We will explicitly mention any use of clause (vi). Before continuing we now highlight two additional facts that are somewhat technical but will play an important role in Section \ref{SEC FACES}.

\begin{lem} \label{LEM BASIC2}
Let $A, B \subseteq \Z^\ell \times \Gamma$ be rectangles, and let $\epsilon$, $\delta$, and $\eta$ be non-negative real numbers with $\delta<1$ and $\Rect(\T{1}) \sqsubseteq \delta \cdot B$.
\begin{enumerate}
\item[\rm (i)] If $A$ meets $B$ and $2\delta\cdot B\sqsubseteq\epsilon\cdot A$, then $(1+\epsilon)\cdot A$ meets $(1-\delta)\cdot B$.
\item[\rm (ii)] Suppose $A$ is centered, let $\T{w}\in(1+\delta)\cdot B$, and suppose that $\eta\cdot A\sqsubseteq (1-\delta)\cdot B$ and $4\delta\cdot B\sqsubseteq\epsilon\cdot A$. Then there exits $\T{s}\in (\eta+\epsilon)\cdot A$ such that $\T{w}+\T{s}+\eta\cdot A\subseteq (1-\delta)\cdot B$.
\end{enumerate}
\end{lem}

\begin{proof}
(i) Let $\T{c}$ be the center of $A$ and $\T{d}$ the center of $B$, and for each $1\leq i\leq\ell$ let $\rho_i=|c_i-d_i|$. Then $A$ meets $B$ if and only if $\bdim_i(A)+\bdim_i(B)\geq\rho_i$ for all $i$, and $(1+\epsilon)\cdot A$ meets $(1-\delta)\cdot B$ if and only if $\lfloor (1+\epsilon)\,\bdim_i(A)\rfloor+\lfloor(1-\delta)\,\bdim_i(B)\rfloor\geq\rho_i$ for all $i$.
Therefore, writing $a_i=\bdim_i(A)$ and $b_i=\bdim_i(B)$, we have that $(1+\epsilon)\cdot A$ meets $(1-\delta)\cdot B$ provided
$$a_i+b_i \ \leq \ \lfloor (1+\epsilon)a_i\rfloor +\lfloor (1-\delta)b_i\rfloor \ = \ a_i+b_i+\lfloor\epsilon a_i\rfloor +\lfloor -\delta b_i\rfloor$$
for all $i$, or in other words provided $\lfloor\epsilon a_i\rfloor\geq\lceil\delta b_i\rceil$ for all $i$. Since $\Rect(\T{1}) \sqsubseteq \delta \cdot B$ and $2\delta\cdot B\sqsubseteq\epsilon\cdot A$, for each $i$ we have $\lceil\delta b_i\rceil\leq\lfloor\delta b_i\rfloor+1\leq\lfloor 2\delta b_i\rfloor\leq\lfloor\epsilon a_i\rfloor$, as needed.

(ii) We have that $\epsilon\cdot A$ is centered and $\epsilon\cdot A\sqsupseteq 4\delta\cdot B$, so by clause (vi) of Lemma \ref{LEM BASIC}, $(1+\delta)\cdot B \subseteq (1-\delta)\cdot B+\epsilon\cdot A$. Thus we can find $\T{u}\in\epsilon\cdot A$ such that $\T{w}+\T{u}\in (1-\delta)\cdot B$. Now since $\eta\cdot A\sqsubseteq (1-\delta)\cdot B$ we can find a translate of $\eta \cdot A$ that contains $\T{w} + \T{u}$ and is contained in $(1 - \delta) \cdot B$. This implies that there is $\T{v}\in\eta\cdot A$ such that $\T{w}+\T{u}+\T{v}+\eta\cdot A\subseteq (1-\delta)\cdot B$. Now let $\T{s}=\T{u}+\T{v}\in\epsilon\cdot A+\eta\cdot A\subseteq (\epsilon+\eta)\cdot A$.
\end{proof}

Now having discussed the relevant geometry of abelian groups, we are ready to present a definition which will allow us to extend a coarse approximation of this geometry to nilpotent groups. The name for this notion draws inspiration from the theory of manifolds.

\begin{defn} \label{DEFN CHART}
Let $G$ be a countable group. A \emph{chart} for $G$ is a $5$-tuple $\Phi = (\ell, \phi, \rZ, \Gamma, \cH)$ where $\cH$ is a finite collection of pairwise conjugate subgroups of $G$, $\ell \in \N$, $\Gamma$ is a finite abelian group, $\rZ \subseteq \Z^\ell\times\Gamma$ is a centered rectangle with $\bdim_i(\rZ) > 0$ for each $1 \leq i \leq \ell$, $\phi$ is an injective function into $G$ with $\dom(\phi)$ a centered rectangle in $\Z^\ell \times \Gamma$ containing $3 \cdot \rZ$, $\phi(\T{0}) = 1_G$, and with the property that for every $\T{r}, \T{s} \in \dom(\phi)$ and every $H \in \cH$,
\begin{align*}
\phi(\T{r}) \cdot H = \phi(\T{s}) \cdot H & \ \Longrightarrow \ \T{r} = \T{s}; \\
\T{r} + \T{s} + \rZ \subseteq \dom(\phi) & \ \Longrightarrow \ \exists \T{z} \in \rZ \quad \phi(\T{r}) \cdot \phi(\T{s}) \cdot H = \phi(\T{r} + \T{s} + \T{z}) \cdot H; \\
\T{r} - \T{s} + \rZ \subseteq \dom(\phi) & \ \Longrightarrow \ \exists \T{z} \in \rZ \quad \phi(\T{r}) \cdot \phi(\T{s})^{-1} \cdot H = \phi(\T{r} - \T{s} + \T{z}) \cdot H; \\
-\T{r} + \T{s} + \rZ \subseteq \dom(\phi) & \ \Longrightarrow \ \exists \T{z} \in \rZ \quad \phi(\T{r})^{-1} \cdot \phi(\T{s}) \cdot H = \phi(-\T{r} + \T{s} + \T{z}) \cdot H; \\
-\T{s} + \rZ \subseteq \dom(\phi) & \ \Longrightarrow \ \exists \T{z} \in \rZ \quad \phi(\T{s})^{-1} \cdot H = \phi(-\T{s} + \T{z}) \cdot H.
\end{align*}
\end{defn}

We remark that the requirements that $\bdim_i(\rZ)$ be positive and that $\dom(\phi)$ be a centered rectangle are not essential; these assumptions simply allow for an easier presentation of our arguments. In the definition above, one should think of $\rZ$ as being a very small error term and $\dom(\phi)$ as being very large. With this mindset the last four implications say that $\phi$ is an ``almost-homomorphism.'' So in the presence of a chart one can pretend locally (since the image of $\phi$ is finite) that the action of $G$ on $G / H$, for $H \in \cH$, is an action of an abelian group; one just needs to pay a small price (the error term $\rZ$) whenever one uses abelian group operations. This definition will allow us to extend a coarse approximation of the abelian geometry used by Gao and Jackson \cite{GJ} to groups admitting charts. In particular, it will allow us to discuss rectangles and facial boundaries in spaces on which such groups act, and we do precisely this in Sections \ref{SEC RECT} and \ref{SEC FACES}, respectively.

Next we show that non-trivial charts exist for finitely generated nilpotent groups. As a warm-up we will first construct charts where $\cH$ consists simply of the trivial subgroup $\{1_G\}$. For notational convenience we write $\cH = \bOne$ in this case. We remark that charts in which $\cH = \bOne$ are the ones appropriate for handling free actions of $G$, and that we are forced to include $\cH$ in the definition of chart solely in order to handle non-free actions.

\begin{lem} \label{LEM EXIST}
Let $G$ be a finitely generated nilpotent group and let $\ell$ be the Hirsch length of $G$. Then for any finite $F \subseteq G$ and any $\lambda > 0$ there is a chart $\Phi = (\ell, \phi, \rZ, \Gamma, \bOne)$ for $G$ with $F \subseteq \phi(\rZ)$ and $\lambda \cdot \rZ \subseteq \dom(\phi)$.
\end{lem}

\begin{proof}
Without loss of generality, we may assume that $\lambda \geq 3$. We prove the claim by induction on the nilpotency class of $G$. When the nilpotency class of $G$ is $1$ the group is abelian and the lemma clearly holds. Now suppose that the claim holds for all finitely generated nilpotent groups of nilpotency class $n-1$, and let $G$ be a finitely generated nilpotent group of nilpotency class $n$. Then $G / \zeta(G)$ is a finitely generated nilpotent group of nilpotency class $n-1$. So by induction there is a chart $\Phi_0 = (\ell_0, \phi_0, \rZ_0, \Gamma_0, \bOne)$ for $G / \zeta(G)$, where $\ell_0$ is the Hirsch length of $G / \zeta(G)$, satisfying $F\cdot \zeta(G) \subseteq \phi_0(\rZ_0)$ and $\lambda \cdot \rZ_0 \subseteq \dom(\phi_0) \subseteq \Z^{\ell_0} \times \Gamma_0$. Now let $\phi_0': \dom(\phi_0) \rightarrow G$ be any function satisfying $\phi_0'(\T{0}) = 1_G$ and $\phi_0'(\T{v}) \zeta(G) = \phi_0(\T{v})$ for all $\T{v} \in \dom(\phi_0)$. Since $\phi_0'$ has finite domain and is a chart modulo $\zeta(G)$, we can find a finite set $K \subseteq \zeta(G)$ such that $F \subseteq \phi_0'(\rZ_0) K$ and for every $\T{r}, \T{s} \in \dom(\phi_0')$,
\begin{align*}
\T{r} + \T{s} + \rZ_0 \subseteq \dom(\phi_0') & \ \Longrightarrow \ \phi_0'(\T{r}) \cdot \phi_0'(\T{s}) \in \phi_0'(\T{r} + \T{s} + \rZ_0) K; \\
\T{r} - \T{s} + \rZ_0 \subseteq \dom(\phi_0') & \ \Longrightarrow \ \phi_0'(\T{r}) \cdot \phi_0'(\T{s})^{-1} \in \phi_0'(\T{r} - \T{s} + \rZ_0) K; \\
-\T{r} + \T{s} + \rZ_0 \subseteq \dom(\phi_0') & \ \Longrightarrow \ \phi'_0(\T{r})^{-1} \cdot \phi_0'(\T{s}) \in \phi_0'(-\T{r} + \T{s} + \rZ_0) K; \\
-\T{s} + \rZ_0 \subseteq \dom(\phi_0') & \ \Longrightarrow \ \phi_0'(\T{s})^{-1} \in \phi_0'(-\T{s} + \rZ_0) K.
\end{align*}
As subgroups of finitely generated nilpotent groups are finitely generated, $\zeta(G)$ is a finitely generated abelian group and hence is isomorphic to some $\Z^{\ell_1} \times \Gamma_1$. Fix an isomorphism $\phi_1: \Z^{\ell_1} \times \Gamma_1 \rightarrow \zeta(G)$. Let $\rZ_1 \subseteq \Z^{\ell_1} \times \Gamma_1$ be any centered rectangle with $\phi_1(\rZ_1) \supseteq K$ and each $\bdim_i(\rZ_1)>0$. Set $\ell = \ell_0 + \ell_1$. Then $\ell$ is the Hirsch length of $G$. Set $\rZ = \rZ_0 \times \rZ_1$, set $\Gamma = \Gamma_0 \times \Gamma_1$, and define $\phi: \dom(\phi_0') \times \lambda \cdot \rZ_1 \rightarrow G$ by
$$\phi(\T{u}, \T{v}) = \phi_0'(\T{u}) \cdot \phi_1(\T{v})$$
for $\T{u} \in \dom(\phi_0')$ and $\T{v} \in \lambda \cdot \rZ_1$. Clearly $\phi(\T{0}) = 1_G$, $F \subseteq \phi_0'(\rZ_0) K \subseteq \phi(\rZ)$, and $\dom(\phi)$ is a centered rectangle in $\Z^\ell \times \Gamma$ containing $\lambda \cdot \rZ$. The injectivity of $\phi$ is readily verified since $\phi_0'$ is injective modulo $\zeta(G)$ and $\phi_1$ is an injective map into $\zeta(G)$. Most importantly, $\phi_1$ is the restriction of a legitimate homomorphism, and since its image is central, it does not interfere with the ``almost-homomorphism'' behavior of $\phi_0'$ described by the four implications above. Therefore $\Phi = (\ell, \phi, \rZ, \Gamma, \bOne)$ is a chart for $G$.
\end{proof}

The construction of charts with $\cH \neq \bOne$ is done in a similar fashion. Admittedly, working with subgroups of $G$ is a bit cumbersome notationally. As motivation, we point out that charts will ``absorb'' stabilizers appearing in non-free actions and therefore, through the eyes of an appropriate family of charts, every action is in some sense free. Thus charts will allow us to tackle both free and non-free actions of $G$ on a nearly equal footing. We record this in the following lemma whose proof is immediate. First it will help to introduce some notation that will be used repeatedly throughout the rest of the paper. If $G$ acts on $X$ and $\cH$ is a finite collection of pairwise conjugate subgroups of $G$, write
$$X^{\cH} \ := \ \{x\in X \: \Stab(x) \in \cH\}.$$

\begin{lem}
Let $G$ be a countable group, let $\Phi = (\ell, \phi, \rZ, \Gamma, \cH)$ be a chart for $G$, and let $G$ act on the set $X$. Then for every $\T{r}, \T{s} \in \dom(\phi)$ and every $x \in X^\cH$,
\begin{align*}
\phi(\T{r}) \cdot x = \phi(\T{s}) \cdot x & \ \Longrightarrow \ \T{r} = \T{s}; \\
\T{r} + \T{s} + \rZ \subseteq \dom(\phi) & \ \Longrightarrow \ \exists \T{z} \in \rZ \quad \phi(\T{r}) \cdot \phi(\T{s}) \cdot x = \phi(\T{r} + \T{s} + \T{z}) \cdot x; \\
\T{r} - \T{s} + \rZ \subseteq \dom(\phi) & \ \Longrightarrow \ \exists \T{z} \in \rZ \quad \phi(\T{r}) \cdot \phi(\T{s})^{-1} \cdot x = \phi(\T{r} - \T{s} + \T{z}) \cdot x; \\
-\T{r} + \T{s} + \rZ \subseteq \dom(\phi) & \ \Longrightarrow \ \exists \T{z} \in \rZ \quad \phi(\T{r})^{-1} \cdot \phi(\T{s}) \cdot x = \phi(-\T{r} + \T{s} + \T{z}) \cdot x; \\
-\T{s} + \rZ \subseteq \dom(\phi) & \ \Longrightarrow \ \exists \T{z} \in \rZ \quad \phi(\T{s})^{-1} \cdot x = \phi(-\T{s} + \T{z}) \cdot x.
\end{align*}
\end{lem}

We now end the section with the construction of general charts.

\begin{lem} \label{LEM NFEXIST}
Let $G$ be a finitely generated nilpotent group, let $S \neq \varnothing$ be a finite collection of pairwise conjugate subgroups of $G$, let $F \subseteq G$ be finite, and let $\lambda, \eta > 0$. Then there is a chart $\Phi = (\ell, \phi, \rZ, \Gamma, \cH)$ for $G$ such that $\ell$ is at most the Hirsch length of $G$, $\lambda \cdot \rZ \subseteq \dom(\phi)$, $F \subseteq \phi(\rZ) H$ for every $H \in \cH$, and
$$\phi(\T{u}) H \phi(\T{u})^{-1} \in \cH \quad \text{for every } H \in S \text{ and } \T{u} \in \eta \cdot \rZ.$$
\end{lem}

\begin{proof}
Without loss of generality, we may assume that $\lambda \geq \eta \geq 3$. We will prove the claim by induction on the nilpotency class of $G$. First suppose that $G$ has nilpotency class $1$. Then $G$ is abelian and $S$ is a singleton, say $S = \{H\}$. We have that $G / H$ is a finitely generated abelian group and is thus isomorphic to an abelian group of the form $\Z^\ell \times \Gamma$, where $\Gamma$ is a finite abelian group. Note that $\ell$ is at most the Hirsch length of $G$. Fix an isomorphism $\psi : \Z^\ell \times \Gamma \rightarrow G / H$. Let $\rZ \subseteq \Z^\ell \times \Gamma$ be a centered rectangle with $F\cdot H\subseteq \psi(\rZ)$ and each $\bdim_i(\rZ)>0$. Set $\cH = \{H\}$. Define $\phi:\lambda\cdot\rZ\to G$ so that $\phi(\T{0})=1_G$ and for $\T{0}\ne\T{r}\in\lambda\cdot\rZ$, $\phi(\T{r})$ is any element of $G$ with $\phi(\T{r})H=\psi(\T{r})$. The desired properties are not difficult to verify, and this completes the proof in this case.

Now suppose that $G$ is of nilpotency class $n$ and that the claim holds for all finitely generated nilpotent groups of nilpotency class less than $n$. Set
$$S_0 = \{H \zeta(G) / \zeta(G) \: H \in S\}.$$
By induction, there is a chart $\Phi_0 = (\ell_0, \phi_0, \rZ_0, \Gamma_0, \cH_0)$ for $G / \zeta(G)$ having the desired properties for $S_0$, $F \cdot \zeta(G)$, $\lambda$, and $\eta$. Set $\phi_0'(\T{0}) = 1_G$ and in general for $\T{0} \neq \T{r} \in \dom(\phi_0)$ let $\phi_0'(\T{r}) \in G$ be any group element with $\phi_0'(\T{r}) \zeta(G) = \phi_0(\T{r})$. Set
$$\cH = \{g H g^{-1} \: H \in S, \ g \in \phi_0'(\eta \cdot \rZ_0) \}.$$
Observe that
$$\{H \zeta(G) / \zeta(G) \: H \in \cH\} \ \subseteq \ \cH_0.$$
Since $\cH$ is finite and $\dom(\phi_0')$ is finite, there exists a finite set $B \subseteq \zeta(G)$ such that $F \subseteq \phi_0'(\rZ_0) B H$ for every $H \in \cH$ and such that for all $\T{r}, \T{s} \in \dom(\phi_0')$ and $H \in \cH$,
\begin{align*}
\T{r} + \T{s} + \rZ_0 \subseteq \dom(\phi_0') & \ \Longrightarrow \ \phi_0'(\T{r}) \cdot \phi_0'(\T{s}) \in \phi_0'(\T{r} + \T{s} + \rZ_0) \cdot B \cdot H; \\
\T{r} - \T{s} + \rZ_0 \subseteq \dom(\phi_0') & \ \Longrightarrow \ \phi_0'(\T{r}) \cdot \phi_0'(\T{s})^{-1} \in \phi_0'(\T{r} - \T{s} + \rZ_0) \cdot B \cdot H; \\
-\T{r} + \T{s} + \rZ_0 \subseteq \dom(\phi_0') & \ \Longrightarrow \ \phi_0'(\T{r})^{-1} \cdot \phi_0'(\T{s}) \in \phi_0'(-\T{r} + \T{s} + \rZ_0) \cdot B \cdot H; \\
-\T{s} + \rZ_0 \subseteq \dom(\phi_0') & \ \Longrightarrow \ \phi_0'(\T{s})^{-1} \in \phi_0'(-\T{s} + \rZ_0) \cdot B \cdot H.
\end{align*}
(It may help the reader to note that all four implications hold modulo $\zeta(G)$).

Pick any $H \in \cH$ and set $S_1 = \{H \cap \zeta(G)\}$, noting that for any $g \in G$ and $H \leq G$ we have $g H g^{-1} \cap \zeta(G) = H \cap \zeta(G)$. Since $\zeta(G)$ is finitely generated and abelian, it follows from the base case of the induction that there is a chart $\Phi_1 = (\ell_1, \phi_1, \rZ_1, \Gamma_1, \cH_1)$ for $\zeta(G)$ having the desired properties for $S_1$, $B$, $\lambda$, and $\eta$. Set $\ell = \ell_0 + \ell_1$, $\Gamma = \Gamma_0 \times \Gamma_1$, $\rZ = \rZ_0 \times \rZ_1$, and define
$$\phi(\T{r}_0, \T{r}_1) = \phi_0'(\T{r}_0) \cdot \phi_1(\T{r}_1)$$
for $(\T{r}_0, \T{r}_1) \in \dom(\phi) = \dom(\phi_0') \times \dom(\phi_1) \subseteq \Z^\ell \times \Gamma$. It only remains to check that $\Phi = (\ell, \phi, \rZ, \Gamma, \cH)$ has the desired properties. From the induction it is easy to see that $\ell$ is at most the Hirsch length of $G$ and that $\lambda \cdot \rZ \subseteq \dom(\phi)$. From the definition of $B$ and the fact that $B \subseteq \phi_1(\rZ_1)$ we see that $F \subseteq \phi(\rZ) H$ for every $H \in \cH$. The definition of $\cH$ and the fact that the image of $\phi_1$ is central imply that $\phi(\T{u}) H \phi(\T{u})^{-1} \in \cH$ for every $H \in S$ and $\T{u} \in \eta \cdot \rZ$. The four ``almost homomorphism'' properties of $\phi$ hold again due to the definition of $B$, the fact that $B \subseteq \phi_1(\rZ_1)$, and the fact that the image of $\phi_1$ is central. Finally, let $\T{r}_0, \T{s}_0 \in \dom(\phi_0')$, $\T{r}_1, \T{s}_1 \in \dom(\phi_1)$, and $H \in \cH$ be such that $\phi(\T{r}_0, \T{r}_1) \cdot H = \phi(\T{s}_0, \T{s}_1) \cdot H$. Then $\phi_0(\T{r}_0) \cdot H \zeta(G) / \zeta(G) = \phi_0(\T{s}_0)\cdot H \zeta(G) / \zeta(G)$ and hence $\T{r}_0 = \T{s}_0$. Therefore $\phi_1(\T{r}_1) \cdot H = \phi_1(\T{s}_1) \cdot H$ and hence $\phi_1(\T{s}_1)^{-1} \phi_1(\T{r}_1)$ lies in $H \cap \zeta(G) \in S_1$. So $\phi_1(\T{r}_1) \cdot (H \cap \zeta(G)) = \phi_1(\T{s}_1) \cdot (H \cap \zeta(G))$ and thus $\T{r}_1 = \T{s}_1$.
\end{proof}

The above two lemmas show that we can arrange for the domain of $\phi$ to be as large as desired relative to $\rZ$. In our arguments we will never require the domain of $\phi$ to be small, so the domain of $\phi$ is not a significant concern. We will state precise requirements on the domain of $\phi$ which are sufficient, but we generally refrain from explicitly checking that these requirements are indeed sufficient, since checking this at every step would make our arguments quite tedious to read. Nevertheless, we will at a few times comment on the sufficiency of the domain when the method of verification is not straightforward.

\section{Rough rectangles} \label{SEC RECT}

The main purpose of charts, and the motivation behind their definition, is that they enable us to see rough rectangles in spaces on which $G$ acts. This idea is stated precisely in the following definition. Recall that $X^{\cH} = \{ x \in X \: \Stab(x) \in \cH\}$.

\begin{defn} \label{DEFN ROUGH}
Let $G$ act on the set $X$, let $A \subseteq \Z^\ell \times \Gamma$ be a rectangle, let $\Phi = (\ell, \phi, \rZ, \Gamma, \cH)$ be a chart for $G$ with $2 \cdot A \subseteq \dom(\phi)$, and let $x\in X^\cH$ and $0 < \epsilon < 1$. A set $R \subseteq X$ is \emph{$(\Phi, \epsilon)$-roughly $A$ at $x$} if $2 \cdot \rZ \sqsubseteq \epsilon \cdot A$ and
$$\phi \Big( (1-\epsilon) \cdot A \Big) \cdot x \ \subseteq \ R \ \subseteq \ \phi \Big( (1+\epsilon) \cdot A \Big) \cdot x.$$
When the above conditions are satisfied we call $x$ a \emph{$(\Phi, A, \epsilon)$-base for $R$}. When we wish not to emphasize $x$ we simply say that $R$ is \emph{$(\Phi, \epsilon)$-roughly $A$}.
\end{defn}

It seems likely that if $G$ is nilpotent and not virtually abelian, then there are no ``rough-cubes'' in spaces on which $G$ acts freely. This may be surprising at first given Lemma \ref{LEM EXIST}. The obstruction is that although $\lambda \cdot \rZ \subseteq \dom(\phi)$, the maximum ratio $\bdim_i(\rZ) / \bdim_j(\rZ)$ ($1 \leq i \neq j \leq \ell$) will generally be greater than $\lambda$ if one follows the construction appearing in the proof of Lemma \ref{LEM EXIST}.

Notice that the requirement $2 \cdot \rZ \sqsubseteq \epsilon \cdot A$ implies $\lambda \cdot \rZ \sqsubseteq \lambda \epsilon \cdot A$ for all $\lambda > 0$ by clause (iv) of Lemma \ref{LEM BASIC}. We will use this fact freely without explicit mention of Lemma \ref{LEM BASIC}. Notice also that $x$ need not belong to the set $R\subseteq X$ for $R$ to be $(\Phi,\epsilon)$-roughly $A$ at $x$. Intuitively the intended description here is that from the vantage point of $x$ and through the chart $\Phi$, $R$ is approximately a rectangle.

The definition of chart says that the map $\phi$ behaves almost like a homomorphism from its domain in $\Z^\ell \times \Gamma$ into a coset space of $G$. With this mindset, intuitively one should believe that if $R$ is a $(\Phi, \epsilon)$-rough rectangle at $x$ and $y \in X^\cH$ is nearby, then $R$ should still be roughly rectangular at $y$. The following lemma confirms this fact.

\begin{lem} \label{LEM SHIFT}
Let $G$ act on the set $X$, let $\Phi = (\ell, \phi, \rZ, \Gamma, \cH)$ be a chart for $G$, let $R \subseteq X$ be $(\Phi, \epsilon)$-roughly $A$ at $x \in X^\cH$ with $\epsilon<\frac{1}{2}$, and let $M$ be a positive integer satisfying $(M+2) \cdot A + 4 \cdot \rZ \subseteq \dom(\phi)$. Let $y \in X^\cH$.
\begin{enumerate}
\item[\rm (i)] If $y = \phi(\T{u}) \cdot x$ and $\T{u} \in M \cdot A$, then $R$ is $(\Phi, 2\epsilon)$-roughly $A - \T{u}$ at $y$.
\item[\rm (ii)] If $x = \phi(\T{v}) \cdot y$ and $-\T{v} \in M \cdot A$, then $R$ is $(\Phi, 2\epsilon)$-roughly $A + \T{v}$ at $y$.
\end{enumerate}
\end{lem}

\begin{proof}
In case (ii) we have by the definition of chart that $y = \phi(\T{v})^{-1} \cdot x \in \phi(-\T{v} + \rZ) \cdot x$ since $-\T{v} + \rZ \subseteq \dom(\phi)$. So there is $\T{u} \in -\T{v} + \rZ \subseteq M \cdot A + \rZ$ with $y = \phi(\T{u}) \cdot x$. So in either case (i) or case (ii) there is $\T{u} \in M \cdot A + \rZ$ with $y = \phi(\T{u}) \cdot x$. In either case fix such a $\T{u}$, and let $\T{z}\in\rZ$ be arbitrary. From the definition of chart we obtain
\begin{equation}
\begin{array}{lll}
\phi \Big( (1 - 2 \epsilon) \cdot A - \T{u} + \T{z} \Big) \phi \Big( \T{u} \Big) \cdot x & \subseteq & \phi \Big( (1-2\epsilon) \cdot A - \T{u} + \T{u} + \T{z} + \rZ \Big) \cdot x \\ & \subseteq & \phi \Big( (1 - \epsilon) \cdot A \Big) \cdot x
\end{array}\label{EQ 1}
\end{equation}
and
\begin{equation}
\begin{array}{lll}
\phi \Big( (1 + \epsilon) \cdot A \Big) \phi \Big( \T{u} \Big)^{-1} \cdot y & \subseteq & \phi \Big( (1+\epsilon) \cdot A + \T{z} + \rZ \Big) \phi \Big( \T{u} \Big)^{-1} \cdot y \\ & \subseteq & \phi \Big( (1+\epsilon) \cdot A - \T{u} + \T{z} + 2 \cdot \rZ \Big) \cdot y \\ & \subseteq & \phi \Big( (1+2\epsilon) \cdot A - \T{u} + \T{z} \Big) \cdot y.
\end{array}\label{EQ 2}
\end{equation}
In appealing to the definition of chart, for the first inclusion in \eqref{EQ 1} we need that $\T{u}\in\dom(\phi)$, $(1-2\epsilon)\cdot A-\T{u}+\T{z}\subseteq\dom(\phi)$, and $(1-2\epsilon)\cdot A+\T{z}+\rZ\subseteq\dom(\phi)$, and for the second inclusion in \eqref{EQ 2} we need that $\T{u}\in\dom(\phi)$, $(1+\epsilon)\cdot A+\T{z}+\rZ\subseteq\dom(\phi)$, and $(1+\epsilon)\cdot A-\T{u}+\T{z}+2\cdot\rZ\subseteq\dom(\phi)$. Examining the last of these six conditions more closely, we see that
$$(1+\epsilon)\cdot A-\T{u}+\T{z}+2\cdot\rZ \ \subseteq \ 2\cdot A-M\cdot A+4\cdot\rZ \ \subseteq \ (M+2)\cdot C+4\cdot\rZ,$$
where $C$ is the centered translate of $A$. By definition $\dom(\phi)$ is a centered rectangle, so $(M + 2) \cdot A + 4 \cdot \rZ \subseteq \dom(\phi)$ implies $(M + 2) \cdot C + 4 \cdot \rZ \subseteq \dom(\phi)$. Given this observation, one easily checks that each of the six conditions above is implied by the condition $(M+2)\cdot A+4\cdot\rZ\subseteq\dom(\phi)$, and this explains the hypothesis on $\dom(\phi)$ in the statement of the lemma.

Now, from \eqref{EQ 1} we obtain
$$\phi \Big( (1-2\epsilon) \cdot A - \T{u} + \T{z} \Big) \cdot y \ \subseteq \ \phi \Big( (1-\epsilon) \cdot A \Big) \cdot x \ \subseteq \ R,$$
and from \eqref{EQ 2} we obtain
$$R \ \subseteq \ \phi \Big( (1+\epsilon) \cdot A \Big) \cdot x \ \subseteq \ \phi \Big( (1+2\epsilon) \cdot A - \T{u} + \T{z} \Big) \cdot y.$$
Claim (i) then follows from using $\T{z} = \T{0}$ and (ii) follows from using $\T{z} = \T{u} + \T{v}$ (the second sentence of this proof shows that indeed $\T{u}+\T{v} \in \rZ$).
\end{proof}

Lemma \ref{LEM SHIFT} implies that under mild assumptions a base for a rough rectangle $R$ can be translated to a base lying ``near the center" of $R$ at the expense of just a small error. More precisely, suppose the set $R\subseteq X^\cH$ is $(\Phi,\epsilon)$-roughly $A$ at $x \in X^\cH$ where $3\cdot A+4\cdot\rZ\subseteq\dom(\phi)$ and $\epsilon<\frac{1}{2}$. Let $\T{c}$ be the center of $A$ and $C$ the centered translate of $A$, so that $A=\T{c}+C$ and in particular $\T{c}\in A$. Then by Lemma \ref{LEM SHIFT}, $R$ is $(\Phi,2\epsilon)$-roughly $C$ at $\phi(\T{c})\cdot x$ where $\phi(\T{c})\cdot x$ belongs to $R$.

Equivalence relations in which every class meeting $X^\cH$ is roughly rectangular will play an important role in our arguments. These equivalence relations will have classes which are uniformly large rectangles (but also uniformly small relative to the domain of $\phi$) with uniformly small error. This is captured in the following definition.

\begin{defn} \label{DEFN RECTEQ}
Let $G$ act on the set $X$, let $\Phi = (\ell, \phi, \rZ, \Gamma, \cH)$ be a chart for $G$, let $0 < \epsilon < 1$, and let $A \subseteq \Z^\ell \times \Gamma$ be a centered rectangle. An equivalence relation $E$ on $X$ is \emph{$(\Phi, A, \epsilon)$-rectangular} if every class of $E$ not meeting $X^\cH$ is a singleton and for every class $U$ of $E$ meeting $X^\cH$ there is $ \delta > 0$ and a rectangle $B \subseteq \Z^\ell \times \Gamma$ such that $U$ is $(\Phi, \delta)$-roughly $B$ where
\begin{enumerate}
\item[\rm (i)] $A \sqsubseteq B$ (uniformly large rectangles);
\item[\rm (ii)] $2^{22\ell} \cdot B \subseteq \dom(\phi)$ (uniformly small rectangles relative to $\dom(\phi)$);
\item[\rm (iii)] $2 \delta \cdot B \sqsubseteq \epsilon \cdot A$ and hence $\delta \leq \epsilon / 2$ (uniformly small error).
\end{enumerate}
We say $E$ is \emph{$(\Phi, A, \epsilon)$-sub-rectangular} if $E$ contains a $(\Phi, A, \epsilon)$-rectangular sub-equivalence relation.
\end{defn}

We note that in the definition above we require $2 \delta \cdot B \sqsubseteq \epsilon \cdot A$ so that we may easily apply clause (iv) of Lemma \ref{LEM BASIC}. In the context of (sub-)rectangular equivalence relations we will use clause (iv) of Lemma \ref{LEM BASIC} with high frequency and will therefore refrain from explicitly citing it. Notice that the existence of a $(\Phi, A, \epsilon)$-rectangular equivalence relation has strong implications for $A$, $\epsilon$, and $\phi$. Specifically it implies that $2 \cdot \rZ \subseteq \epsilon \cdot A$ and $2^{22 \ell} \cdot A \subseteq \dom(\phi)$ (since $A$ is centered). Furthermore, we emphasize that if $E$ is $(\Phi,A,\epsilon)$-rectangular and $U$ is an $E$-class meeting $X^\cH$, then we can take $U$ to be $(\Phi,\delta)$-roughly $B$ with $\delta\leq\epsilon/2<\frac{1}{2}$. We shall make frequent use of this bound on $\delta$ without explicit mention.

One may notice that in the definition above, there is no requirement on where the $(\Phi, B, \delta)$-base for $U$ is (although by definition it must lie in $X^\cH$). This poses no difficulties as the lemma below shows that if $E$ is $(\Phi,A,\epsilon)$-rectangular and $x \in X^\cH$ is near enough some $E$-class $U$ that is $(\Phi,\delta)$-roughly $B$, then we can find $\T{v}\in\dom(\phi)$ such that $\phi(\T{v})\cdot x$ is a $(\Phi,B,\delta)$-base for $U$. Moreover $U$ will be $(\Phi,2\delta)$-roughly $B+\T{v}$ at $x$.

\begin{lem} \label{LEM BASE}
Let $G$ act on the set $X$, let $\Phi=(\ell,\phi,\rZ,\Gamma, \cH)$ be a chart for $G$, let $0<\epsilon<1$, let $A\subseteq\Z^\ell\times\Gamma$ be a centered rectangle, and let $E$ be a $(\Phi,A,\epsilon)$-rectangular equivalence relation on $X$. Let $U$ be an $E$-class meeting $X^\cH$, let $y \in X^\cH$, and suppose that $\phi(M \cdot A) \cdot y$ meets $U$, where $M \leq 2^{22 \ell} - 8$. Then there is $\T{v} \in \dom(\phi)$, a rectangle $B \subseteq \Z^\ell \times \Gamma$, and $\delta > 0$ such that $U$ is $(\Phi, \delta)$-roughly $B$ at $\phi(\T{v}) \cdot y \in X^\cH$, $A \sqsubseteq B$, $2 \delta \cdot B \sqsubseteq \epsilon \cdot A$, and $2^{22 \ell} \cdot B \subseteq \dom(\phi)$. Moreover $U$ is $(\Phi, 2 \delta)$-roughly $B + \T{v}$ at $y$.
\end{lem}

\begin{proof}
Since $E$ is $(\Phi, A, \epsilon)$-rectangular, there exist $x \in X^\cH$, a rectangle $B$, and $\delta > 0$ such that $U$ is $(\Phi, \delta)$-roughly $B$ at $x$, $A \sqsubseteq B$, $2 \delta \cdot B \sqsubseteq \epsilon \cdot A$, and $2^{22 \ell} \cdot B \subseteq \dom(\phi)$. In particular $U \subseteq \phi((1 + \delta) \cdot B) \cdot x$ and therefore
$$\phi((1 + \delta) \cdot B) \cdot x \ \bigcap \ \phi(M \cdot A) \cdot y \ \neq \ \varnothing.$$
So there exist $\T{r} \in (1 + \delta) \cdot B$ and $\T{s} \in M \cdot A$ with $x = \phi(\T{r})^{-1} \phi(\T{s}) \cdot y$. Since $A$ is centered and $A \sqsubseteq B$, clause (v) of Lemma \ref{LEM BASIC} gives
$$-\T{r} + \T{s} + \rZ \ \subseteq \ M \cdot A - (1 + \delta) \cdot B + \rZ \ \subseteq \ M \cdot A - 2 \cdot B \ \subseteq \ -(M + 2) \cdot B.$$
Now $M + 2 < 2^{22 \ell}$ and $\dom(\phi)$ is a centered rectangle, so $2^{22 \ell} \cdot B \subseteq \dom(\phi)$ implies $-(M + 2) \cdot B \subseteq \dom(\phi)$. Thus by the definition of chart we can find $\T{v} \in M \cdot A - 2 \cdot B$ with $\phi(\T{v}) \cdot y = \phi(\T{r})^{-1} \phi(\T{s}) \cdot y$. Thus $x = \phi(\T{v}) \cdot y$. We have $- \T{v} \in (M + 2) \cdot B$ and
$$((M + 2) + 2) \cdot B + 4 \cdot \rZ \ \subseteq \ (M + 8) \cdot B \ \subseteq \ 2^{22 \ell} \cdot B \ \subseteq \ \dom(\phi)$$
and also $\delta \leq \epsilon / 2 < 1/2$, so by clause (ii) of Lemma \ref{LEM SHIFT} we conclude that $U$ is $(\Phi, 2 \delta)$-roughly $B + \T{v}$ at $y$.
\end{proof}

\section{Facial boundaries} \label{SEC FACES}

When working with a rectangular equivalence relation $E$, we want an unambiguous way of defining the ``facial boundaries'' of $E$. Such a definition must be somewhat delicate in order to overcome the rough nature of our rectangles. Furthermore, we want a definition that does not rely on picking, for each class $U$ of $E$ meeting $X^\cH$, a $B$, $\delta$, and $x$ such that $U$ is $(\Phi, \delta)$-roughly $B$ at $x$. This is achieved in the definition below. Recall that if $E$ is an equivalence relation on $X$ and $Y \subseteq X$ then we write $[Y]_E$ for the set of all $x \in X$ which are $E$-equivalent to some $y \in Y$. Also recall that for a rectangle $A \subseteq \Z^\ell \times \Gamma$, the sets $A^i - \bdim_i(A) \cdot \T{e}_i$ and $A^i + \bdim_i(A) \cdot \T{e}_i$ are the two faces of $A$ which are perpendicular to the coordinate vector $\T{e}_i$.

\begin{defn}
Let $G$ act on the set $X$, let $\Phi = (\ell, \phi, \rZ, \Gamma, \cH)$ be a chart for $G$, and let $E$ be an equivalence relation on $X$. For a centered rectangle $A \subseteq \dom(\phi) \subseteq \Z^\ell \times \Gamma$ and $1 \leq i \leq \ell$ we define the \emph{$(\Phi, i, A)$-boundary of $E$}, denoted $\partial_i^\Phi(E, A)$, to be the set of $x \in X^\cH$ such that
$$\Big[ \phi \big( A^i - \bdim_i(A) \cdot \T{e}_i \big) \cdot x \Big]_E \ \bigcap \ \Big[ \phi \big( A^i + \bdim_i(A) \cdot \T{e}_i \big) \cdot x \Big]_E \ = \ \varnothing.$$
\end{defn}

Notice that $\partial_i^\Phi(E, A) \subseteq X^\cH$ and also that $\partial_i^\Phi(E, A) \subseteq \partial_i^\Phi(F, A)$ whenever $F \subseteq E$ is a sub-equivalence relation of $E$. In the next lemma we show that boundaries behave well with respect to the property of $G$-clopenness discussed in Section \ref{SEC MARKER}.

\begin{lem} \label{LEM BNDCLOPEN}
Let $G$ act continuously on the Polish space $X$, let $\Phi = (\phi, \ell, \rZ, \Gamma, \cH)$ be a chart for $G$, and let $E$ be an equivalence relation on $X$. If $X^\cH$ is clopen and $E$ is $G$-clopen then $\partial_i^\Phi(E, A)$ is clopen for every centered rectangle $A \subseteq (\Z^\ell \times \Gamma) \cap \dom(\phi)$ and every $1 \leq i \leq \ell$.
\end{lem}

\begin{proof}
Fix $1 \leq i \leq \ell$. Set $\T{a} = \bdim(A)$ and for $g \in G$ set $E_g = \{x \in X \: x \ E \ g \cdot x\}$. We have that $x \in \partial_i^\Phi(E, A)$ if and only if $x \in X^\cH$ and $\phi(a_i \T{e}_i + \T{u}) \cdot x$ and $\phi(-a_i \T{e}_i + \T{v}) \cdot x$ are $E$-inequivalent for every $\T{u}, \T{v} \in A^i$. Rewriting this, $x \in \partial_i^\Phi(E, A)$ if and only if
$$x \ \in \ X^\cH \setminus \phi(a_i \T{e}_i + \T{u})^{-1} \cdot E_{\phi(-a_i \T{e}_i + \T{v}) \phi(a_i \T{e}_i + \T{u})^{-1}}$$
for every $\T{u}, \T{v} \in A^i$. The above set is clopen and $A^i$ is finite, so the claim follows.
\end{proof}

Intuitively, when the classes of $E$ are roughly rectangular the set $\partial_i^\Phi(E, A)$ should be the union of the ``faces'' of $E$ that are perpendicular to $\T{e}_i$. The following two lemmas confirm this intuition by showing that boundaries are located close to where one would expect.

\begin{lem} \label{LEM FACES}
Let $G$ act on the set $X$, let $\Phi = (\ell, \phi, \rZ, \Gamma, \cH)$ be a chart for $G$, let $\epsilon > 0$, let $q$ satisfy $6 \epsilon < q < 1 - \epsilon$, and let $A \subseteq \Z^\ell \times \Gamma$ be a centered rectangle. Let $E$ be a $(\Phi, A, \epsilon)$-rectangular equivalence relation on $X$ and suppose the class $U$ of $E$ is $(\Phi, \delta)$-roughly $B$ at $x \in X^\cH$ where $A \sqsubseteq B$, $2^{22 \ell} \cdot B \subseteq \dom(\phi)$, and $2 \delta \cdot B \sqsubseteq \epsilon \cdot A$. If $\T{b} = \bdim(B)$ then for each $1\leq i\leq\ell$,
$$\partial_i^\Phi(E, q \cdot A) \cap U \ \subseteq \ \phi \Big( {-} b_i \cdot \T{e}_i + B^i + 2 q \cdot A \Big) \cdot x \cup \phi \Big( b_i \cdot \T{e}_i + B^i + 2 q \cdot A \Big) \cdot x.$$
\end{lem}

\begin{proof}
We remind the reader that $2 \delta \leq \epsilon$ (see the paragraph following Definition \ref{DEFN RECTEQ}). Fix $y \in \partial_i^\Phi(E, q \cdot A) \cap U \subseteq X^\cH$. As $U \subseteq \phi((1 + \delta) \cdot B) \cdot x$, we may pick $\T{t} \in (1 + \delta) \cdot B$ with $y = \phi(\T{t}) \cdot x$. Since $6 \delta \cdot B \sqsubseteq 6 \epsilon \cdot A \sqsubseteq q \cdot A$ and $q\cdot A$ is centered, clause (vi) of Lemma \ref{LEM BASIC} implies
$$(1 + \delta) \cdot B \ \subseteq \ (1 - 2 \delta) \cdot B + q \cdot A.$$
Therefore $\T{t} + q \cdot A$ must meet $(1 - 2 \delta) \cdot B$. In particular, since $q \cdot A \sqsubseteq (1 - 2 \delta) \cdot A \sqsubseteq (1 - 2 \delta) \cdot B$ there must be a corner vertex of $q \cdot A$, call it $\T{c}$, such that $\T{t} + \T{c} \in (1 - 2 \delta) \cdot B$. Let $\T{a} = \bdim(q \cdot A)$ and choose $j = \pm 1$ so that $\T{c} \in q \cdot A^i + j \cdot a_i \cdot \T{e}_i$. By Lemma \ref{LEM SHIFT} we have that $U$ is $(\Phi, 2 \delta)$-roughly $B - \T{t}$ at $y \in \partial_i^\Phi(E, q \cdot A) \subseteq X^\cH$. So $\phi(\T{c}) \cdot y \in \phi((1 - 2 \delta) \cdot B - \T{t}) \cdot y \subseteq U$. As $y \in \partial_i^\Phi(E, q \cdot A)$ and $\phi(\T{c}) \cdot y \in U$ we must have
$$\phi \Big( \T{c} - 2 j \cdot a_i \cdot \T{e}_i \Big) \cdot y \ \not\in \ U.$$
Thus $\T{t} + \T{c} \in (1 - 2 \delta) \cdot B$, but if we translate $\T{t}+\T{c}$ by $-2j\cdot a_i\cdot\T{e}_i$ we see that
$$\T{t} + \T{c} - 2 j \cdot a_i \cdot \T{e}_i \ \not\in \ (1 - 2 \delta) \cdot B.$$
Translating $\T{t}+\T{c}$ even farther, we claim that $\T{t} + \T{c} - 3 j \cdot a_i \cdot \T{e}_i$ is not even in $B$. For this it suffices to show that
$$b_i-\lfloor (1-2\delta) b_i\rfloor \ \leq \ a_i.$$ Using $2\cdot\rZ\sqsubseteq\delta\cdot B$ and $6\delta\cdot B\sqsubseteq q\cdot A$ we see that indeed
$$b_i-\lfloor (1-2\delta) b_i\rfloor \ = \ \lceil 2\delta b_i\rceil \ \leq \ \lfloor 3\delta b_i\rfloor \ < \ \lfloor 6\delta b_i\rfloor \ \leq \ a_i.$$
Therefore $\T{t}+\T{c} \in B$ and $\T{t} + (\T{c}-3j\cdot a_i\cdot\T{e}_i)\not\in B$, where $\T{c}-3j\cdot a_i\cdot \T{e}_i\in 2q\cdot A$.  Hence $\T{t} + 2q\cdot A$ meets both $B$ and its complement.  More specifically, since the points $\T{t} + \T{c}$ and $\T{t} + \T{c}-3j\cdot a_i\cdot\T{e}_i$ differ only in their $i$th coordinate, $\T{t}+2q\cdot A$
meets $(-b_i \cdot \T{e}_i + B^i) \cup (b_i \cdot \T{e}_i + B^i)$. It now follows from $A$ being centered that
$$\T{t} \ \in \ (-b_i \cdot \T{e}_i + B^i + 2 q \cdot A) \cup (b_i \cdot \T{e}_i + B^i + 2 q \cdot A).$$
This completes the proof since $y = \phi(\T{t}) \cdot x$.
\end{proof}

Below is the second lemma showing that boundaries occur near where one would expect. The reader may notice in this lemma that one minor shortcoming of our notion of boundary is that it only detects the highest dimensional boundaries of rough rectangles, failing in general to detect lower dimensional boundaries like corners or edges. For the most part, ignoring lower dimensional boundaries simplifies our arguments; however, it does make proving the lemma below more difficult.

\begin{lem} \label{LEM STRONG BND}
Let $G$ act on the set $X$, let $\Phi = (\ell, \phi, \rZ, \Gamma, \cH)$ be a chart for $G$, let $\epsilon > 0$, let $q$ satisfy $12 \epsilon < q < 1/ (24 \ell)$, and let $A \subseteq \Z^\ell \times \Gamma$ be a centered rectangle. Suppose $E$ is a $(\Phi, A, \epsilon)$-sub-rectangular equivalence relation on $X$. If $x \in X^\cH$ satisfies $\phi(15 \ell q \cdot A) \cdot x \subseteq X^\cH$ and $\phi(3 \cdot \rZ) \cdot x \not\subseteq [x]_E$, then there is $1 \leq i \leq \ell$ with $x \in \phi(30 \ell \cdot (q \cdot A)) \cdot \partial_i^\Phi(E, q \cdot A)$.
\end{lem}

\begin{proof}
We remark that we use $30 \ell \cdot (q \cdot A)$ instead of $30 \ell q \cdot A$ for a minor technical reason that will be made clear in Section \ref{SEC ORTHO} (see for example Definition \ref{DEFN ORTHO}). In fact we will show that there is $1\leq i\leq\ell$ with $x\in\phi(15\ell q\cdot A)\cdot\partial_i^{\Phi}(E,q\cdot A)$.  This suffices since $15 \ell q \cdot A \subseteq 30 \ell \cdot (q \cdot A)$ by the following computation. Set $\T{a} = \bdim(A)$. As $\Rect(\T{1}) \sqsubseteq \rZ \sqsubseteq \epsilon \cdot A \sqsubseteq q \cdot A$, we have that
$$\lfloor 15 \ell q a_i \rfloor \ \leq \ 15 \ell \lfloor q a_i \rfloor + 15 \ell \ \leq \ 15 \ell \lfloor q a_i \rfloor + 15 \ell \lfloor q a_i \rfloor \ = \ 30 \ell \lfloor q a_i \rfloor.$$

Assume that $\phi(15 \ell q \cdot A) \cdot x \subseteq X^\cH$ and that, for every $1 \leq i \leq \ell$, $x \not\in \phi(15 \ell q \cdot A) \cdot \partial_i^\Phi(E, q \cdot A)$. It suffices to show that $\phi(3 \cdot \rZ) \cdot x \subseteq [x]_E$. Let $F$ be a $(\Phi, A, \epsilon)$-rectangular sub-equivalence relation of $E$, and let $V$ be an $F$-class with $V \cap \phi(3 \cdot \rZ) \cdot x \neq \varnothing$. It will suffice to show that $V \subseteq [x]_E$. We have that $V$ meets $X^\cH$. So using Lemma \ref{LEM BASE}, fix $\T{t}$, $B$, and $\delta$ such that $V$ is $(\Phi, 2 \delta)$-roughly $B + \T{t}$ at $x$, where $A \sqsubseteq B$, $2 \delta \cdot B \sqsubseteq \epsilon \cdot A$, and $2^{22 \ell} \cdot B \subseteq \dom(\phi)$. Let $\T{v} \in \Z^\ell \times \{1_\Gamma\}$ be the vector with smallest coordinates, in absolute value, satisfying
$$\T{v} + 6 \ell q \cdot A \ \subseteq \ (1 - 2 \delta) \cdot B + \T{t}.$$
Such a $\T{v}$ exists since
$$6 \ell q \cdot A \ \sqsubseteq \ (1/2) \cdot A \ \sqsubseteq \ (1 - 2 \delta) \cdot A \ \sqsubseteq \ (1 - 2 \delta) \cdot B.$$
Notice that since $\phi(3 \cdot \rZ) \cdot x$ meets $V$ and $3 \cdot \rZ \subseteq 3 \epsilon \cdot A \subseteq q \cdot A$, we have that $q\cdot A$ meets $(1 + 2 \delta) \cdot B + \T{t}$. As $8 \delta \cdot B \sqsubseteq 8 \epsilon \cdot A \sqsubseteq q \cdot A$, it therefore follows from Lemma \ref{LEM BASIC2}(i) that $2 q \cdot A$ meets $(1 - 2 \delta) \cdot B + \T{t}$. Hence we see that $\T{v} \in (6 \ell + 2) q \cdot A$.

The idea now is to travel from $x = \phi(\T{0}) \cdot x$ to a point $y=\phi(\T{w})\cdot x$ in $\phi(\T{v}+6\ell q\cdot A)\cdot x\subseteq V$ without leaving the $E$-class of $x$. We do this in $\ell$-stages. In each stage we change (primarily) a single coordinate $1 \leq i \leq \ell$. In changing a single coordinate, we want to pass through $q \cdot A$-facial boundaries of $F$ whenever we change $F$-classes. Since $x$ is not near any facial boundaries of $E$, this will imply that the two $F$-classes we pass between are contained in the same $E$-class. In order to pass through the facial boundaries of $F$ we are required to slightly perturb all other coordinates by at most $6 q a_i$. The vector $\T{v}$ was chosen so that after accounting for the cumulative effect of these perturbations from all $\ell$ stages we will still end up with a vector $\T{w}$ in $(1 - 2 \delta) \cdot B + \T{t}$ and thus $\phi(\T{w})\cdot x\in V$. The construction is completed inductively. For $0 \leq j \leq \ell$ consider the following property $P(j)$: there is $\T{w} \in \Z^\ell \times \Gamma$ with
\begin{enumerate}
\item[\rm (i)] $- 6 j q a_i \leq w_i - v_i \leq 6 j q a_i$ for $1 \leq i \leq j$;
\item[\rm (ii)] $-6 j q a_i \leq w_i - 0 \leq 6 j q a_i$ for $j < i \leq \ell$;
\item[\rm (iii)] $\phi(\T{w}) \cdot x \ E \ x$.
\end{enumerate}
By using $\T{w} = \T{0}$ we readily see that $P(0)$ is true. We will prove by induction that $P(j)$ is true for all $0 \leq j \leq \ell$. Then the vector $\T{w}$ witnessing the truth of $P(\ell)$ will have the property that
$$\T{w} \ \in \ \T{v} + 6 \ell q \cdot A \ \subseteq \ (1 - 2 \delta) \cdot B + \T{t}$$
(by clause (i)) and thus $\phi(\T{w}) \cdot x \in V$ but also $\phi(\T{w}) \cdot x \ E \ x$ (by clause (iii)), thus proving that $V \subseteq [x]_E$.

Assume $P(j-1)$ is true, and let $\T{w}$ witness the truth of $P(j-1)$. We will find $\T{w}'$ satisfying $P(j)$. We first make a small adjustment to $\T{w}$. Let $W$ be the $F$-class containing $\phi(\T{w}) \cdot x$. Note that by clauses (i) and (ii) of $P(j-1)$ we have $\T{w} \in (12 \ell + 2) q \cdot A \subseteq 15 \ell q \cdot A$. Thus $\phi(\T{w}) \cdot x \in X^\cH$ and so $W$ meets $X^\cH$. Using Lemma \ref{LEM BASE}, say that $W$ is $(\Phi, 2 \eta)$-roughly $C$ at $x$, where $A \sqsubseteq C$ and $2 \eta \cdot C \sqsubseteq \epsilon \cdot A$. So we have $\T{w} \in (1 + 2 \eta) \cdot C \subseteq (1+3\eta)\cdot C$. Since
$$12 \eta \cdot C \ \sqsubseteq \ 12 \epsilon \cdot A \ \sqsubseteq \ q \cdot A$$
$$\text{and} \quad 3 q \cdot A \ \sqsubseteq \ (1/2) \cdot A \ \sqsubseteq \ (1 - 3 \eta) \cdot A \ \sqsubseteq \ (1 - 3 \eta) \cdot C$$
we may apply Lemma \ref{LEM BASIC2}(ii) to obtain $\T{s} \in 4 q \cdot A \cap (\Z^\ell \times \{1_\Gamma\})$ such that
$$\T{w} + \T{s} + 3 q \cdot A \ \subseteq \ (1 - 3 \eta) \cdot C.$$
Fix such an $\T{s}$. This is our first small adjustment to $\T{w}$. Next we make a substantial change to the $j^\text{th}$-coordinate in order to achieve clause (i). Specifically, let
$$d \ = \ v_j-w_j-s_j \ \in \ \Z.$$ By clause (ii) of $P(j-1)$ and the bound on the size of $\T{v}$ observed in the first paragraph, we have
$$|d| \ \leq \ (6 \ell + 2 + 6(j-1) + 4) q a_j \ \leq \ 12\ell qa_j \ \leq \ \mbox{$\frac{1}{2}$} a_j.$$
Finally, we make one more small adjustment. Let $W'$ be the $F$-class containing $\phi(\T{w} + \T{s} + d \cdot \T{e}_j) \cdot x$. We again observe that $\T{w} + \T{s} + d \cdot \T{e}_j$ satisfies clauses (i) and (ii) of $P(j)$ and hence
$$\phi(\T{w} + \T{s} + d \T{e}_j) \cdot x \ \in \ \phi(15 \ell q \cdot A) \cdot x \ \subseteq \ X^\cH.$$
Again using Lemma \ref{LEM BASE}, say that $W'$ is $(\Phi, 2 \eta')$-roughly $C'$ at $x$, where $A \sqsubseteq C'$ and $2 \eta' \cdot C' \sqsubseteq \epsilon \cdot A$. As before, since $12 \eta' \cdot C' \sqsubseteq q \cdot A$ and $q \cdot A \sqsubseteq (1 - 3 \eta') \cdot C'$, by Lemma \ref{LEM BASIC2}(ii) we can find $\T{s}' \in 2 q \cdot A \cap (\Z^\ell \times \{1_\Gamma\})$ satisfying
$$\T{w} + \T{s} + d \cdot \T{e}_j + \T{s}' + q \cdot A \ \subseteq \ (1 - 3 \eta') \cdot C'.$$
We set $\T{w}' = \T{w} + \T{s} + d \T{e}_j + \T{s}'$. Clauses (i) and (ii) of $P(j)$ are immediately satisfied for coordinates $i \neq j$ since $\T{s} + \T{s}' \in 6 q \cdot A$. For the $j^\text{th}$ coordinate the definition of $d$ gives $w_j + s_j + d + s_j' = v_j + s_j'$. Since $\T{s}' \in 2 q \cdot A$, we see that clauses (i) and (ii) of $P(j)$ hold. We verify clause (iii) of $P(j)$ in the next paragraph.

Set $\T{r} = \T{w} + \T{s} + \T{s}'$ and $\T{r}' = \T{r} + d \T{e}_j = \T{w}'$. Our construction guarantees that $\phi(\T{r}) \cdot x$ is $E$ (in fact $F$) equivalent to $\phi(\T{w}) \cdot x$, and $\phi(\T{r}') \cdot x$ is $E$-equivalent (in fact equal) to $\phi(\T{w}') \cdot x$. By assumption we have $\phi(\T{w}) \cdot x \ E \ x$. So it suffices to prove that $\phi(\T{r}) \cdot x$ and $\phi(\T{r}') \cdot x$ are $E$-equivalent. Notice that by clauses (i) and (ii) of $P(j)$ and $P(j-1)$ we have
$$\T{r}, \T{r}' \in (12 \ell + 2) q \cdot A.$$
So if $\T{u}$ lies on the direct path between $\T{r}$ and $\T{r}'$ then $\T{u}$ is contained in the rectangle $(12 \ell + 2) q \cdot A$ by convexity and we have
$$x \ = \ \phi(\T{u})^{-1} \phi(\T{u}) \cdot x \ \in \ \phi((12 \ell + 2) q \cdot A + \rZ) \cdot \phi(\T{u}) \cdot x \ \subseteq \ \phi(15 \ell q \cdot A) \cdot \phi(\T{u}) \cdot x.$$
Our assumptions then imply that $\phi(\T{u}) \cdot x \in X^\cH \setminus \partial_j^\Phi(E, q \cdot A)$. We must make use of this fact carefully. A key observation is that $\T{r} + q \cdot A$ and $\T{r}' + q \cdot A$ are contained in $(1 - 3 \eta) \cdot C$ and $(1 - 3 \eta') \cdot C'$ respectively, and $\T{r}$ and $\T{r}'$ differ only in the $j^\text{th}$ coordinate and the difference is small compared to the scale of $F$-classes (the difference is at most $\frac{1}{2} a_j$ by an earlier computation, and the classes of $F$ are rough rectangles containing $A$). This implies that the direct path from $\T{r}$ to $\T{r}'$ is contained in $(1 + 2 \eta) \cdot C \cup (1 + 2 \eta') \cdot C'$. To see this in detail, let $\T{u}$ be any point on the direct path between $\T{r}$ and $\T{r}'$, let $U$ be the $F$-class containing $\phi(\T{u}) \cdot x$, and say $U$ is $(\Phi, 2 \theta)$-roughly $D$ at $x$, where $A \sqsubseteq D$ and $2 \theta \cdot D \sqsubseteq \epsilon \cdot A$. Then $\T{u} \in (1 + 2 \theta) \cdot D$, and (again by Lemma \ref{LEM BASIC2}(ii)) this implies that there is $\T{s}'' \in q \cdot A$ with $\T{u} + \T{s}'' \in (1 - 2 \theta) \cdot D$. Then we have
\begin{align*}
\T{r} + \T{s}'' & \in (1 - 2 \eta) \cdot C; \\
\T{u} + \T{s}'' & \in (1 - 2 \theta) \cdot D; \\
\T{r}' + \T{s}'' & \in (1 - 2 \eta') \cdot C'.
\end{align*}
As $\T{r}$ and $\T{r}'$ are close to one another and $\T{u}$ is in between them, we must have that $(1 - 2 \theta) \cdot D$ meets either $(1 - 2 \eta) \cdot C$ or $(1 - 2 \eta') \cdot C'$. So $U$ is equal to either $W$ or $W'$, and hence $\T{u}$ lies in $(1 + 2 \eta) \cdot C$ or $(1 + 2 \eta') \cdot C'$. So the direct path from $\T{r}$ to $\T{r}'$ is contained in $(1 + 2 \eta) \cdot C \cup (1 + 2 \eta') \cdot C$. Now if the $F$-classes $W$ and $W'$ are equal then we are done, and if they are unequal then $(1 - 2 \eta) \cdot C$ and $(1 - 2 \eta') \cdot C'$ are disjoint. In this case let $\T{u}$ be on the direct path between $\T{r}$ and $\T{r}'$ and lie in both $(1 + 3 \eta) \cdot C$ and $(1 + 3 \eta') \cdot C'$. Then of the two vectors $\T{u} \pm \lfloor q a_j \rfloor \T{e}_j$, one is contained in $(1 - 3 \eta) \cdot C$ and the other is contained in $(1 - 3 \eta') \cdot C'$ (this follows from clause (vi) of Lemma \ref{LEM BASIC} since translates of $12 \eta \cdot C$ and $12 \eta' \cdot C'$ are contained in $12 \epsilon \cdot A \subseteq q \cdot A$). The $q \cdot A$ neighborhoods of $\T{r}$ and $\T{r}'$ are contained in $(1 - 3 \eta) \cdot C$ and $(1 - 3 \eta') \cdot C'$, so it follows that of the two sets $\T{u} + \lfloor q a_j \rfloor \T{e}_j + q \cdot A^j$ and $\T{u} - \lfloor q a_j \rfloor \T{e}_j + q \cdot A^j$, one is contained in $(1 - 3 \eta) \cdot C$ and the other is contained in $(1 - 3 \eta') \cdot C'$. Now $\phi(\T{u}) \cdot x \not\in \partial_j^\Phi(E, q \cdot A)$ so there are $\T{p}, \T{p}' \in q \cdot A^j$ with
$$\phi(-\lfloor q a_j \rfloor \T{e}_j + \T{p}) \cdot \phi(\T{u}) \cdot x \ E \ \phi(\lfloor q a_j \rfloor \T{e}_j + \T{p}') \cdot \phi(\T{u}) \cdot x.$$
However
$$\phi(-\lfloor q a_j \rfloor \T{e}_j + \T{p}) \cdot \phi(\T{u}) \cdot x \ \in \ \phi(\T{u} - \lfloor q a_j \rfloor \T{e}_j + q \cdot A^j + \rZ) \cdot x,$$
$$\phi(\lfloor q a_j \rfloor \T{e}_j + \T{p}') \cdot \phi(\T{u}) \cdot x \ \in \ \phi(\T{u} + \lfloor q a_j \rfloor \T{e}_j + q \cdot A^j + \rZ) \cdot x,$$
and of the two sets $\T{u} - \lfloor q a_j \rfloor \T{e}_j + q \cdot A^j + \rZ$ and $\T{u} + \lfloor q a_j \rfloor \T{e}_j + q \cdot A^j + \rZ$, one is contained in $(1 - 2 \eta) \cdot C$ and the other is contained in $(1 - 2 \eta') \cdot C'$. Since $W$ and $W'$ are $(\Phi, 2 \eta)$-roughly $C$ and $(\Phi, 2 \eta')$-roughly $C'$ at $x$, respectively, it follows that there is a point in $W$ which is $E$-equivalent to a point in $W'$. Thus $W$ and $W'$ are contained in the same $E$-class and $\phi(\T{r}) \cdot x \ E \ \phi(\T{r}') \cdot x$.
\end{proof}

To conclude this section we present a technical counting lemma that will play a critical role in the next section. We remark that polynomial growth has been an important property in all previous results on the hyperfiniteness of orbit equivalence relations. This continues to be the case in the present paper, as can be seen in the lemma below where the bound on $t$ is independent of $A$. For the proof of this lemma it will be helpful to note that for a rectangle $A \subseteq \Z^\ell \times \Gamma$,
$$|A| \ = \ |\Gamma| \cdot \prod_{i = 1}^\ell \Big(2 \bdim_i(A) + 1 \Big).$$

\begin{lem} \label{LEM COUNT}
Let $G$ act on the set $X$, let $\Phi = (\ell, \phi, \rZ, \Gamma, \cH)$ be a chart for $G$, let $0 < \epsilon < 1/16$, let $q$ satisfy $6 \epsilon < q < 1/2$, and let $A \subseteq \Z^\ell\times\Gamma$ be a centered rectangle. Suppose $E$ is a $(\Phi, A, \epsilon)$-sub-rectangular equivalence relation on $X$. Let $x \in X^\cH$ and let $1 \leq i \leq \ell$. If $y_1, y_2, \ldots, y_t \in \partial_i(E, q \cdot A) \cap \phi(2^{17 \ell} \cdot A) \cdot x$ satisfy
$$\phi \Big( 2^{19 \ell} \cdot A^i + 5 q \cdot A \Big) \cdot y_r \ \cap \ \phi \Big( 2^{19 \ell} \cdot A^i + 5 q \cdot A \Big) \cdot y_s \ = \ \varnothing$$
for all $r \neq s$, then $t \leq 2^{22 \ell^2}$.
\end{lem}

\begin{proof}
We remark that $2^{22 \ell} \cdot A \subseteq \dom(\phi)$ by Definition \ref{DEFN RECTEQ}. Let $y_1, y_2, \ldots, y_t \in \partial_i(E, q \cdot A) \cap \phi(2^{17 \ell} \cdot A) \cdot x$ be as in the statement of the lemma, and let $F$ be a $(\Phi, A, \epsilon)$-rectangular sub-equivalence relation of $E$. We claim that each class of $F$ can contain at most two members of $y_1, y_2, \ldots, y_t$ and that furthermore any class of $F$ meeting $\{y_1, \ldots, y_t\}$ must meet $\phi((2^{17 \ell} + 2) \cdot A) \cdot x$ with cardinality at least $|(1/4) \cdot A|$. Assuming this claim, we have that $t$ is at most twice the number of $F$-classes meeting $\{y_s \: 1 \leq s \leq t\}$, and the number of such $F$ classes is at most
$$\begin{array}{lllll}
\displaystyle{\frac{\left|(2^{17\ell}+2)\cdot A\right|}{\left|\frac{1}{4}\cdot A\right|}} & = & \displaystyle{\prod_{j=1}^\ell\frac{2(2^{17\ell}+2)\bdim_j(A)+1}{2 \lfloor\frac{1}{4}\bdim_j(A)\rfloor+1}} & \leq & \displaystyle{\prod_{j=1}^\ell\frac{2(2^{17\ell}+3)\bdim_j(A)}{\frac{1}{4}\bdim_j(A)}} \vspace{0.05in} \\
& = & 8^\ell(2^{17\ell}+3)^\ell & \leq & 8^\ell(2\cdot 2^{17\ell})^\ell \vspace{0.05in} \\
& = & 16^\ell\cdot 2^{17\ell^2} & \leq & 2^{21\ell^2}.
\end{array}$$
Thus assuming the claim we have $t \leq 2 \cdot 2^{21\ell^2} \leq 2^{22\ell^2}$, so it suffices to prove the claim.

Let $U$ be an $F$-class meeting $\{y_s \: 1 \leq s \leq t\}$. Then $U$ must also meet both $X^\cH$ and $\phi(2^{17 \ell} \cdot A) \cdot x$. So $U$ must be $(\Phi, \delta)$-roughly $B$ for some rectangle $B$ and $\delta>0$, which we fix, satisfying $A \sqsubseteq B$, $2 \delta \cdot B \sqsubseteq \epsilon \cdot A$, and $2^{22\ell}\cdot B\subseteq\dom(\phi)$. Using Lemma \ref{LEM BASE}, fix $\T{u}\in \dom(\phi)$ such that $U$ is $(\Phi,2\delta)$-roughly $B+\T{u}$ at $x$. Then $(1 + 2 \delta) \cdot B + \T{u}$ meets $2^{17 \ell} \cdot A$, so clause (i) of Lemma \ref{LEM BASIC2} implies that $(1 - 2 \delta) \cdot B + \T{u}$ meets $(2^{17 \ell} + 8 \epsilon) \cdot A \subseteq (2^{17 \ell} + 1/2) \cdot A$. Since $(1/4) \cdot A \sqsubseteq (1 - 2 \delta) \cdot A \sqsubseteq (1 - 2 \delta) \cdot B$, we can find $\T{w}$ with
$$\T{w} + (1/4) \cdot A \ \subseteq \ \Big( (1 - 2 \delta) \cdot B + \T{u} \Big) \cap \Big( (2^{17 \ell} + 2) \cdot A \Big).$$
Then $\phi(\T{w} + (1/4) \cdot A) \cdot x$ is contained in both $U$ and $\phi((2^{17 \ell} + 2) \cdot A) \cdot x$ and therefore $|U \cap \phi((2^{17 \ell} + 2) \cdot A)| \geq |(1/4) \cdot A|$.

Now towards a contradiction suppose $U$ contains three elements of $\{y_s \: 1 \leq s \leq t\}$. With $\T{u}$ as above, set $z = \phi(\T{u}) \cdot x$ and let $\T{b} = \bdim(B)$. Note that $U$ is $(\Phi, \delta)$-roughly $B$ at $z \in X^\cH$ by Lemma \ref{LEM BASE}. Since $\partial_i^\Phi(E, q \cdot A) \subseteq \partial_i^\Phi(F, q \cdot A)$, by Lemma \ref{LEM FACES} there must be $j = \pm 1$ and $r \neq s$ with
$$y_r, y_s \ \in \ \phi \Big( j \cdot b_i \cdot \T{e}_i + B^i + 2 q \cdot A \Big) \cdot z.$$
Therefore
$$y_r \ \in \ \phi(B^i - B^i + 4 q \cdot A + \rZ) \cdot y_s \ \subseteq \ \phi(B^i - B^i + 5 q \cdot A) \cdot y_s$$
(we do not write $2 \cdot B^i$ because $B$ may not be centered; however, note that $B^i - B^i$ is centered). On the other hand, $y_r, y_s \in \phi(2^{17 \ell} \cdot A) \cdot x$ and thus $y_r \in \phi((2 \cdot 2^{17 \ell} + \epsilon) \cdot A) \cdot y_s$. Since $y_s \in \partial_i^\Phi (E, q \cdot A) \subseteq X^\cH$, the map $\T{v} \mapsto \phi(\T{v}) \cdot y_s$ is injective, so we obtain
$$\phi(\T{0}) \cdot y_r = y_r \in \phi \Big( (B^i - B^i + 5 q \cdot A) \cap (2 \cdot 2^{17 \ell} + \epsilon) \cdot A \Big) \cdot y_s \subseteq \phi \Big( 2^{19 \ell} \cdot A^i + 5 q \cdot A \Big) \cdot y_s.$$
This is a contradiction since $A$ is centered and hence $\T{0} \in 2^{19 \ell} \cdot A^i + 5 q \cdot A$.
\end{proof}

\section{Orthogonal equivalence relations} \label{SEC ORTHO}

Pairs of orthogonal equivalence relations were first introduced by Gao and Jackson in \cite{GJ}. This notion is of great importance in the study of hyperfiniteness of orbit equivalence relations and is the key ingredient in both the present paper and \cite{GJ}. The precise definition of orthogonality will necessarily vary from application to application (in particular our definition is not identical to that in \cite{GJ}) but the fundamental concept remains the same: under a suitable notion of boundary types, we require boundaries of the same type occurring in two equivalence relations to be separated by a sufficient distance. In our setting this concept takes the following precise form.

\begin{defn} \label{DEFN ORTHO}
Let $G$ act on the set $X$, let $\Phi = (\ell, \phi, \rZ, \Gamma, \cH)$ be a chart for $G$, and let $A \subseteq \Z^\ell \times \Gamma$ be a centered rectangle with $30 \ell \cdot A \subseteq \dom(\phi)$. Two equivalence relations $E$ and $F$ on $X$ are \emph{$(\Phi, A)$-orthogonal} if
$$\phi(30 \ell \cdot A) \cdot \partial_i(E, A) \ \cap \ \phi(30 \ell \cdot A) \cdot \partial_i(F, A) \ = \ \varnothing$$
for every $1 \leq i \leq \ell$.
\end{defn}

Note that if $E$ and $F$ are rectangular equivalence relations, then in general their boundaries must intersect; however, if additionally $E$ and $F$ are orthogonal, then only boundaries of distinct types can meet each other, in which case these boundaries will be perpendicular or orthogonal to each other.  It is for this reason that Gao and Jackson call such pairs of equivalence relations orthogonal.

Our next lemma illustrates how orthogonality can be used to obtain hyperfiniteness results, thus demonstrating one of the key ideas behind the proof of the main theorem.

\begin{lem} \label{LEM ORTHOSEQ}
Let $G$ act freely on the set $X$, let $\Phi = (\ell, \phi, \rZ, \Gamma, \bOne)$ be a chart for $G$, let $0 < \epsilon < 1/4$, let $q$ satisfy $12 \epsilon < q < 1 / (24 \ell)$, and let $A \subseteq \Z^\ell \times \Gamma$ be centered rectangle. Assume that $\phi(\rZ)$ generates $G$. If $E_1, E_2, \ldots$ is an infinite sequence of $(\Phi, A, \epsilon)$-sub-rectangular equivalence relations on $X$ which are pairwise $(\Phi, q \cdot A)$-orthogonal, then for every $x, y \in X$ we have
$$G \cdot x = G \cdot y \quad \Longrightarrow \quad (\exists m \in \N) \ (\forall n \geq m) \quad x \ E_n \ y.$$
\end{lem}

Notice that if the $E_n$'s are finite and Borel, then the conclusion of the Lemma implies that the $G$-orbit equivalence relation is hyperfinite.

\begin{proof}
Note that $X^\bOne = X$ since $G$ acts freely. Since $\phi(\rZ)$ generates $G$, it suffices to prove the claim when $y \in \phi(\rZ) \cdot x$. If $n$ satisfies $\neg\, x \ E_n \ y$ then by Lemma \ref{LEM STRONG BND} we have
$$x \ \in \ \phi \big( 30 \ell \cdot (q \cdot A) \big) \cdot \partial_i^\Phi(E_n, q \cdot A)$$
for some $1 \leq i \leq \ell$. Thus our orthogonality assumption implies that this can occur for at most $\ell$ positive integers $n$. Therefore $x \ E_n \ y$ for all but finitely many $n \in \N$.
\end{proof}

The previous lemma motivates the construction of a countably infinite family of pairwise orthogonal sub-rectangular equivalence relations. This is one of the most important constructions of the paper and is contained in the lemma below.

\begin{lem} \label{LEM MAIN}
Let $G \acts X$ be a Borel action of $G$ on the Polish space $X$. Let $\Phi = (\ell, \phi, \rZ, \Gamma, \cH)$ be a chart for $G$, let $A \subseteq \Z^\ell \times \Gamma$ be a centered rectangle, and let $\epsilon,q\in\R$ and $b\in\N$ be positive numbers satisfying:
$$2^{40 \ell} \cdot A \ \subseteq \ \dom(\phi);$$
$$2 \cdot 36^2 \cdot 2^{14 \ell} \cdot \rZ \ \sqsubseteq \ \epsilon \cdot A;$$
$$8 \epsilon \ < \ q \ < \ \frac{1}{4} \cdot 306^{-1} \ell^{-1} b^{-1} 2^{-22 \ell^2}.$$
Let $E_1, E_2, \ldots, E_s$ be Borel equivalence relations on $X$. Suppose that each $E_k$ is $(\Phi,A, \epsilon)$-sub-rectangular, and that for every $x \in X^\cH$ there are at most $b$-many integers $t \in \{1, 2, \ldots, s\}$ with $\phi(8 \cdot 2^{14 \ell} \cdot A) \cdot x \not\subseteq [x]_{E_t}$. Then there is a finite Borel $(\Phi, A, \epsilon)$-rectangular equivalence relation $F$ which is $(\Phi, q \cdot A)$-orthogonal to each $E_t$. Furthermore, if $X$ is zero-dimensional, $G$ acts continuously and freely, $\cH = \bOne$, and each $E_t$ is $G$-clopen, then $F$ can be chosen to be $G$-clopen as well.
\end{lem}

The proof of this lemma is long and technical, but we include plenty of discussion to aid the reader through it.

\begin{proof}
We assume that $X$ is zero-dimensional, $G$ acts continuously (but not necessarily freely), $X^\cH$ is clopen, and each $E_k$ is $G$-clopen. Our arguments easily transfer to the Borel setting by ignoring these topological considerations. Note that if $G$ acts freely and $\cH = \bOne$ then $X^\cH = X$ is indeed clopen. At one clearly identified paragraph in the proof we will temporarily assume that $G$ acts freely and $\cH = \bOne$ and show that $F$ is $G$-clopen.

Our goal is to build a finite Borel equivalence relation $F$ which is both $(\Phi, A, \epsilon)$-rectangular and $(\Phi, q\cdot A)$-orthogonal to $E_1, E_2, \ldots, E_s$. In order to build $F$, we will start with a collection of rough rectangles $\{R_y\}$ which cover $X^\cH$. This cover will then generate an equivalence relation $E$ via taking intersections of the $R_y$'s and their complements. The equivalence relation $E$ will not be rectangular as its classes will be rough polygons (with faces perpendicular to the coordinate axes $\T{e}_i$) instead of rough rectangles. However $E$ will in fact be a sub-rectangular equivalence relation and we will proceed to divide $E$-classes just as much as is necessary in order to obtain our rectangular sub-equivalence relation $F$.

The first step of the construction is to cover $X^\cH$ with a collection of rough rectangles. We want the equivalence relation $E$ generated from this cover to have classes that are of size $A$ or larger, when they meet $X^\cH$. We therefore want our collection of rough rectangles to cover $X^\cH$ but not be packed too densely or have parallel faces that are too close together. We achieve this by using a marker set. Set $p = 2^{14 \ell}$. Let $K$ be the symmetric closure of $\phi((3p / 4) \cdot A)$. Note that $K \cap \Stab(x) = \{1_G\}$ for all $x \in X^\cH$ and that $K \cdot x \subseteq \phi((3p / 4) \cdot A + \rZ) \cdot x \subseteq \phi(p \cdot A) \cdot x$ for $x \in X^\cH$. Therefore applying Lemma \ref{LEM MARKER} with $F = K$ gives a clopen marker set $Y \subseteq X^\cH$ satisfying:
$$(\forall x \in X^\cH) \ (\exists y \in Y) \ x \in \phi(p \cdot A) \cdot y;$$
$$(\forall y, y' \in Y) \ [\,y \in \phi ((3p/4) \cdot A) \cdot y' \ \Longrightarrow \ y = y'\,].$$
Our rough rectangles $R_y$ will be indexed by the marker points $y \in Y$. In fact, we will associate to each $y\in Y$ a centered rectangle $D_y\subseteq\Z^\ell\times\Gamma$ and then set $R_y=\phi(D_y)\cdot y$. Thus in order for the $R_y$'s to cover $X^\cH$ we will require $p \cdot A \subseteq D_y$, and in order for the $R_y$'s not to be packed too densely we will require $D_y \subseteq 2 p \cdot A$ (clause (i) below). In the end each $E$-class meeting $X^\cH$ must roughly contain $A$ and therefore we will require that parallel faces of $R_y$ and $R_{y'}$ be sufficiently far apart for $y \neq y'$ (clause (ii) below). Finally, in order to achieve orthogonality with $E_1, E_2, \ldots ,E_s$ we will require the faces of the $R_y$'s be sufficiently far away from the corresponding $q \cdot A$-boundaries of $E_1, E_2, \ldots, E_s$ (clause (iii) below). We set $\T{a} = \bdim(A)$. The specific conditions are as follows. To each $y \in Y$ we wish to associate, in a continuous fashion, a centered rectangle $D_y$ with radius vector $\T{d}(y) \in \Z^\ell \times \{1_\Gamma\}$ such that for every $y \in Y$ and $1 \leq i \leq \ell$,
\begin{enumerate}
\item[\rm (i)] $p \cdot A \subseteq D_y \subseteq 2 p \cdot A$;
\item[\rm (ii)] if $y' \in Y$ and $y' = \phi(\T{u}) \cdot y$ with $\T{0} \neq \T{u} \in 13 p \cdot A$, then the four numbers $\pm d_i(y), u_i\pm d_i(y') \in \Z$ are at least $3 a_i$ apart from one another;
\item[\rm (iii)] if $1 \leq k \leq s$, $z \in \partial_i^\Phi(E_k, q \cdot A)$, and $z = \phi(\T{u}) \cdot y$ with $\T{u} \in 7 p \cdot A$, then $u_i$ and $\pm d_i(y)$ are at least $64 \ell q \cdot a_i$ apart from one another.
\end{enumerate}
These conditions are in fact not very restrictive as the following claim points out.

\vspace{0.2in}
\noindent
\underline{Claim:} Suppose that $\T{d}$ is defined on a subset of $Y$ and satisfies clauses (i), (ii), and (iii). Then for any $y \in Y$ we can define $\T{d}(y)$ so that clauses (i), (ii), and (iii) continue to hold.

\vspace{0.1in}
\noindent
\textit{Proof of Claim:} Fix $y \in Y$ and fix $1 \leq i \leq \ell$. We will pick $d_i(y) \in [p a_i, 2 p a_i] \cap \Z$ to satisfy clauses (ii) and (iii). Notice that $\phi((p/4) \cdot A)^{-1} \phi((p/4) \cdot A) \subseteq \phi((3p/4) \cdot A)$ and therefore the sets $\phi((p/4) \cdot A) \cdot y'$, $y' \in Y$, are pairwise disjoint. Furthermore if $y' \in Y \cap \phi(13 p \cdot A) \cdot y$ then $\phi((p/4) \cdot A) \cdot y' \subseteq \phi(14 p \cdot A) \cdot y$. Therefore $\phi(13 p \cdot A) \cdot y$ can meet at most $N_1$ points of $Y$, where
$$\frac{\left|14p\cdot A\right|}{\left|\frac{1}{4}p\cdot A\right|} \ = \ \prod_{j=1}^\ell\frac{2(14\cdot 2^{14\ell}\cdot\bdim_i(A))+1}{2(2^{14\ell-2}\cdot\bdim_i(A))+1} \ \leq \ \prod_{j=1}^\ell\frac{3}{2}\left(\frac{14\cdot 2^{14\ell}}{2^{14\ell-2}}\right) \ = \ 84^\ell \ =: \ N_1.$$
Thus clause (ii) requires that $d_i(y)$ avoid at most $4 \cdot N_1$-many intervals of cardinality $6 a_i + 1 < 7 a_i$.

Now we consider clause (iii). Notice that if $\phi(7 p \cdot A) \cdot y \cap \partial_i^\Phi(E_r, q \cdot A) \neq \varnothing$ then $\phi(8 p \cdot A) \cdot y \not\subseteq [y]_{E_r}$. Let $J$ be the collection of integers $r \in \{1, 2, \ldots, s\}$ for which $\phi(8 p \cdot A) \cdot y \not\subseteq [y]_{E_r}$. By assumption we have $|J| \leq b$. Fix $r \in J$. Observe that if $z_1, z_2 \in \partial_i^\Phi(E_r, q \cdot A) \cap \phi(2^{17 \ell} \cdot A) \cdot y$ satisfy $z_1 = \phi(\T{u}) \cdot y$, $z_2 = \phi(\T{v}) \cdot y$, and
$$\phi(2^{19 \ell} \cdot A^i + 5 q \cdot A) \cdot z_1 \ \cap \ \phi(2^{19 \ell} \cdot A^i + 5 q \cdot A) \cdot z_2 \ \neq \ \varnothing,$$
then
$$\phi(2^{19 \ell} \cdot A^i + 6 q \cdot A + \T{u}) \cdot y \ \cap \ \phi(2^{19 \ell} \cdot A^i + 6 q \cdot A + \T{v}) \cdot y \ \neq \ \varnothing$$
and hence $|u_i - v_i| \leq 12 q a_i$ since $A$ is centered and the map $\T{w} \mapsto \phi(\T{w}) \cdot y$ is injective (as $y \in Y \subseteq X^\cH$). Thus Lemma \ref{LEM COUNT} and the fact that $7 p < 2^{17 \ell}$ implies that the set
$$\{ \pm u_i \: \T{u} \in 7 p \cdot A \wedge \phi(\T{u}) \cdot y \in \partial_i^\Phi(E_r, q \cdot A)\}$$
can be covered by fewer than $2 \cdot N_2$-many intervals of cardinality $24 q a_i + 1$, where
$$N_2 := 2^{22 \ell^2}.$$

Now letting $r \in J$ vary, clause (iii) requires that $d_i(y)$ avoid at most $2 b \cdot N_2$-many intervals of cardinality $152 \ell q a_i + 1 < 153 \ell q a_i$. From our assumptions on $b$, $q$, and $\epsilon$ we obtain
$$\begin{array}{ll}
  & \displaystyle{p a_i - 4 \cdot N_1 \cdot 7 a_i - 2 b \cdot N_2 \cdot 153 \ell q a_i} \\
= & \displaystyle{p a_i \cdot \left( 1 - 28 \cdot 84^\ell p^{-1} - 306 b \ell q \cdot 2^{22 \ell^2} p^{-1} \right)} \\
= & \displaystyle{p a_i \cdot \left(1 - 28 \left(\frac{84}{2^{14}} \right)^\ell - 306 b \ell q \cdot 2^{22 \ell^2-14\ell} \right)} \\
> & p a_i \cdot (1 - \frac{1}{6} - \frac{1}{4}) \ > \ \frac{1}{4} p a_i \ \geq \ 1.
\end{array}$$
Therefore there is a choice of $d_i(y) \in [p a_i, 2 p a_i] \cap \Z$ satisfying clauses (i), (ii), and (iii). \hfill $\blacksquare$
\vspace{0.1in}

With the above claim one can easily construct a (not necessarily Borel) function $\T{d}$ satisfying (i), (ii), and (iii). However we want $\T{d}$ to be continuous. To this end we apply Corollary \ref{COR PART}, using the symmetric closure of $\phi(13 p \cdot A)$ as $F$, to partition $Y$ into clopen sets $\{Y_1, Y_2, \ldots, Y_k\}$ with the property that for each $1 \leq j \leq k$ and each $y, y' \in Y_j$,
$$y' \in \phi(13 p \cdot A) \cdot y \ \Longrightarrow \ y' = y.$$
We will inductively define $\T{d}$ on the $Y_j$'s starting with $Y_1$. It is important to observe that if $y$ and $y'$ are as in clause (ii), then they must lie in distinct members of this partition. This means that for $y, y' \in Y_j$ the definitions of $\T{d}(y)$ and $\T{d}(y')$ will not conflict with one another. So in defining $\T{d}$ on $Y_1$ we only need to worry about clause (iii). For $1 \leq i \leq \ell$, $1 \leq r \leq s$, and $\T{u} \in 7 p \cdot A$, the set of $y \in Y_1$ with $\phi(\T{u}) \cdot y \in \partial_i^\Phi(E_r, q \cdot A)$ is clopen by Lemma \ref{LEM BNDCLOPEN} (recall that we are assuming that each $E_r$ is $G$-clopen). Thus for each $1 \leq i \leq \ell$, the set of potential values for $d_i(y)$ prohibited by clause (iii) varies continuously with $y \in Y_1$. So for $1 \leq i \leq \ell$ and $y \in Y_1$, we can define $d_i(y) \in [p a_i, 2 p a_i] \cap \Z$ to be the least value which is allowed by clause (iii). Such a value exists by the above claim, and the assignment thus defined is continuous. Now suppose inductively that $\T{d}$ has been defined on $Y_1 \cup \cdots \cup Y_{j-1}$ so that it satisfies (i), (ii), and (iii) and is continuous. We again observe that for $\T{u} \in 13 p \cdot A$ the set of $y \in Y_j$ with $\phi(\T{u}) \cdot y \in Y_1 \cup \cdots \cup Y_{j-1}$ is clopen since $G$ acts continuously. Furthermore, on this clopen set the values $\T{d}(\phi(\T{u}) \cdot y)$ vary continuously with $y$ by our inductive hypothesis. Thus for each $1 \leq i \leq \ell$, the set of potential values for $d_i(y)$ prohibited by clause (ii) varies continuously with $y \in Y_j$. And as previously argued, the set of potential values for $d_i(y)$ prohibited by clause (iii) varies continuously with $y \in Y_j$ as well. Therefore for $1 \leq i \leq \ell$ and $y \in Y_j$ we can define $d_i(y) \in [p a_i, 2 p a_i] \cap \Z$ to be the least value which is allowed by clauses (ii) and (iii). This assignment is well-defined by the claim above, and is continuous. By continuing in this manner we can define $\T{d}$ on all of $Y$ so that $\T{d}$ satisfies clauses (i), (ii), and (iii) and is continuous.

Now, as promised, for $y \in Y$ set
$$D_y := \Rect(\T{d}(y)) \quad \mbox{and} \quad R_y := \phi(D_y) \cdot y,$$ and also set $\theta := \epsilon / (36^2 p^2)$. Our assumption on $A$ and $\epsilon$ and the fact that $p\in\N$ gives $2 \cdot \rZ \sqsubseteq \theta p \cdot A = \theta \cdot (p \cdot A) \sqsubseteq \theta\cdot D_y$ and thus $R_y$ is $(\Phi, \theta)$-roughly $D_y$ at $y$. As indicated at the beginning of the proof, we let $E$ be the equivalence relation generated by the $R_y$'s. Specifically, define
$$x\mathrel{E}x' \ \Longleftrightarrow \ (\forall y\in Y)\,(x\in R_y\,\Longleftrightarrow\, x'\in R_y).$$

We now show that $E$ is $(\Phi, A, \epsilon)$-sub-rectangular by constructing the $(\Phi, A, \epsilon)$-rectangular sub-equivalence relation $F$ of $E$. In order to construct $F$, we proceed to divide the $E$-classes just as much as is necessary in order to obtain a rectangular sub-equivalence relation. The basic idea is that when $R_y$ meets $R_{y'}$ we wish to broaden the faces of $R_y$ 
so that they divide $R_{y'} \setminus R_y$ into pieces. A convenient way to do this is to use a collection of large rough-rectangles surrounding $R_y$. For $y \in Y$ consider the centered rectangle having radius vector $9 \T{d}(y) + \T{1}$. After removing $D_y$ from this rectangle, we partition the resulting region into $3^\ell$-many rectangles which are indexed by tuples $\alpha \in \{-1, 0, 1\}^\ell$. Specifically, for $y \in Y$ and $\alpha \in \{-1, 0, 1\}^\ell$, define $D_y^\alpha$ to be the set of all vectors $\T{b} \in \Z^\ell \times \Gamma$ such that for each $1 \leq i \leq \ell$,
\[
\begin{array}{rccclll}
-9d_i(y)-1 & \leq & b_i & \leq & -d_i(y)-1 & & \mbox{when }\alpha_i=-1; \vspace{0.05in} \\
-d_i(y)    & \leq & b_i & \leq & d_i(y)    & & \mbox{when }\alpha_i=0;  \vspace{0.05in} \\
d_i(y)+1   & \leq & b_i & \leq & 9d_i(y)+1 & & \mbox{when }\alpha_i=1.
\end{array}
\]
Notice that $D_y^{\T{0}} = D_y$ and that the $D_y^\alpha$'s partition the centered rectangle having radius vector $9 \T{d}(y) + \T{1}$. We use $9 \T{d}(y) + \T{1}$ and not $9 \T{d}(y)$ simply because we want each $D_y^\alpha$ to have a genuine center in $\Z^\ell$ (as required by our non-standard definition of rectangle). Define $R_y^\alpha = \phi(D_y^\alpha) \cdot y$. Then $R_y^{\T{0}} = R_y$ and $R_y^\alpha$ is $(\Phi, \theta)$-roughly $D_y^\alpha$ at $y$ since $2 \cdot \rZ \sqsubseteq \theta \cdot D_y \sqsubseteq \theta \cdot D_y^\alpha$. Now we can define the equivalence relation $F$. As a preliminary step, for any $x\in X$ define the set
$$\mathcal{S}(x)=\{\,y\in Y\: (\exists y'\in Y)\; x\in R_{y'}\,\wedge\,R_y\cap R_{y'}\ne\varnothing\,\}.$$
For $x \in X$ and $x' \in X^\cH$ we declare $x\mathrel{F}x'$ if and only if $\mathcal{S}(x)=\mathcal{S}(x') \neq \varnothing$ and for all $y\in\mathcal{S}(x)$ and $\alpha\in\{-1,0,1\}^\ell$,
$$x \in R_y^\alpha \ \Longleftrightarrow \ x' \in R_y^\alpha.$$
We let $F$ be the smallest equivalence relation on $X$ having the above property. Note that, as required, any $F$-class not meeting $X^\cH$ is a singleton. From the definition it is easy to check that $F$ is a sub-equivalence relation of $E$ (by using $\alpha = \T{0}$).

Before beginning a detailed study of $F$, we make a few useful observations. Note that if $y,y'\in Y$ and $R_y\cap R_{y'}\ne\varnothing$ then
\begin{equation}\label{EQ 3}
 R_{y'} \ \subseteq \ \phi(D_{y'}) \phi(D_{y'})^{-1} \phi(D_y) \cdot y.
\end{equation}
Since each $D_y\sqsubseteq 2p\cdot A$, from this it follows that if $R_y\cap R_{y'}\ne\varnothing$ then
\begin{equation}\label{EQ 4}
R_{y'} \ \subseteq \ \phi(6p\cdot A+2\cdot\rZ)\cdot y.
\end{equation}
Alternatively, one can use \eqref{EQ 3} and the relation $D_{y'} \sqsubseteq 2 \cdot D_y$ to obtain
\begin{equation}\label{EQ 5}
R_{y'} \ \subseteq \ \phi(6 \cdot D_y) \cdot y.
\end{equation}
So if $x \in X$ and $y \in \mathcal{S}(x)$, then there is $y' \in Y$ with $x \in R_{y'}$ and $R_{y'} \cap R_y \neq \varnothing$. So $y \in \mathcal{S}(x)$ implies $x \in \phi(7p \cdot A) \cdot y$. In particular, if $y \in \mathcal{S}(x)$ then there is a unique $\alpha \in \{-1, 0, 1\}^\ell$ with $x \in R_y^\alpha$. This implies that every class $U$ of $F$ is the intersection of some finite subfamily of the $R_y^\alpha$'s. Furthermore, this means that the boundaries of an $F$-class $U$ come from nearby boundaries of the $R_y$'s. Thus from clauses (ii) and (iii) it will be possible to show that $F$ is rectangular and orthogonal to the pre-existing equivalence relations $E_1,\ldots, E_s$.

In this paragraph we temporarily assume that $G$ acts freely and $\cH = \bOne$, and we will show that $F$ is $G$-clopen as required. Note that under these assumptions we have $X^\cH = X$. For $g \in G$ and $\alpha \in \{-1, 0, 1\}^\ell$ set
$$Y_g^\alpha = \{y \in Y \: g \cdot y \in R_y^\alpha\}.$$
It follows from the continuity of $\T{d}$, the freeness of the action, and the definition of the $R_y^\alpha$'s that $Y_g^\alpha$ is clopen for every $g \in G$ and $\alpha \in \{-1, 0, 1\}^\ell$. When $\alpha = \T{0}$ we write $Y_g$ for $Y_g^\alpha$. For $x \in X$ set
$$C(x) \ = \ \{k \in \phi(7 p \cdot A) \: \exists y \in \mathcal{S}(x) \quad k \cdot y = x\}.$$
By the previous paragraph, for every $y \in \mathcal{S}(x)$ there is $k \in C(x)$ with $k \cdot y = x$. If $y \in \mathcal{S}(x)$ then there are $a, b, c \in G$ and $y' \in Y$ with $c \in \phi(D_y)$ and $a, b \in \phi(D_{y'})$ such that $c \cdot y = b \cdot y'$ and $a \cdot y' = x$. In this situation we have $y \in Y_c$, $y' \in Y_a \cap Y_b$, and $a b^{-1} c \cdot y = x$. Thus we see that $k \in C(x)$ if and only if
$$x \ \in \ \bigcup_{a,b,c}a \cdot \Big( Y_{a} \cap Y_b \cap b^{-1} c \cdot Y_c \Big)$$
where the union is over the set of all $a, b, c \in \phi(2 p \cdot A)$ with $k = a b^{-1} c$. This set is clopen, and since in all cases $C(x) \subseteq \phi(7 p \cdot A)$, we deduce that the set
$$X_K = \{x \in X \: C(x) = K\}$$
is clopen for any $K \subseteq \phi(7 p \cdot A)$. Notice that for $g \in G$ we have $\mathcal{S}(x) = \mathcal{S}(g \cdot x)$ if and only if $g C(x) = C(g \cdot x)$ (this follows from the action being free). Observe that if $k \cdot y = x$ then $x \in R_y^\alpha$ if and only if $k^{-1} \cdot x = y \in Y_k^\alpha$. Similarly, if $k \cdot y = x$ then $g \cdot x \in R_y^\alpha$ if and only if $k^{-1} \cdot x = y \in Y_{gk}^\alpha$. Therefore $x \ F \ g \cdot x$ if and only if there is a set $K \subseteq \phi(7 p \cdot A)$ such that $x \in X_K \cap g^{-1} \cdot X_{g K}$ and
$$x \in \bigcap_{k \in K} \bigcap_{\alpha \in \{-1, 0, 1\}^\ell} k \Big(Y_k^\alpha \cap Y_{g k}^\alpha \Big) \cup k \cdot \Big( (X \setminus Y_k^\alpha) \cap (X \setminus Y_{g k}^\alpha) \Big).$$
Thus $F$ is $G$-clopen as claimed.

It only remains to verify that $F$ is $(\Phi, A, \epsilon)$-rectangular and $(\Phi, q \cdot A)$-orthogonal to $E_1, E_2, \ldots, E_s$. It will be convenient to check these properties of $F$ one class at a time. So fix a class $U$ of $F$ meeting $X^\cH$. We ultimately want to use properties (ii) and (iii) of the function $\T{d}$, but first we must obtain a nice description of $U$. Fix $x \in U \cap X^\cH$. As remarked just after \eqref{EQ 5}, $U$ can be expressed as a finite intersection
$$U = \bigcap_{k = 1}^n R_{y(k)}^{\alpha(k)}$$
where $n \in \N$, $y(k) \in \mathcal{S}(x)\subseteq Y$ and $\alpha(k) \in \{-1, 0, 1\}^\ell$ for each $1 \leq k \leq n$. We may assume that if $y \in Y$ and $x \in R_y$ then there is $1 \leq k \leq n$ with $y(k) = y$ and $\alpha(k) = \T{0}$. For $1 \leq k \leq n$ let $\T{u}(k) \in \Z^\ell \times \Gamma$ be such that $\phi(\T{u}(k)) \cdot x = y(k)$, and note that each $\T{u}(k)\in 9p\cdot A$. By Lemma \ref{LEM SHIFT} each $R_{y(k)}^{\alpha(k)}$ is $(\Phi, 2 \theta)$-roughly $D_{y(k)}^{\alpha(k)} + \T{u}(k)$ at $x$. So in order to better understand $U$, we first consider the related intersection
$$\bigcap_{k = 1}^n \left( D_{y(k)}^{\alpha(k)} + \T{u}(k) \right).$$
For each $1 \leq i \leq \ell$, we may find
$$\mu_i, \nu_i \ \in \ \bigcup_{k = 1}^n \Big( \{u_i(k) \pm d_i(k)\} \cup \{u_i(k) \pm (d_i(k) + 1)\} \cup \{u_i(k) \pm (9 d_i(k) + 1)\} \Big)$$
so that
$$\bigcap_{k = 1}^n \Big( D_{y(k)}^{\alpha(k)} + \T{u}(k) \Big) \ = \ \{\T{b} \in \Z^\ell \times \Gamma \: \mu_i \leq b_i \leq \nu_i \text{ for all } 1 \leq i \leq \ell \}.$$
Recall that we work with a restricted definition of rectangle in which rectangles must have a genuine center in $\Z^\ell \times \Gamma$. Let $B \subseteq \Z^\ell \times \Gamma$ be the largest rectangle with
$$B \subseteq \bigcap_{k = 1}^n \Big( D_{y(k)}^{\alpha(k)} + \T{u}(k) \Big),$$
chosen for definiteness so that its center is furthest from the origin. Notice that $\bdim_i(B) = \lfloor (\nu_i - \mu_i) / 2 \rfloor$ for each $1 \leq i \leq \ell$. Our goal will be to show that $A\sqsubseteq B$ and $U$ is roughly $B$ at $x$.

First we will show that $A \sqsubseteq B$. Let $1\leq m\leq n$ be arbitrary, and fix $t$ such that $x \in R_{y(t)}$, $\alpha(t) = \T{0}$, and $R_{y(t)} \cap R_{y(m)} \neq \varnothing$. By \eqref{EQ 5} above we have
$$R_{y(t)} \ = \ \phi(D_{y(t)}) \cdot \phi(\T{u}(t)) \cdot x \ \subseteq \ \phi(6 \cdot D_{y(m)}) \cdot \phi(\T{u}(m)) \cdot x.$$
After using the equation $\phi(\T{u}(t)) \cdot x = y(t)$, we obtain
$$\begin{array}{lll}
\phi(D_{y(t)}) \cdot y(t) & \subseteq & \phi(6\cdot D_{y(m)})\,\phi(\T{u}(m))\,\phi(\T{u}(t))^{-1} \cdot y(t) \\
               & \subseteq & \phi(6\cdot D_{y(m)}+\T{u}(m)-\T{u}(t)+2\cdot\rZ) \cdot y(t),
\end{array}$$
and therefore $D_{y(t)} \subseteq 7 \cdot D_{y(m)} + \T{u}(m) - \T{u}(t)$ since $y(t) \in X^\cH$. This implies that
$$\bigcap_{k = 1}^n \Big( D_{y(k)}^{\alpha(k)} + \T{u}(k) \Big) \ \subseteq \ D_{y(t)} + \T{u}(t) \ \subseteq \ 7 \cdot D_{y(m)} + \T{u}(m),$$
so for every $1 \leq i \leq \ell$ we must have that $\mu_i, \nu_i \neq u_i(m) \pm (9 d_i(m) + 1)$. As this holds for every $1 \leq m \leq n$, we conclude that in fact
$$\mu_i, \nu_i \ \in \ \bigcup_{k = 1}^n \Big( \{u_i(k) \pm d_i(k)\} \cup \{u_i(k) \pm (d_i(k) + 1)\} \Big).$$
Now, we also have
$$x \ \in \ \phi(D_{y(t)}) \cdot y(t) \ \subseteq \ \phi(6 p \cdot A + 2 \cdot \rZ) \cdot y(m)$$
by \eqref{EQ 4} above. In particular, $x\in \phi(6 p \cdot A + 2 \cdot \rZ) \cdot y(m)$ for every $m$, so if $m \neq k$ then
$$\phi(6 p \cdot A + 2 \cdot \rZ) \cdot y(k) \ \cap \ \phi(6 p \cdot A + 2 \cdot \rZ) \cdot y(m) \ \neq \ \varnothing$$
and hence
$$y(k) \ \in \ \phi(12 p \cdot A + 5 \cdot \rZ) \cdot y(m) \ \subseteq \ \phi(13 p \cdot A) \cdot y(m).$$
But we also have $y(k) \in \phi(\T{u}(k) - \T{u}(m) + \rZ) \cdot y(m)$, so since $\rZ \subseteq \epsilon \cdot A$ clause (ii) implies that the numbers $u_i(k) \pm d_i(k)$ and $u_i(m) \pm d_i(m)$ are separated by at least $3 a_i - \epsilon a_i > 2 a_i + 2$ for every $1 \leq i \leq \ell$. This holds for all $m \neq k$, and thus we deduce that $|\mu_i - \nu_i| \geq 2 a_i$ for all $1 \leq i \leq \ell$. Therefore $\bdim_i(B) \geq a_i$ for each $1 \leq i \leq \ell$, and hence $A \sqsubseteq B$ as claimed.

Now we show that $U$ is $(\Phi, \delta)$-roughly $B$ at $x$, where $\delta := \epsilon / (18 p)$, and moreover that $F$ is $(\Phi, A, \epsilon)$-rectangular. Note that since $B \sqsupseteq A$, our assumption on $A$ and $\epsilon$ gives
$$\delta \cdot B \ \sqsupseteq \ \frac{\epsilon}{18 p} \cdot A \ \sqsupseteq \ 2 \cdot \rZ.$$
So $U$ will indeed be $(\Phi, \delta)$-roughly $B$ at $x$ once we show that $\phi((1 - \delta) \cdot B) \cdot x \subseteq U \subseteq \phi((1 + \delta) \cdot B) \cdot x$. Since $U=\bigcap_{k=1}^nR^{\alpha(k)}_{y(k)}$ and each $R_{y(k)}^{\alpha(k)}$ is $(\Phi, 2 \theta)$-roughly $D_{y(k)}^{\alpha(k)} + \T{u}(k)$ at $x$, we have
$$\bigcap_{k=1}^n\left(\phi\left((1-2\theta)\cdot D^{\alpha(k)}_{y(k)}+\T{u}(k)\right)\cdot x\right) \ \subseteq \ U \ \subseteq \ \bigcap_{k=1}^n\left(\phi\left((1+2\theta)\cdot D^{\alpha(k)}_{y(k)}+\T{u}(k)\right)\cdot x\right).$$
As the map $\T{w} \mapsto \phi(\T{w}) \cdot x$ is injective, we may rearrange intersections to obtain
$$\phi \left( \bigcap_{k = 1}^n \Big( (1 - 2 \theta) \cdot D_{y(k)}^{\alpha(k)} + \T{u}(k) \Big) \right) \cdot x \ \subseteq \ U \ \subseteq \ \phi \left( \bigcap_{k = 1}^n \Big( (1 + 2 \theta) \cdot D_{y(k)}^{\alpha(k)} + \T{u}(k) \Big) \right) \cdot x.$$
For each $1 \leq k \leq n$ we have
$$\frac{\delta}{2} \cdot B \ \sqsupseteq \ \frac{\epsilon}{36p} \cdot A \ = \ (2 \theta \cdot 9 \cdot 2 p) \cdot A \ = \ 2 \theta \cdot (9 \cdot 2 p \cdot A) \ \sqsupseteq \ 2 \theta \cdot D_{y(k)}^{\alpha(k)} \ \sqsupseteq \ \rZ \ \sqsupseteq \ \Rect(\T{1}).$$
Thus $\delta \cdot B \sqsupseteq 2 \theta \cdot D_{y(k)}^{\alpha(k)} + \Rect(\T{1})$, and by definition of $B$,
$$B \ \subseteq \ \bigcap_{k = 1}^n \Big( D_{y(k)}^{\alpha(k)} + \T{u}(k) \Big) \ \subseteq \ B + \Rect(\T{1}),$$
and therefore we deduce that
$$\begin{array}{rcll}
(1 - \delta) \cdot B & \subseteq & \displaystyle{\bigcap_{k = 1}^n \Big( (1 - 2 \theta) \cdot D_{y(k)}^{\alpha(k)} + \T{u}(k) \Big)}  & \text{and} \\
(1 + \delta) \cdot B & \supseteq & \displaystyle{\bigcap_{k = 1}^n \Big( (1 + 2 \theta) \cdot D_{y(k)}^{\alpha(k)} + \T{u}(k) \Big).} &
\end{array}$$
It follows that
$$\phi((1 - \delta) \cdot B) \cdot x \ \subseteq \ U \ \subseteq \ \phi((1 + \delta) \cdot B) \cdot x,$$
as required. So $U$ is $(\Phi, \delta)$-roughly $B$ at $x$. We already verified that $A \sqsubseteq B$, and since $B \sqsubseteq 2p \cdot A$ we see that $2 \delta \cdot B \sqsubseteq \epsilon \cdot A$. So in order for $F$ to be rectangular it only remains to check that $2^{22 \ell} \cdot B \subseteq \dom(\phi)$. For any $1 \leq k \leq n$ with $\alpha(k) = 0$ we have $D_{y(k)} \subseteq 2 p \cdot A$ and $\T{u}(k) \in 9p \cdot A$ and thus
$$B \subseteq D_{y(k)} + \T{u}(k) \ \subseteq \ 2p \cdot A + 9p \cdot A \ \subseteq \ 11p \cdot A.$$
Therefore
$$2^{22 \ell} \cdot B \ \subseteq \ 2^{22\ell}\cdot (11p\cdot A) \ \subseteq \ (2^{22\ell}\cdot 11\cdot 2^{14\ell})\cdot A \ \subseteq \ 2^{40\ell}\cdot A \ \subseteq \ \dom(\phi).$$
We conclude that $F$ is $(\Phi, A, \epsilon)$-rectangular.

The orthogonality condition is now all that remains to check. We continue working with $U$, $x$, and $B$ as above. Fix $1 \leq r \leq s$, and towards a contradiction suppose there exist $1 \leq i \leq \ell$, $z \in \partial_i^\Phi(F, q \cdot A) \cap U$, and $z' \in \partial_i^\Phi(E_r, q \cdot A)$ with
$$\phi(30 \ell q \cdot A) \cdot z \ \cap \ \phi(30 \ell q \cdot A) \cdot z' \ \neq \ \varnothing.$$
Let $\T{c}$ be the center of $B$, and set $\T{b} = \bdim(B)$. Since $z \in U$, there is $\T{w} \in (1 + \delta) \cdot B$ such that $z = \phi(\T{w}) \cdot x$. By Lemma \ref{LEM FACES} there is $j = \pm 1$ such that $w_i$ is within $2 q a_i$ of $j b_i + c_i$. By definition of $B$, $\mu_i$, and $\nu_i$, we have that $j b_i + c_i$ equals either $\mu_i$, $\mu_i + 1$, $\nu_i$, or $\nu_i - 1$. Therefore we can find $1 \leq k \leq n$ such that $j b_i + c_i$ is within a distance of two of either $u_i(k) + d_i(y(k))$ or $u_i(k) - d_i(y(k))$. Hence there is $j' = \pm 1$ such that $w_i$ is within $3 q a_i$ of $u_i(k) + j' d_i(y(k))$. Now consider $\T{v}$ such that $\phi(\T{v}) \cdot y(k) = z'$. We have that $z'$ belongs to
$$\phi \Big( 30 \ell q \cdot A \Big)^{-1} \phi \Big( 30 \ell q \cdot A \Big) \phi \Big( \T{w} \Big) \phi \Big( \T{u}(k) \Big)^{-1} \cdot y(k)$$
and therefore
$$\T{v} \ \in \ 60 \ell q \cdot A + \T{w} - \T{u}(k) + 3 \cdot \rZ.$$
It follows that $v_i$ is within $64 \ell q a_i$ of $j' d_i(y(k))$. In order for this to contradict clause (iii), we must show that $z' \in \phi(7p \cdot A) \cdot y(k)$. This fact is quickly obtainable as there is $y \in Y$ with $z \in R_y$ and $R_y \cap R_{y(k)} \neq \varnothing$, and hence by \eqref{EQ 4}
$$\begin{array}{lll}
z' & \in       & \phi(30 \ell q \cdot A)^{-1} \phi(30 \ell q \cdot A) \cdot z \vspace{.03in} \\
   & \subseteq & \phi(30 \ell q \cdot A)^{-1} \phi(30 \ell q \cdot A) \phi(6p \cdot A + 2 \cdot \rZ) \cdot y(k) \vspace{.03in} \\
   & \subseteq & \phi(60 \ell q \cdot A + 6 p \cdot A + 4 \cdot \rZ) \cdot y(k) \vspace{.03in} \\
   & \subseteq & \phi(7p \cdot A) \cdot y(k).
\end{array}$$
We conclude that $F$ is $(\Phi, q \cdot A)$-orthogonal to each of the equivalence relations $E_1, E_2, \ldots, E_s$.
\end{proof}

\section{Stabilizers and the subgroup conjugacy relation} \label{SEC STAB}

Before proceeding to the proof of the main theorem in the next section, we discuss here some issues surrounding non-free actions and the subgroup conjugacy relation. The fact that it is often easier to work with free actions than it is to work with general ones has become something of a recurring theme in the project of determining which countable groups have hyperfinite orbit equivalence relations. In proving Thereom \ref{THM FREE} in the next section, where we assume the action to be free, we will be able to structure our proof in a manner that closely resembles the arguments of Gao-Jackson in \cite{GJ}. However, the proof of the general case (Theorem \ref{THM NFREE}) appears to require new techniques due to an obstacle that appears for the first time in considering nilpotent groups: namely, the complexity of the subgroup conjugacy relation. In this section we will briefly discuss subgroup conjugacy and its impact on the relationship between the free-action and general cases. We remark that, despite their short proofs, we do not know if either of Propositions \ref{LEM STABRED} or \ref{THM NONSMOOTH} has previously appeared in the literature.

For a countable group $G$, let $\Sub(G)$ denote the space of subgroups of $G$ with the relative topology as a subset of the product space $2^G = \{0,1\}^G$. Then $\Sub(G)$ is closed in $2^G$, and hence is itself a Polish space. The subgroups $H, L \in \Sub(G)$ are \emph{conjugate} if there is $g \in G$ with $g H g^{-1} = L$. Conjugacy is a Borel equivalence relation on $\Sub(G)$, called the \emph{subgroup conjugacy relation} of $G$. If $G$ acts in a Borel fashion on the standard Borel space $X$, then the stabilizer map $x\mapsto\Stab(x)\in\Sub(G)$ is a Borel (though in general not continuous) function taking orbit-equivalent points to conjugate subgroups. In some cases smoothness of the subgroup conjugacy relation for a class of groups reduces the hyperfiniteness question for their orbit equivalence relations to a consideration of just their \emph{free} actions.

\begin{prop} \label{LEM STABRED}
Let $\mathcal{G}$ be a class of countable groups and $\mathcal{E}$ a class of countable Borel equivalence relations such that:
\begin{enumerate}
 \item[(i)] $\mathcal{G}$ is closed under subgroups and quotients;
 \item[(ii)] any group in $\mathcal{G}$ has at most countably many subgroups; and
 \item[(iii)] $\mathcal{E}$ is closed under Borel reducibility and countable disjoint unions.
\end{enumerate}
If $E^X_G\in\mathcal{E}$ for every free Borel action $G \acts X$ with $G \in \mathcal{G}$, then $E^X_G\in\mathcal{E}$ for every Borel action $G \acts X$ with $G \in \mathcal{G}$.
\end{prop}

\begin{proof}
Let $\mathcal{G}$ and $\mathcal{E}$ satisfy the hypotheses of the lemma and fix $G\in\mathcal{G}$ and a Borel action of $G$ on some standard Borel space $X$. Write $E=E^X_G$. Since $\Sub(G)$ is countable, the subgroup conjugacy relation is smooth. For each conjugacy class $\mathcal{C}\subseteq\Sub(G)$, fix $H_{\mathcal{C}}\in\mathcal{C}$. Let $Y_{\mathcal{C}}$ be the set of $x\in X$ whose stabilizer belongs to $\mathcal{C}$, so that $X=\bigsqcup_\mathcal{C} Y_{\mathcal{C}}$ is a decomposition of $X$ into countably many $E$-invariant Borel sets. By (iii), we need only show that $E\res Y_\mathcal{C}$ belongs to $\mathcal{E}$ for every conjugacy class $\mathcal{C}$ in $\Sub(G)$. So fix a conjugacy class $\mathcal{C}$ and consider $E\res Y_\mathcal{C}$. Write $H=H_\mathcal{C}$, and let $Y_H\subseteq Y_\mathcal{C}$ be the set of $x\in X$ whose stabilizer equals $H$. Note that $Y_H$ is a Borel set meeting each $E$-class in $Y_\mathcal{C}$. Let $N$ be the normalizer of $H$ in $G$. Then $N \in \mathcal{G}$ and $N/H \in \mathcal{G}$ by (i). Furthermore, $N/H$ acts freely on $Y_H$ and $E\res Y_H =E^{Y_H}_{N/H}$. Thus by our hypotheses, $E\res Y_H\in\mathcal{E}$. Now for each $x\in Y_{\mathcal{C}}$ let $f(x)=g\cdot x$ where $g$ is least (in some fixed well-ordering of $G$) such that $g \cdot x \in Y_H$. Then $f$ is a Borel reduction from $E\res Y_\mathcal{C}$ to $E\res Y_H$. Hence $E\res Y_\mathcal{C}\in\mathcal{E}$ by (iii). \end{proof}

We mention an immediate but interesting corollary of this lemma. Recall that a group is \emph{polycyclic} if it admits a cyclic series, or equivalently if it is solvable and every subgroup is finitely generated.

\begin{cor}
If all free Borel actions of polycyclic groups have hyperfinite orbit equivalence relations, then all Borel actions of polycyclic groups have hyperfinite orbit equivalence relations. \label{COR POLYCYCLIC}
\end{cor}

The argument given in the proof of Proposition \ref{LEM STABRED} runs into at least two obstacles if there are groups in $\mathcal{G}$ with uncountably many subgroups. First, if $\Sub(G)$ has uncountably many conjugacy classes then hyperfiniteness of $E^X_G$ does not follow immediately from the hyperfiniteness of each $E^X_G\res Y_\mathcal{C}$, since the decomposition $X=\bigsqcup Y_\mathcal{C}$ has uncountably many pieces. Nevertheless, if one has a sufficiently constructive way of showing that $E^X_G\in\mathcal{E}$ for every free Borel action $G\acts X$ with $G\in\mathcal{G}$, then one may be able to piece together the uncountably many individual reductions in a Borel manner so as to obtain a global reduction. (This is precisely what Gao and Jackson do in Section 7 of \cite{GJ}). However, in order even to do this, one must still be able choose in a Borel manner a distinguished representative from each conjugacy class of subgroups. If the subgroup conjugacy relation is non-smooth, then there is no way to do this. Unfortunately this is already the situation for nilpotent groups.

\begin{prop}
There exists a countable nilpotent group $G$ of nilpotency class $2$ such that the subgroup conjugacy relation of $G$ is non-smooth. \label{THM NONSMOOTH}
\end{prop}

\noindent By our main theorem, of course, the subgroup conjugacy relation of $G$ is hyperfinite whenever $G$ is countable and locally nilpotent.

\begin{proof}
For each $i \in \N$ let
$$\Gamma_i = \langle a_i, b_i, c_i \ \ | \ \ [a_i, b_i] = c_i, [a_i, c_i] = 1_{\Gamma_i}, [b_i, c_i] = 1_{\Gamma_i} \rangle$$
be a copy of the discrete Heisenburg group. Note that $\Gamma_i$ is nilpotent of class $2$ since $\langle c_i \rangle \cong \Z$ is central and $\Gamma_i / \langle c_i \rangle \cong \Z \times \Z$ is abelian. Note also that $\langle a_i^{-1} c_i \rangle \neq \langle a_i^{-1} \rangle$ and that $b_i a_i^{-1} b_i^{-1} = a_i^{-1} c_i$.

Set
$$G = \bigoplus_{i \in \N} \Gamma_i.$$
Then $G$ is a countable nilpotent group of nilpotency class $2$. For each $i \in \N$ let $\pi_i : G \rightarrow \Gamma_i$ be the natural projection. We will embed $E_0$ into the subgroup conjugacy relation of $G$. For $x \in 2^\N$ set
$$H_x = \langle a_i^{-1} c_i^{x(i)} : i \in \N \rangle.$$
The map $x \mapsto H_x$ is Borel, in fact continuous. First suppose that $x \ E_0 \ y$. Set
$$g = b_0^{y(0) - x(0)} \cdot b_1^{y(1) - x(1)} \cdots b_i^{y(i) - x(i)} \cdots.$$
Then $g \in G$ since $x \ E_0 \ y$, and we have $g H_x g^{-1} = H_y$. On the other hand, suppose that $\neg\, x \ E_0 \ y$. Fix any $g \in G$. Then for all but finitely many $i \in \N$ we have $\pi_i(g) = 1_{\Gamma_i}$. So there is $i \in \N$ with $x(i) \neq y(i)$ and $\pi_i(g) = 1_{\Gamma_i}$ and hence
$$\pi_i(g H_x g^{-1}) = \pi_i(g) \pi_i(H_x) \pi_i(g^{-1}) = \langle a_i^{-1} c_i^{x(i)} \rangle \neq \langle a_i^{-1} c_i^{y(i)} \rangle = \pi_i(H_y).$$
Thus $H_x$ and $H_y$ are not conjugate.
\end{proof}

Proposition \ref{THM NONSMOOTH} points to a new difficulty in establishing the hyperfiniteness of orbit equivalence relations arising from non-free actions. For those classes of groups previously known to have only hyperfininte orbit equivalence relations, the subgroup conjugacy relation is smooth. For finitely-generated nilpotent-by-finite groups, this is due to the fact that such groups have only countably many subgroups. Although Jackson--Kechris--Louveau did not use this fact in \cite{JKL}, Proposition \ref{LEM STABRED} immediately shows that in order to obtain hyperfiniteness of orbit equivalence relations arising from general Borel actions of finitely-generated nilpotent-by-finite groups, it suffices to consider only free Borel actions of such groups. On the other hand countable abelian groups might have uncountably many subgroups, but the subgroup conjugacy relation for such groups is trivial, in particular smooth, and Gao-Jackson \cite{GJ} used the method described in the paragraph following Corollary \ref{COR POLYCYCLIC} in a critical way.

In the final stages of their proof, Gao and Jackson performed a multi-scale inductive construction of a two-dimensional array of equivalence relations. The scales they used in this construction were linearly ordered and grew extremely fast. Consequently, perturbations at a given scale became unnoticeable at larger ones, allowing them to work at certain scales without disturbing others. We will perform a very similar construction in the proof of Theorem \ref{THM FREE}, where we work with free actions. We consider general actions in Theorem \ref{THM NFREE}, and in order to overcome the non-smoothness of the subgroup conjugacy relation we perform a multi-scale inductive construction of a three-dimensional rather than a two-dimensional array of equivalence relations. As a result we will lose the linear order on the scales, and any adjustments we make within the scope of one scale will significantly alter what can be detected at other scales. In this sense the proof of Theorem \ref{THM NFREE} is more delicate than the proof of Theorem \ref{THM FREE} and departs more significantly from the constructions in \cite{GJ}.

\section{Borel actions of countable locally nilpotent groups} \label{SEC FINAL}

We now have all of the tools we need to prove that every Borel action of a countable locally nilpotent group gives rise to a hyperfinite orbit equivalence relation. As a warm up, we first prove this result under the assumption that the action is free. The proof is simpler under this assumption and closely resembles the Gao--Jackson proof \cite{GJ} for actions of abelian groups.

\begin{thm} \label{THM FREE}
If the countable locally nilpotent group $G$ acts freely and continuously on the zero-dimensional Polish space $X$, then the induced orbit equivalence relation $E^X_G$ is continuously embeddable into $E_0$.
\end{thm}

\begin{proof}
Let $G_1 \leq G_2 \leq \cdots$ be an increasing sequence of finitely generated subgroups of $G$ with $G = \bigcup_{k \geq 1} G_k$. Since $G$ is locally nilpotent, each $G_k$ is a finitely generated nilpotent group. Let $\ell_k$ be the Hirsch length of $G_k$. Set $b_k = \ell_{k+1} + 1$. Fix $q_k > 0$ with
$$q_k \ < \ \frac{1}{4} \cdot 306^{-1} \ell_k^{-1} b_k^{-1} 2^{-22 \ell_k^2}.$$
Fix $\epsilon_k > 0$ with $12 \epsilon_k < q_k$.

We now construct, for each $k \geq 1$, a chart $\Phi_k = (\ell_k, \phi_k, \rZ_k, \Gamma_k, \bOne)$ for $G_k$. Note that since $G$ acts freely we have $X^\bOne = X$. For each $k$ fix a finite generating set $U_k$ for $G_k$. We desire our charts to satisfy the following for all $k \geq 1$:
\begin{enumerate}
\item[\rm (a)] $2^{40 \ell_k} \cdot \left(\frac{2}{\epsilon_k} \cdot 2 \cdot 36^2 \cdot 2^{14 \ell_k} \cdot \rZ_k\right) \subseteq \dom(\phi_k)$;
\item[\rm (b)] $\mathrm{Im}(\phi_k) \cup U_{k+1} \subseteq \phi_{k+1}(\rZ_{k+1})$.
\end{enumerate}
We build the charts $\Phi_k$ inductively on $k$. Formally letting $\mathrm{Im}(\phi_0)=\varnothing$ for the base case $k=1$, inductively we obtain $\Phi_k$ by applying Lemma \ref{LEM EXIST} to $G_k$ with
\begin{align*}
 F       \ & = \ \mathrm{Im}(\phi_{k-1}) \cup U_k; \\
 \lambda \ & = \ 2^{40 \ell_k} \cdot \frac{2}{\epsilon_k} \cdot 2 \cdot 36^2 \cdot 2^{14 \ell_k}. \\
\intertext{Now define}
 A_k     \ & = \ \frac{2}{\epsilon_k} \cdot 2 \cdot 36^2 \cdot 2^{14 \ell_k} \cdot \rZ_k.
\end{align*}
Clause (a) and the definitions of $q_k$ and $\epsilon_k$ imply that all inequality and containment requirements of Lemma \ref{LEM MAIN} are satisfied by $\Phi_k$, $A_k$, $b_k$, $q_k$, and $\epsilon_k$ for the action of $G_k$ on $X$. Thus in future applications of Lemma \ref{LEM MAIN} using these parameters, we only need to verify that $b_k$ has the stated bounding property. Furthermore, whenever we apply Lemma \ref{LEM MAIN} to $\Phi_k$ below, we will take it to be implicitly understood that we have in mind additionally the action $G_k\acts X$ and the parameters $A_k$, $b_k$, $q_k$, and $\epsilon_k$ without explicitly listing these each time.

We seek to build a collection of equivalence relations $\{E_{k,n} \: 1 \leq k \leq n\}$. We visualize these equivalence relations as being arranged in a two-dimensional array with the $E_{n,n}$'s placed along the rising diagonal and all other equivalence relations placed below the diagonal, as shown in Figure \ref{FIG1}. Thus for each $1\leq k\leq n$, the equivalence relation $E_{k,n}$ lies in row $k$ and column $n$. We will want the equivalence relations appearing along row $k$ to be $(\Phi_k, A_k, \epsilon_k)$-sub-rectangular equivalence relations which are pairwise $(\Phi_k, q_k \cdot A_k)$-orthogonal (clauses (ii) and (iii) below). The equivalence relations will be constructed column by column from left to right, and within a column by starting on the diagonal and moving vertically downward as indicated by the arrows in Figure \ref{FIG1}. Intuitively, higher rows will correspond to larger scales in our construction. Whenever $k<n$, we will construct $E_{k,n}$ so that at the scale of $\Phi_{k+1}$, $E_{k+1,n}$ and $E_{k,n}$ are essentially indistinguishable from one another (clause (iv) below). Finally, we want the equivalence relations appearing in row $k$ to satisfy, with respect to $b_k$, the bounding property that appears in the statement of Lemma \ref{LEM MAIN}. Specifically we want the array of equivalence relations to satisfy the following:
\begin{enumerate}
\item[\rm (i)] each $E_{k,n}$ is a finite $G$-clopen equivalence relation contained in the $G_n$-orbit equivalence relation;
\item[\rm (ii)] if an equivalence relation $E_{k,n}$ lies in the $k^\text{th}$ row, then it is $(\Phi_k, A_k, \epsilon_k)$-sub-rectangular;
\item[\rm (iii)] if two equivalence relations $E_{k,m}$ and $E_{k,n}$ both lie in row $k$, then they are $(\Phi_k, q_k \cdot A_k)$-orthogonal;
\item[\rm (iv)] if two equivalence relations $E_{k,n}$ and $E_{k+1,n}$ in column $n$ and rows $k$ and $k+1$ are vertically adjacent, then for every $x, y \in X$ with $$\phi_{k+1}(\rZ_{k+1}) \cdot x \ \subseteq \ [x]_{E_{k+1,n}} \ \mbox{ and } \ \phi_{k+1}(\rZ_{k+1}) \cdot y \ \subseteq \ [y]_{E_{k+1,n}}$$ we have $x \ E_{k,n} \ y$ if and only if $x \ E_{k+1,n} \ y$;
\item[\rm (v)] for every $x \in X$, there are at most $b_k$-many integers $t \geq k$ (equivalently, there are at most $b_k$-many equivalence relations in row $k$) such that $\phi_k(8 \cdot 2^{14 \ell_k} \cdot A_k) \cdot x \not\subseteq [x]_{E_{k,t}}$.
\end{enumerate}

\begin{figure}
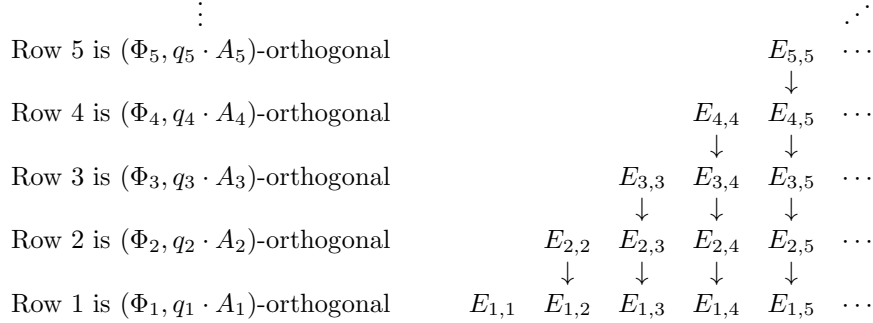

\begin{tabular}{cccccccc}
$\displaystyle{\vdots}$ & $\quad$ & & & & & & $\displaystyle{\Ddots}$\\
Row 5 is $\displaystyle{(\Phi_5, q_5 \cdot A_5)}$-orthogonal & $\quad$ & & & & & $\displaystyle{E_{5,5}}$ & $\displaystyle{\cdots}$\\
 & & & & & & $\displaystyle{\downarrow}$ & \\
Row 4 is $\displaystyle{(\Phi_4, q_4 \cdot A_4)}$-orthogonal & $\quad$ & & & & $\displaystyle{E_{4,4}}$ & $\displaystyle{E_{4,5}}$ & $\displaystyle{\cdots}$\\
 & & & & & $\displaystyle{\downarrow}$ & $\displaystyle{\downarrow}$ & \\
Row 3 is $\displaystyle{(\Phi_3, q_3 \cdot A_3)}$-orthogonal & $\quad$ & & & $\displaystyle{E_{3,3}}$ & $\displaystyle{E_{3,4}}$ & $\displaystyle{E_{3,5}}$ & $\displaystyle{\cdots}$\\
 & & & & $\displaystyle{\downarrow}$ & $\displaystyle{\downarrow}$ & $\displaystyle{\downarrow}$ & \\
Row 2 is $\displaystyle{(\Phi_2, q_2 \cdot A_2)}$-orthogonal & $\quad$ & & $\displaystyle{E_{2,2}}$ & $\displaystyle{E_{2,3}}$ & $\displaystyle{E_{2,4}}$ & $\displaystyle{E_{2,5}}$ & $\displaystyle{\cdots}$\\
 & & & $\displaystyle{\downarrow}$ & $\displaystyle{\downarrow}$ & $\displaystyle{\downarrow}$ & $\displaystyle{\downarrow}$ & \\
Row 1 is $\displaystyle{(\Phi_1, q_1 \cdot A_1)}$-orthogonal & $\quad$ & $\displaystyle{E_{1,1}}$ & $\displaystyle{E_{1,2}}$ & $\displaystyle{E_{1,3}}$ & $\displaystyle{E_{1,4}}$ & $\displaystyle{E_{1,5}}$ & $\displaystyle{\cdots}$
\end{tabular}
\caption{The first five columns of the two-dimensional array of equivalence relations. Arrows indicate the order of construction. \label{FIG1}}
\end{figure}

We construct the two-dimensional array of equivalence relations inductively, one column at a time, from left to right. To begin, apply Lemma \ref{LEM MAIN} to $\Phi_1$ with $s = 0$ (no previous sequence of equivalence relations) to obtain a $G_1$-clopen $(\Phi_1, A_1, \epsilon_1)$-rectangular equivalence relation $E_{1,1}$. This completes the first column. Trivially (i) -- (v) hold, where for $G$-clopenness in (i) we use Lemma \ref{LEM GHCLOPEN}. Now suppose that all equivalence relations in columns $1$ through $n-1$ have been constructed and satisfy clauses (i) through (v). Before constructing the equivalence relations in column $n$, we first construct a column of auxiliary equivalence relations $\{F_k \: 1 \leq k \leq n\}$. We require the following:
\begin{enumerate}
\item[\rm (i$'$)] if $F_k$ lies in row $k$ then it is a finite $G$-clopen equivalence relation contained in the $G_k$-orbit equivalence relation;
\item[\rm (ii$'$)] if $F_k$ lies in row $k$ then it is $(\Phi_k, A_k, \epsilon_k)$-rectangular;
\item[\rm (iii$'$)] if $F_k$ lies in row $k$ then it is $(\Phi_k, q_k \cdot A_k)$-orthogonal to each of the pre-existing equivalence relations in row $k$, namely $E_{k,k}, E_{k,k+1}, \ldots, E_{k,n-1}$.
\end{enumerate}
For each $k$ we apply Lemma \ref{LEM MAIN} to $\Phi_k$ with respect to the pre-existing equivalence relations $E_{k,k}, E_{k,k+1}, \ldots, E_{k,n-1}$ in row $k$ in order to obtain a $G_k$-clopen $(\Phi_k, A_k, \epsilon_k)$-rectangular equivalence relation $F_k$ which is $(\Phi_k, q_k \cdot A_k)$-orthogonal to each of $E_{k,k}, \ldots, E_{k,n-1}$. We remark that in applying Lemma \ref{LEM MAIN}, the bounding property of $b_k$ holds by the inductive assumption (v), and the other requirements of the lemma were verified immediately after the listing of clauses (a) and (b) above. As $F_k$ is $G_k$-clopen and contained in the $G_k$-orbit equivalence relation, it is $G$-clopen by Lemma \ref{LEM GHCLOPEN}. This defines the $F_k$'s.

For each $1 \leq k < n$ let $\sigma_k : X \rightarrow X$ be a $G$-clopen and continuous selector for $F_k$ as given by Lemma \ref{LEM SELECT}. Define $E_{n,n} = F_n$, and in general once $E_{k+1,n}$ has been defined let $E_{k,n}$ be the $F_k$-approximation to $E_{k+1,n}$ induced by $\sigma_k$. Specifically, $E_{k,n}$ is defined by the rule
$$x \ E_{k,n} \ y \quad \Longleftrightarrow \quad \sigma_k(x) \ E_{k+1,n} \ \sigma_k(y).$$
Clearly $E_{k,n}$ contains $F_k$ as a sub-equivalence relation, and thus $E_{k,n}$ is $(\Phi_k, A_k, \epsilon_k)$-sub-rectangular as required by clause (ii). As $F_k \subseteq E_{k,n}$ we immediately obtain $\partial_i^{\Phi_k}(E_{k,n}, q_k \cdot A_k) \subseteq \partial_i^{\Phi_k}(F_k, q_k \cdot A_k)$ for every $1 \leq i \leq \ell_k$, and thus clause (iii) is satisfied. We verify clauses (iv), (v), and (i) below.

(iv). Let $x, y \in X$ be such that $\phi_{k+1}(\rZ_{k+1}) \cdot x \subseteq [x]_{E_{k+1,n}}$ and $\phi_{k+1}(\rZ_{k+1}) \cdot y \subseteq [y]_{E_{k+1,n}}$. Since $F_k$ is $\Phi_k$-rectangular (as opposed to just sub-rectangular), we have $\sigma_k(x) \in \mathrm{Im}(\phi_k) \cdot x \subseteq \phi_{k+1}(\rZ_{k+1}) \cdot x$ and similarly $\sigma_k(y) \in \phi_{k+1}(\rZ_{k+1}) \cdot y$. So our assumption on $x$ and $y$ gives that $x \ E_{k+1,n} \ \sigma_k(x)$ and $y \ E_{k+1,n} \ \sigma_k(y)$. It follows that $x \ E_{k+1,n} \ y$ if and only if $\sigma_k(x) \ E_{k+1,n} \ \sigma_k(y)$ if and only if $x \ E_{k,n} \ y$.

(v). Fix $x \in X$ and let $J$ be the set of integers $t \in \{k + 1, \ldots, n\}$ with $\phi_k(8 \cdot 2^{14 \ell_k} \cdot A_k) \cdot x \not\subseteq [x]_{E_{k,t}}$. Note that we do not allow $k \in J$ and therefore we need only show that $|J| \leq b_k - 1 = \ell_{k+1}$. For any such $t \in J$ there is $y \in \phi_k(8 \cdot 2^{14 \ell_k} \cdot A_k) \cdot x$ which is $E_{k,t}$-inequivalent to $x$. Since $t > k$, there is an equivalence relation $E_{k+1,t}$ sitting above $E_{k,t}$ in the array. Clause (iv) implies that either $\neg\, x \ E_{k+1,t} \ y$, $\phi_{k+1}(\rZ_{k+1}) \cdot x \not\subseteq [x]_{E_{k+1,t}}$ or $\phi_{k+1}(\rZ_{k+1}) \cdot y \not\subseteq [y]_{E_{k+1,t}}$. In any case we have $\phi_{k+1}(3 \cdot \rZ_{k+1}) \cdot x \not\subseteq [x]_{E_{k+1,t}}$ since $\mathrm{Im}(\phi_k) \subseteq \phi_{k+1}(\rZ_{k+1})$. Now Lemma \ref{LEM STRONG BND} implies that for every $t \in J$ there is $1 \leq i \leq \ell_{k+1}$ with
$$x \ \in \ \phi_{k+1}(30 \ell_{k+1} \cdot (q_{k+1} \cdot A_{k+1})) \cdot \partial_i^{\Phi_{k+1}}(E_{k+1,t}, q_{k+1} \cdot A_{k+1}).$$
Since the equivalence relations in row $k+1$ are $(\Phi_{k+1}, q_{k+1} \cdot A_{k+1})$-orthogonal by clause (iii) and $x$ is fixed, we must have $|J| \leq \ell_{k+1}$.

(i). The equivalence relation $E_{n,n} = F_n$ is finite and contained in the $G_n$-orbit equivalence relation. Since each $\sigma_k$ is finite-to-one and moves points within their $G_n$-orbits, it follows that each $E_{k,n}$ is finite and contained in the $G_n$-orbit equivalence relation. It only remains to check the $G$-clopen property. We have that $E_{n,n} = F_n$ is $G$-clopen by clause (i$'$). Now inductively assume that $E_{k+1,n}$ is $G$-clopen. Fix $g\in G$. For each $x\in X$ there exist unique $u,v\in \mathrm{Im}(\phi_k)$ such that $\sigma_k(x)=u\cdot x$ and $\sigma_k(g\cdot x)=vg\cdot x$. Fixing $u,v\in \mathrm{Im}(\phi_k)$, let $X_{u,v}$ consist of those $x\in X$ for which $\sigma_k(x)=u\cdot x$ and $\sigma_k(g\cdot x)=vg\cdot x$, and note that $X_{u,v}$ is clopen since $\sigma_k$ is $G$-clopen. Set
$$W_{u,v}=\{x\in X_{u,v}\:x \ E_{k+1,n} \ vgu^{-1}\cdot x\}.$$
Then $W_{u,v}$ is clopen by our inductive assumption, and for $x\in X_{u,v}$ we have
$$x \ E_{k,n} \ g \cdot x \ \Longleftrightarrow \ u \cdot x \ E_{k+1,n} \ v g \cdot x \ \Longleftrightarrow \ u \cdot x \in W_{u,v} \ \Longleftrightarrow \ x \in u^{-1} \cdot W_{u,v}.$$
Hence the set of $x\in X_{u,v}$ such that $x \ E_{k,n} \ g\cdot x$ is clopen. As this holds for every pair $u,v$ of group elements ranging over the finite set $\mathrm{Im}(\phi_k)$, and since each $X_{u, v}$ is itself clopen, it follows that the set of $x\in X$ for which $x \ E_{k,n} \ g\cdot x$ is clopen, and therefore $E_{k,n}$ is $G$-clopen, establishing (i). This completes the construction of the equivalence relations $E_{k,n}$, $1 \leq k \leq n$.

We note that our equivalence relations have a useful sixth property which is an extension of clause (iv):
\begin{enumerate}
\item[\rm (vi)] if two equivalence relations $E_{k,n}$ and $E_{t,n}$ both lie in column $n$ with $t \leq k$, then for every $x, y \in X$ with $\phi_k(3 \cdot \rZ_k) \cdot x \subseteq [x]_{E_{k,n}}$ and $\phi_k(3 \cdot \rZ_k) \cdot y \subseteq [y]_{E_{k,n}}$ we have $x \ E_{k,n} \ y$ if and only if $x \ E_{t,n} \ y$.
\end{enumerate}
To see that this is true, first notice that $x \ E_{k-1,n} \ y$ if and only if $x \ E_{k,n} \ y$ by clause (iv). Next, let $x,y$ satisfy the hypotheses and consider $x' \in \phi_{k-1}(3 \cdot \rZ_{k-1}) \cdot x$. Note that
$$\phi_k(\rZ_k) \phi_{k-1}(3 \cdot \rZ_{k-1}) \ \subseteq \ \phi_k(\rZ_k) \phi_k(\rZ_k) \ \subseteq \ \phi_k(3 \cdot \rZ_k),$$
so we have $\phi_k(\rZ_k) \cdot x' \subseteq \phi_k(3 \cdot \rZ_k) \cdot x \subseteq [x]_{E_{k,n}}$. Thus clause (iv), applied to $x$ and $x'$, immediately implies that $x' \ E_{k-1,n} \ x$. Since $x'$ was chosen arbitrarily, we have $\phi_{k-1}(3 \cdot \rZ_{k-1}) \cdot x \subseteq [x]_{E_{k-1,n}}$. Similarly, $\phi_{k-1}(3 \cdot \rZ_{k-1}) \cdot y \subseteq [y]_{E_{k-1,n}}$. Thus the assumption of clause (vi) is renewed as one travels vertically down the column, and clause (vi) now follows by induction.

We now consider the bottom row of equivalence relations in the array and claim that for $x, y \in X$,
$$x\mathrel{E^X_G}y \quad \Longleftrightarrow \quad (\exists m \in \N) \ (\forall n \geq m) \quad x \ E_{1,n} \ y.$$
Clearly each $E_{1,n}$ is contained in $E^X_G$ and thus the right-hand side implies the left. So suppose $x\mathrel{E^X_G}y$. Since $G = \bigcup_k G_k$, there is $k$ with $G_k \cdot x = G_k \cdot y$. Clauses (ii) and (iii) and Lemma \ref{LEM ORTHOSEQ} imply that $x \ E_{k,n} \ y$ for all but finitely many $n \geq k$. In other words, along the $k^\text{th}$ row $x$ and $y$ are inequivalent only finitely many times. Applying Lemma \ref{LEM ORTHOSEQ} again repeatedly to each pair $x,x'$ for $x'\in\phi_k(3\cdot\rZ_k)\cdot x$, and $y,y'$ for $y'\in\phi_k(3\cdot\rZ_k)\cdot y$, we have that $\phi_k(3 \cdot \rZ_k) \cdot x \subseteq [x]_{E_{k,n}}$ and $\phi_k(3 \cdot \rZ_k) \cdot y \subseteq [y]_{E_{k,n}}$ for all but finitely many $n$. So clause (vi) implies that $x \ E_{1,n} \ y$ for all but finitely many $n$. We conclude that the orbit equivalence relation $E^X_G$ induced by $G$ is hyperfinite. Furthermore, it follows from Lemma \ref{LEM E0} (using $F_n = E_{1,n}$) that $E^X_G$ continuously embeds into $E_0$. \end{proof}

\begin{rem}
An alternative approach at the end of the proof of Theorem \ref{THM FREE} would be to show that for $n \geq k$, the equivalence relations $E_{1,n}$ along the bottom row are $(\Phi_k, A_k, 2 \epsilon_k)$-sub-rectangular and pairwise $(\Phi_k, q_k \cdot A_k)$-orthogonal. Then one could apply Lemma \ref{LEM ORTHOSEQ} directly without clause (vi). In fact for $n\geq k$ the equivalence relations $E_{1,n}$ \emph{are} $(\Phi_k, A_k, 2 \epsilon_k)$-sub-rectangular. Morally speaking they should also be pairwise $(\Phi_k, q_k \cdot A_k)$-orthogonal, but unfortunately they are not due to minor technical shortcomings in our definition of ``boundary." These shortcomings could perhaps be overcome with some effort, and in any case we believe that this alternative approach is a good viewpoint to have.
\end{rem}

We now prove the main theorem in the general situation where the action is not necessarily free. In this case the proof has three substantial differences from the free action case.
\begin{list}{}{\leftmargin=6mm}
\item[(1)] As before, we will express $G$ as the increasing union of finitely generated subgroups $G_k$. Previously when the action was free we were able to use a single chart $\Phi_k$ to describe the action of $G_k$ on $X$. However, since we only know how to construct charts where $\cH$ is finite, in general the set $X^\cH$ will be a proper subset of $X$ and need not even be $G_k$-invariant. So in order to describe the action of $G_k$ on $X$ we must use a countably infinite family of charts. The charts for $G_k$ will be grouped into finite families denoted $\Phi_{k,n}$, for $n \geq k$. We visualize these finite families of charts as sitting in a two-dimensional triangular array, just as we viewed the equivalence relations $E_{k,n}$ in the proof of Theorem \ref{THM FREE}.
\item[(2)] In the free action case we built a two-dimensional array of equivalence relations in order to diagonalize over the finitely generated subgroups $G_k$. In the general case, we must diagonalize not only over the $G_k$'s but also over subgroups of the $G_k$'s that arise as stabilizers. As a result, we must now build a three-dimensional array of equivalence relations.
\item[(3)] In the free action case, the $\Phi_k$'s had a nice linear order in the sense that the scale of $\Phi_k$ was tiny compared to that of $\Phi_{k+1}$. This fact allowed for clauses (iv) and (vi) in the proof of \ref{THM FREE}. In the general case we are not able to put any such linear order on the $\Phi_{k,n}$'s. So clauses (iv) and (vi) disappear in the general case and are replaced by a single new clause, which is again given number (iv). This new clause plays a similar role in proving the theorem, but this time it is far more delicate and technical.
\end{list}
Aside from these three differences, the details of the free action proof and the general case proof are quite similar. In fact, in order to help the reader transition between the two cases, we intentionally structured the proofs to emphasize their similarities.

In the general case we are unable to obtain a continuous reduction to $E_0$ and we therefore drop all topological constraints such as requiring that certain sets be clopen or requiring the action to be continuous.

\begin{thm} \label{THM NFREE}
If the countable locally nilpotent group $G$ acts in a Borel fashion on the standard Borel space $X$, then the induced orbit equivalence relation $E^X_G$ is hyperfinite.
\end{thm}

\begin{proof}
Let $G_1 \leq G_2 \leq \cdots$ be an increasing sequence of finitely generated subgroups of $G$ which exhaust $G$. Since $G$ is locally nilpotent, each $G_k$ is a finitely generated nilpotent group. Let $\ell_k$ be the Hirsch length of $G_k$. For $n \geq k$ set
$$b_{k,n} = 1 + \max_{k \leq t \leq n} \ell_t.$$
Fix any $q_{k,n} > 0$ with
$$q_{k,n} \ < \ \frac{1}{4} \cdot 306^{-1} \ell_k^{-1} b_{k,n}^{-1} 2^{-22 \ell_k^2},$$
and choose $\epsilon_{k,n} > 0$ so that $12 \epsilon_{k,n} < q_{k,n}$.

We will use a countably infinite collection of charts for $G_k$ to describe the action of $G_k$ on $X$. These charts will be grouped into finite families denoted $\Phi_{k,n}$, $n \geq k$, which for fixed $k$ we visualize as sitting in an infinite horizontal row extending to the right. When both $n$ and $k$ vary we imagine the $\Phi_{k,n}$'s as sitting in a two-dimensional array as in Figure \ref{FIG3}. Since every subgroup of $G_k$ is finitely generated (\cite[5.2.17]{R}), the space $\Sub(G_k)$ of subgroups of $G_k$ is countable. Therefore we can pick a transversal $T_k \subseteq \Sub(G_k)$ for the action of $G_k$ on $\Sub(G_k)$ by conjugation, and this transversal will automatically be Borel. For $H \in T_k$ define
$$X_k^H \ := \ \{x \in X \: \Stab(x) \cap G_k \text{ is } G_k \text{-conjugate to } H\}.$$
Note that $X_k^H$ is a $G_k$-invariant Borel set and if $H \neq L \in T_k$ then $X_k^H$ and $X_k^L$ are disjoint. For each $k$ fix an increasing sequence
$$S_{k,k} \subseteq S_{k, k+1} \subseteq \cdots \subseteq S_{k,n} \subseteq \cdots$$
of finite subsets of $T_k$ which exhausts $T_k$. The finite family $\Phi_{k,n}$ will consist of charts $\Phi_{k,n}^H$ indexed by $H \in S_{k,n}$. We will have $$\Phi_{k,n}^H = (\ell_{k,n}^H, \phi_{k,n}^H, \rZ_{k,n}^H, \Gamma_{k,n}^H, \cH_{k,n}^H),$$
where $\cH_{k,n}^H$ consists of finitely many $G_k$-conjugates of $H$. For each $1\leq k\leq n$ and $H\in S_{k,n}$, we will write
$$X_{k,n}^H \ := \ \{x \in X \: \Stab(x) \cap G_k \in \cH_{k,n}^H\} \ \subseteq \ X_k^H.$$
Thus for fixed $1\leq k\leq n$, the finitely many sets $X_{k,n}^H$, $H \in S_{k,n}$, are contained in pairwise disjoint $G_k$-invariant Borel sets. Notice that in our notation from prior sections, $X_{k,n}^H$ is just the set $X^{\cH_{k,n}^H}$, where $X^{\cH_{k,n}^H}$ is defined using the $G_k$-action.

\begin{figure}
\setlength{\unitlength}{4mm}
\centering
\begin{picture}(27,14)
\put(0,0.5){\makebox{Families of charts for $G_1$}}
\put(0,4){\makebox{Families of charts for $G_2$}}
\put(0,7.5){\makebox{Families of charts for $G_3$}}
\put(0,11){\makebox{Families of charts for $G_4$}}
\put(5,12.75){\makebox{$\displaystyle{\vdots}$}}
\put(12,0.5){\makebox{$\displaystyle{\Phi_{1,1}}$}}
\put(16,0.5){\makebox{$\displaystyle{\Phi_{1,2}}$}}
		\put(13.7,0.75){\vector(1,0){2}}
\put(20,0.5){\makebox{$\displaystyle{\Phi_{1,3}}$}}
		\put(17.1,3.7){\vector(1,-1){2.7}}
\put(24,0.5){\makebox{$\displaystyle{\Phi_{1,4}}$}} 
		\put(20.9,7.2){\vector(1,-2){3.05}} 
\put(16,4){\makebox{$\displaystyle{\Phi_{2,2}}$}}
		\put(16.5,1.35){\vector(0,1){2.3}}
\put(20,4){\makebox{$\displaystyle{\Phi_{2,3}}$}}
		\put(20.5,1.35){\vector(0,1){2.3}}
\put(24,4){\makebox{$\displaystyle{\Phi_{2,4}}$}} 
		\put(24.5,1.35){\vector(0,1){2.3}} 
\put(20,7.5){\makebox{$\displaystyle{\Phi_{3,3}}$}}
		\put(20.5,4.85){\vector(0,1){2.3}}
\put(24,7.5){\makebox{$\displaystyle{\Phi_{3,4}}$}} 
		\put(24.5,4.85){\vector(0,1){2.3}} 
\put(24,11){\makebox{$\displaystyle{\Phi_{4,4}}$}} 
		\put(24.5,8.35){\vector(0,1){2.3}} 
\put(26,0.5){\makebox{$\cdots$}}
\put(26,4){\makebox{$\cdots$}}
\put(26,7.5){\makebox{$\cdots$}}
\put(26,11){\makebox{$\cdots$}}
\put(26,12.75){\makebox{$\Ddots$}}
\end{picture}
\caption{The order of construction of the finite families of charts $\Phi_{k,n}$, $k\leq n$. \label{FIG3}}
\end{figure}
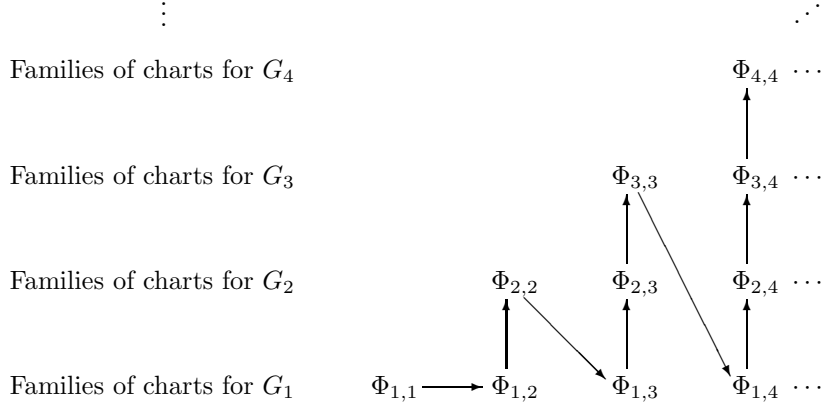

We now construct the two-dimensional array of finite families of charts $\Phi_{k,n} = \{\Phi_{k,n}^H \: H \in S_{k,n}\}$ for $G_k$, $k \leq n$. We will simultaneously construct increasing finite sets $V_{k,n} \subseteq G_k$, $n \geq k$, which will play an important role in the new version of clause (iv). For each $k$ fix a finite generating set $U_k$ for $G_k$. Without loss of generality, we may assume that the $U_k$'s are increasing. We will want the families of charts $\Phi_{k,n}$ and subsets $V_{k,n}\subseteq G_k$ to possess the following properties for all $1 \leq k \leq n$ and $H \in S_{k,n}$:
\begin{enumerate}
\item[\rm (a)] $2^{40 \ell_k} \cdot \left(\frac{2}{\epsilon_{k,n}} \cdot 2 \cdot 36^2 \cdot 2^{14 \ell_k} \cdot \rZ_{k,n}^H\right) \subseteq \dom(\phi_{k,n}^H)$;
\item[\rm (b)] $V_{k,n} V_{k,n} \cdot L \subseteq \phi_{k,n}^H(\rZ_{k,n}^H) \cdot L$ for all $L \in \cH_{k,n}^H$;
\item[\rm (c)] $\ell_{k,n}^H \leq \ell_k$;
\item[\rm (d)] the elements of $V_{k,n}$ are the $n^2$-fold products $g_1 g_2 \cdots g_{n^2}$ where for every $1 \leq i \leq n^2$, either $g_i$ or $g_i^{-1}$ is contained in $U_k$ or is contained in $\mathrm{Im}(\phi_{s,t}^L)$ for some $1 \leq s \leq t$ and $L \in S_{s,t}$ satisfying $s \leq k$, $t \leq n$, and $(s, t) \neq (k,n)$;
\item[\rm (e)] if $n > k$ and $H \in S_{k,n-1}$ then for any
$$g \ \in \ \phi_{k,n}^H \left( 16 \ell_k q_{k,n} \cdot \frac{2}{\epsilon_{k,n}} \cdot 2 \cdot 36^2 \cdot 2^{14 \ell_k} \cdot \rZ_{k,n}^H \right)$$
we have that $g L g^{-1} \in \cH_{k,n}^H$ for every $L \in \cH_{k,n-1}^H$.
\end{enumerate}
Once the array of $\Phi_{k,n}$'s is constructed we will define
$$A_{k,n}^H \ := \ \frac{2}{\epsilon_{k,n}} \cdot 2 \cdot 36^2 \cdot 2^{14 \ell_k} \cdot \rZ_{k,n}^H$$
for $H \in S_{k,n}$. Then from clauses (a) and (c) and the definitions of $q_{k,n}$ and $\epsilon_{k,n}$ we will have that all inequality and containment requirements of Lemma \ref{LEM MAIN} are satisfied by $\Phi_{k,n}^H$, $A_{k,n}^H$, $b_{k,n}$, $q_{k,n}$, and $\epsilon_{k,n}$ for the action of $G_k$ on $X$. Thus in future applications of Lemma \ref{LEM MAIN} using these parameters, we need only verify that $b_{k,n}$ has the stated bounding property. As before, whenever we apply Lemma \ref{LEM MAIN} to $\Phi_{k,n}^H$ below, we will have in mind the action $G_k\acts X$ and the parameters $A_{k,n}^H$, $b_{k,n}$, $q_{k,n}$, and $\epsilon_{k,n}$ without stating this explicitly.

We build the finite families of charts $\Phi_{k,n}$ inductively in the order indicated by the arrows in Figure \ref{FIG3}: namely, $\Phi_{k+1,n}$ is constructed immediately after $\Phi_{k,n}$ if $k < n$, and $\Phi_{1,n+1}$ is constructed immediately after $\Phi_{n,n}$. The base case is $\Phi_{1,1}$. Observe that if the families of charts and the $V_{i,j}$'s are constructed in this order, then the conditions imposed on $\Phi_{k,n}$ by clauses (a) through (e) refer only to families that have already been constructed. Furthermore, $V_{k,n}$ can be defined by clause (d). Note that $V_{k,n}$ is finite since the sets $S_{s,t}$, $U_k$, and $\mathrm{Im}(\phi_{s,t}^L)$ are all finite. Therefore the chart $\Phi_{k,n}^H$ can be constructed immediately by applying Lemma \ref{LEM NFEXIST} to $G_k$ with
$$\begin{array}{rcl}
 S       & = & \left\{\begin{array}{ll} \cH_{k,n-1}^H & \mbox{if } n > k \mbox{ and } H \in S_{k,n-1} \\ \{H\} & \mbox{otherwise;} \end{array}\right. \vspace{2mm} \\
 F       & = & V_{k,n} V_{k,n}; \vspace{3mm} \\
 \lambda & = & 2^{40 \ell_k} \cdot \frac{2}{\epsilon_{k,n}} \cdot 2 \cdot 36^2 \cdot 2^{14 \ell_k}; \vspace{2mm} \\
 \eta    & = & 16 \ell_k q_{k,n} \cdot \frac{2}{\epsilon_{k,n}} \cdot 2 \cdot 36^2 \cdot 2^{14 \ell_k}.
\end{array}$$

This completes the construction of the finite families $\Phi_{k,n}$ and the sets $V_{k,n}$. Notice that if $H \in S_{k,n-1}$ then there is almost no relationship between $\Phi_{k,n-1}^H$ and $\Phi_{k,n}^H$ beyond clause (e). (One could scrutinize the proof of Lemma \ref{LEM NFEXIST} to see that $\ell_{k,n-1}^H = \ell_{k,n}^H$, but we do not need this fact). In particular $\Phi_{k,n-1}^H \neq \Phi_{k,n}^H$, so the families $\Phi_{k,n}$ are not increasing despite the fact that the $S_{k,n}$'s are increasing.

We now make a few comments on the properties of the constructed charts. Recall that we defined
$$X_{k,n}^H \ = \ \{x \in X \: \Stab(x) \cap G_k \in \cH_{k,n}^H\} \ \subseteq \ X_k^H.$$
Also recall that $\Phi_{k,n}^H$ and $X_{k,n}^H$ are only defined if $H\in S_{k,n}$. We first point out that clauses (b) and (e) can be restated as follows:
\begin{enumerate}
\item[\rm (b$'$)] $V_{k,n} V_{k,n} \cdot x \subseteq \phi_{k,n}^H(\rZ_{k,n}^H) \cdot x$ for every $x \in X_{k,n}^H$;
\item[\rm (e$'$)] if $n > k$ and $H \in S_{k,n-1}$, then
$$\phi_{k,n}^H \left( 16 \ell_k q_{k,n} \cdot A_{k,n}^H\right) \cdot X_{k,n-1}^H \ \subseteq \ X_{k,n}^H.$$
\end{enumerate}
Observe that since $U_k$ generates $G_k$, clause (d) implies that the sets $V_{k,n}$ are increasing and $G_k = \bigcup_{n \geq k} V_{k,n}$. Also, clauses (b$'$) and (e$'$) imply that for any $H \in S_{k,n}$,
$$(\{1_G\} \cup U_k \cup U_k^{-1}) \cdot X_{k,n}^H \ \subseteq \ X_{k,n+1}^H$$
(note that indeed $H \in S_{k,n+1}$ since the $S_{k,i}$'s are increasing). From the definitions we see that for $H \in S_{k,n}$ the set $X_{k,n}^H$ meets every $G_k$-orbit of the $G_k$-invariant set $X_k^H$. So it follows from the above containment that for any $H \in S_{k,n}$ and any $x \in X_k^H$ there is $N \geq n$ with $x \in X_{k,n'}^H$ for all $n' \geq N$. Moreover, from the definitions we have that the sets $X_k^H$, $H \in T_k$, partition $X$, and for any $H \in T_k$ there is an $n$ with $H \in S_{k,n}$. In particular, for every $x \in X$ and $k \geq 1$, there is $N \geq k$ and $H \in S_{k,N}$ with $x \in X_{k,n}^H$ for all $n \geq N$.

We now introduce a function which plays an important role in the new version of clause (iv), to be introduced below. For $x \in X$ and $n \geq 1$ define
$$r_n(x) = \max \{k \: 1 \leq k \leq n \text{ and } \exists H \in S_{k,n} \text{ with } \phi_{k,n}^H(15 \ell_{k,n}^H q_{k,n} \cdot A_{k,n}^H) \cdot x \subseteq X_{k, n}^H\}.$$
We set $r_n(x) = 0$ if the above set is empty. Note that if $r_n(x) = k$ then there is $H \in S_{k,n}$ with $x \in X_{k,n}^H$, and furthermore, if there is $1 \leq k \leq n$ and $H \in S_{k,n}$ with $x \in X_{k,n}^H$ then by (e$'$) we have $r_{n+1}(x) \geq k$. So for each fixed $x \in X$, the values $r_n(x)$ are non-decreasing with $n$, and by the previous paragraph for every $k \geq 1$ there is an $n$ with $r_n(x) \geq k$.

Now we build a three-dimensional array of equivalence relations in order to diagonalize over both the $G_k$'s and the stabilizers appearing in each $\mathrm{Sub}(G_k)$. To describe displacement in the new third dimension we use the term ``layer.'' We will build a three-dimensional array of equivalence relations $\{E_{k,n}^m \: 1 \leq k \leq n \leq m\}$, where $E_{k,n}^m$ lies in row $k$, column $n$, and layer $m$, as shown in Figure \ref{FIG2}. We require these equivalence relations to satisfy the following:
\begin{enumerate}
\item[\rm (i)] each $E_{k,n}^m$ is a finite Borel equivalence relation contained in the $G_m$-orbit equivalence relation;
\item[\rm (ii)] if $E_{k,n}^m$ lies in row $k$ and column $n$, then it is $(\Phi_{k,n}^H, A_{k,n}^H, \epsilon_{k,n})$-sub-rectangular for every $H \in S_{k,n}$;
\item[\rm (iii)] if two equivalence relations $E_{k,n}^m$ and $E_{k,n}^t$ both lie in row $k$ and column $n$ then they are $(\Phi_{k,n}^H, q_{k,n} \cdot A_{k,n}^H)$-orthogonal for every $H \in S_{k,n}$;
\item[\rm (iv)] if $H \in S_{k,n}$, $x \in X_{k,n}^H$, and $r_n(x) = k$, then $\phi_{k,n}^H(\rZ_{k,n}^H) \cdot x \subseteq [x]_{E_{k,n}^m}$ implies $V_{k,n} \cdot x \subseteq [x]_{E_{s,t}^m}$ for all $1 \leq s \leq t < n$;
\item[\rm (v)] if $H \in S_{k,n}$ and $x \in X_{k,n}^H$ then there are at most $b_{k,n}$-many integers $t \geq n$ (equivalently, there are at most $b_{k,n}$-many equivalence relations in row $k$ and column $n$) such that $\phi_{k,n}^H(8 \cdot 2^{14 \ell_{k,n}^H} \cdot A_{k,n}^H) \cdot x \not\subseteq [x]_{E_{k,n}^t}$.
\end{enumerate}
We construct the three-dimensional array of equivalence relations inductively, one layer at a time. The first layer consists simply of $E_{1,1}^1$, which we now define. If there is no $H\in S_{1,1}$ for which $x\in X_1^H$, then declare the $E_{1,1}^1$-class of $x$ to consist simply of $x$ itself. We require that $E_{1,1}^1$ be contained in the $G_1$-orbit equivalence relation, and therefore it will suffice to define $E_{1,1}^1$ on each of the finitely many pairwise disjoint $G_1$-invariant Borel sets $X_1^H$, $H \in S_{1,1}$. For each $H \in S_{1,1}$ we use the chart $\Phi_{1,1}^H$ and apply Lemma \ref{LEM MAIN} to the action $G_1 \acts X_1^H$ with $s = 0$ (no previous sequence of equivalence relations) to obtain a $(\Phi_{1,1}^H, A_{1,1}^H, \epsilon_{1,1})$-rectangular equivalence relation on $X_1^H$, and we declare this equivalence relation to be the restriction of $E_{1,1}^1$ to $X_1^H$. This defines $E_{1,1}^1$. It is easy to check that (i) through (v) hold.

\begin{figure}
\setlength{\unitlength}{0.58cm}
\centering
\begin{picture}(22,14)
\put(0,1.5){\makebox{Layer $1$}}
\put(1,2.5){\makebox{Layer $2$}}
\put(2,3.5){\makebox{Layer $3$}}
\put(3,4.5){\makebox{Layer $4$}}
\put(20,3){\makebox{Row $1$}}
\put(20,6.5){\makebox{Row $2$}}
\put(20,10){\makebox{Row $3$}}
\put(20,13.5){\makebox{Row $4$}}
\put(3,0.5){\makebox{Column $1$}}
\put(8,0.5){\makebox{Column $2$}}
\put(13,0.5){\makebox{Column $3$}}
\put(18,0.5){\makebox{Column $4$}}
\put(3,1.5){\makebox{$\displaystyle{E_{1,1}^1}$}}
\put(4,2.5){\makebox{$\displaystyle{E_{1,1}^2}$}}
\put(5,3.5){\makebox{$\displaystyle{E_{1,1}^3}$}}
\put(6,4.5){\makebox{$\displaystyle{E_{1,1}^4}$}}
\put(8,2){\makebox{$\displaystyle{E_{1,2}^2}$}}
		\put(7.75,2){\vector(-4,1){2.5}}
\put(9,3){\makebox{$\displaystyle{E_{1,2}^3}$}}
		\put(8.75,3){\vector(-4,1){2.5}}
\put(10,4){\makebox{$\displaystyle{E_{1,2}^4}$}}
		\put(9.75,4){\vector(-4,1){2.5}}
\put(13,2.5){\makebox{$\displaystyle{E_{1,3}^3}$}}
		\put(13,3){\vector(-1,1){3}}
\put(14,3.5){\makebox{$\displaystyle{E_{1,3}^4}$}}
		\put(14,4){\vector(-1,1){3}}
\put(18,3){\makebox{$\displaystyle{E_{1,4}^4}$}}
		\put(18,3.5){\vector(-1,2){3.25}}
\put(8,5.5){\makebox{$\displaystyle{E_{2,2}^2}$}}
		\put(8.5,5.25){\vector(0,-1){2.5}}
\put(9,6.5){\makebox{$\displaystyle{E_{2,2}^3}$}}
		\put(9.5,6.25){\vector(0,-1){2.5}}
\put(10,7.5){\makebox{$\displaystyle{E_{2,2}^4}$}}
		\put(10.5,7.25){\vector(0,-1){2.5}}
\put(13,6){\makebox{$\displaystyle{E_{2,3}^3}$}}
		\put(13.5,5.75){\vector(0,-1){2.5}}
\put(14,7){\makebox{$\displaystyle{E_{2,3}^4}$}}
		\put(14.5,6.75){\vector(0,-1){2.5}}
\put(18,6.5){\makebox{$\displaystyle{E_{2,4}^4}$}}
		\put(18.5,6.25){\vector(0,-1){2.5}}
\put(13,9.5){\makebox{$\displaystyle{E_{3,3}^3}$}}
		\put(13.5,9.25){\vector(0,-1){2.5}}
\put(14,10.5){\makebox{$\displaystyle{E_{3,3}^4}$}}
		\put(14.5,10.25){\vector(0,-1){2.5}}
\put(18,10){\makebox{$\displaystyle{E_{3,4}^4}$}}
		\put(18.5,9.75){\vector(0,-1){2.5}}
\put(18,13.5){\makebox{$\displaystyle{E_{4,4}^4}$}}
		\put(18.5,13.25){\vector(0,-1){2.5}}
\end{picture}
\caption{The first four layers of the three-dimensional array of equivalence relations. Arrows indicate the order of construction. \label{FIG2}}
\end{figure}
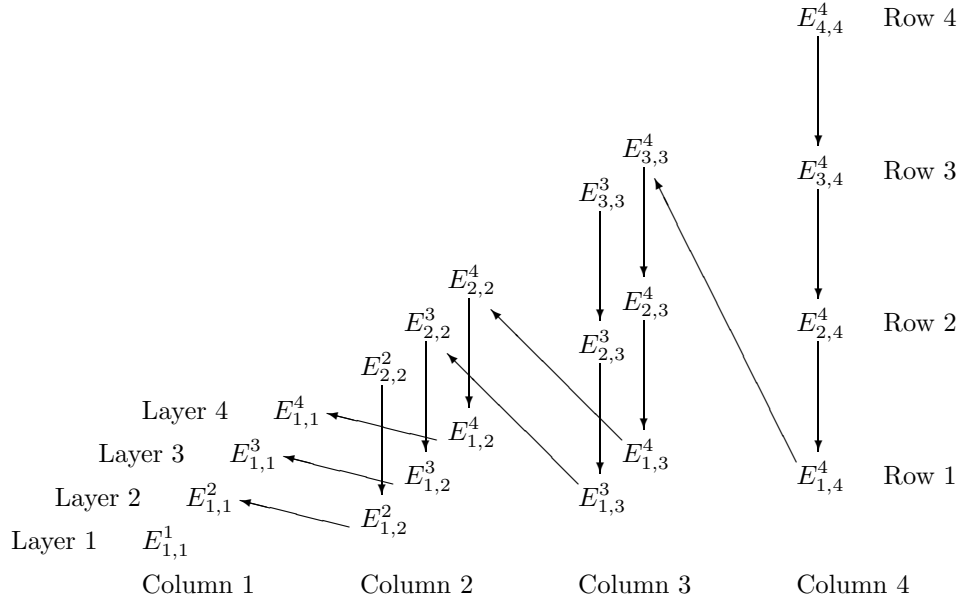

Now suppose that all equivalence relations on layers $1$ through $m-1$ have been constructed and satisfy clauses (i) through (v). We construct the equivalence relations in layer $m$: $\{E_{k,n}^m \: 1 \leq k \leq n \leq m\}$. We first define a whole layer of auxiliary equivalence relations $\{F_{k,n} \: 1 \leq k \leq n \leq m\}$. We will require the $F_{k,n}$'s to have the following properties:
\begin{enumerate}
\item[\rm (i$'$)] if $F_{k,n}$ lies in row $k$ then it is a finite Borel equivalence relation contained in the $G_k$-orbit equivalence relation;
\item[\rm (ii$'$)] if $F_{k,n}$ lies in row $k$ and column $n$, then for every $H \in S_{k,n}$ the restriction of $F_{k,n}$ to the $G_k$-invariant Borel set $X_k^H$ is $(\Phi_{k,n}^H, A_{k,n}^H, \epsilon_{k,n})$-rectangular;
\item[\rm (iii$'$)] if $F_{k,n}$ lies in row $k$ and column $n$, then for every $H \in S_{k,n}$, $F_{k,n}$ is $(\Phi_{k,n}^H, q_{k,n} \cdot A_{k,n}^H)$-orthogonal to each of the pre-existing equivalence relations in row $k$ and column $n$, namely $E_{k,n}^n, E_{k,n}^{n+1}, \ldots, E_{k,n}^{m-1}$.
\end{enumerate}
Fix $1 \leq k \leq n \leq m$. We will define $F_{k,n}$. If there is no $H\in S_{k,n}$ for which $x\in X_k^H$, then declare the $F_{k,n}$-class of $x$ to consist simply of $x$. Based on clause (i$'$) it will suffice to define $F_{k,n}$ on each of the finitely many pairwise disjoint $G_k$-invariant Borel sets $X_k^H$, $H \in S_{k,n}$. For each $H \in S_{k,n}$ we use the chart $\Phi_{k,n}^H$ and apply Lemma \ref{LEM MAIN} to the action $G_k \acts X_k^H$ with pre-existing equivalence relations $E_{k,n}^n, E_{k,n}^{n+1}, \ldots, E_{k,n}^{m-1}$ (note that this list is empty if $F_{k,n}$ lies in the right-most column or equivalently if $n = m$) to obtain a $(\Phi_{k,n}^H, A_{k,n}^H, \epsilon_{k,n})$-rectangular equivalence relation on $X_k^H$, and we declare this equivalence relation to be the restriction of $F_{k,n}$ to $X_k^H$. We remark that in applying Lemma \ref{LEM MAIN}, the bounding property of $b_{k,n}$ holds by the inductive assumption (v), and the other requirements of the lemma were verified immediately after the listing of clauses (a) through (e). Note that if the $F_{k,n}$-class of $x$ does not meet any of the sets $X_{k,n}^H$, $H \in S_{k,n}$, then the $F_{k,n}$-class of $x$ is a singleton. This defines the $F_{k,n}$'s.

Now we define the equivalence relations $E_{k,n}^m$ in layer $m$ in the order given by the arrows in Figure \ref{FIG2}. For $(k,n) \neq (m,m)$, fix a Borel selector $\sigma_{k,n} : X \rightarrow X$ for $F_{k,n}$ such that for all $x\in X$,
$$\sigma_{k,n}(x) \neq x \ \Longrightarrow \ (\exists H \in S_{k,n}) \ \sigma_{k,n}(x) \in X_{k,n}^H.$$
(It is clear that such a Borel selector exists; we are not using Lemma \ref{LEM SELECT} here). Set $E_{m,m}^m = F_{m,m}$. In general, if $E_{k+1,n}^m$ is defined then we let $E_{k,n}^m$ be the $F_{k,n}$-approximation to $E_{k+1,n}^m$ induced by $\sigma_{k,n}$. Specifically we define
$$x \ E_{k,n}^m \ y \ \Longleftrightarrow \ \sigma_{k,n}(x) \ E_{k+1,n}^m \ \sigma_{k,n}(y).$$
Similarly, whenever $E_{1,n+1}^m$ is defined we let $E_{n,n}^m$ be the $F_{n,n}$-approximation to $E_{1,n+1}^m$ induced by $\sigma_{n,n}$, meaning
$$x \ E_{n,n}^m \ y \ \Longleftrightarrow \ \sigma_{n,n}(x) \ E_{1,n+1}^m \ \sigma_{n,n}(y).$$
This defines the equivalence relations lying in layer $m$.

We check that clauses (i) through (v) continue to hold. By construction $E_{m,m}^m$ is contained in the $G_m$-orbit equivalence relation and each $\sigma_{k,n}$ moves points within their $G_m$-orbit. Therefore each $E_{k,n}^m$ is contained in the $G_m$-orbit equivalence relation. Also, $E_{m,m}^m$ is finite and Borel, and each map $\sigma_{k,n}$ is finite-to-one and Borel. So it follows that each $E_{k,n}^m$ is finite and Borel. Thus clause (i) holds. Clause (ii) holds since $E_{k,n}^m$ contains $F_{k,n}$ as a sub-equivalence relation, and $F_{k,n}$ is $(\Phi_{k,n}^H, A_{k,n}^H, \epsilon_{k,n})$-rectangular on $X_k^H$ for each $H \in S_{k,n}$. Similarly, clause (iii) holds since $E_{k,n}^m$ contains $F_{k,n}$ and for every $H \in S_{k,n}$ the equivalence relation $F_{k,n}$ is $(\Phi_{k,n}^H, q_{k,n} \cdot A_{k,n}^H)$-orthogonal to all pre-existing equivalence relations lying in row $k$ and column $n$. We prove that clauses (iv) and (v) hold in the next two paragraphs.

(iv). Let $1 \leq k \leq n \leq m$, let $H \in S_{k,n}$, let $x \in X_{k,n}^H$, and assume that $r_n(x) = k$. Suppose that
$$\phi_{k,n}^H(\rZ_{k,n}^H) \cdot x \ \subseteq \ [x]_{E_{k,n}^m}.$$
Fix $1 \leq s \leq t < n$ and fix any $y \in V_{k,n} \cdot x$. We must show that $x \ E_{s,t}^m \ y$. Set $y_{s,t} = y$. In general, if $y_{i,j}$ is defined and $i < j < n$ then set $y_{i+1,j} = \sigma_{i,j}(y_{i,j})$. If $y_{i,n}$ is defined and $i < k$ then set $y_{i+1,n} = \sigma_{i,n}(y_{i,n})$. If $y_{j,j}$ is defined and $j < n$ then set $y_{1,j+1} = \sigma_{j,j}(y_{j,j})$. This defines the collection of points
$$\{y_{i,t} \: s \leq i \leq t\} \cup \{y_{i,j} \: 1 \leq i \leq j \text{ and } t < j < n\} \cup \{y_{i,n} \: 1 \leq i \leq k\}.$$
If one imagines these points as lying in a two-dimensional traingular array in the natural way, then these points occupy the upper portion of column $t$, all columns strictly between $t$ and $n$, and the lower portion of column $n$. Our first goal is to show that
$$y_{k,n} \ \in \ \phi_{k,n}^H(\rZ_{k,n}^H) \cdot x.$$
In order to achieve this goal we must study which group element takes $y$ to $y_{k,n}$. For this it suffices to study which group element takes each $y_{i,j}$ to $y_{i+1,j}$ and which group element takes each $y_{j,j}$ to $y_{1,j+1}$. For $(i,j) \neq (k,n)$, we say that an \emph{event} occurs at $(i,j)$ if $\sigma_{i,j}(y_{i,j}) \neq y_{i,j}$. When $i < j$ this is equivalent to saying $y_{i+1,j} \neq y_{i,j}$, and when $i = j$ this is equivalent to saying $y_{1,i+1} \neq y_{i,i}$. In order to understand how one travels from $y$ to $y_{k,n}$, one only needs to understand what happens at the locations $(i,j)$ where an event occurs. If an event occurs at $(i,j)$ with $i < j$ then there must be $L \in S_{i,j}$ with $y_{i+1,j} = \sigma_{i,j}(y_{i,j}) \in X_{i,j}^L$ and hence
$$y_{i+1,j} \ \in \ ( \mathrm{Im}(\phi_{i,j}^L))^{-1} \cdot y_{i,j} \ \subseteq \ G_i \cdot y_{i,j}.$$
Similarly if an event occurs at $(i,j)$ with $i=j$ then there must be $L \in S_{i,i}$ with
$$y_{1,i+1} \ \in \ ( \mathrm{Im}(\phi_{i,i}^L))^{-1} \cdot y_{i,i} \ \subseteq \ G_i \cdot y_{i,i}.$$
We claim that no event occurs in any row higher than row $k$. If this claim holds then clause (d) will imply that $y_{k,n} \in V_{k,n} \cdot y$. To prove this claim, let $i$ be maximal such that there is some $j$ with an event occurring at $(i,j)$. Fix such a $j$. Note that if $j = n$ then by definition $i < k$ (we did not define, nor do we need, $y_{k+1,n}$). If $i \leq k$ then there is nothing to prove. So suppose that $i \geq k$, in which case $j < n$. By definition, either there is $L \in S_{i,i}$ with $y' := y_{1,i+1} \in X_{i,i}^L$ (if $i = j$) or there is $L \in S_{i,j}$ with $y' := y_{i+1,j} \in X_{i,j}^L$ (if $i < j$). In either case we have $L \in S_{i,n}$ since $S_{i,j} \subseteq S_{i,n}$. Since $j < n$ we deduce from clause (e$'$) that
$$\phi_{i,n}^L (16 \ell_{i,n}^L q_{i,n} \cdot A_{i,n}^L) \cdot y' \ \subseteq \ X_{i,n}^L.$$
Additionally, clause (d) and the maximality of $i$ implies that $y \in V_{i,n} \cdot y'$. Hence clause (b$'$) gives
$$x \ \in \ V_{k,n}^{-1} V_{i,n} \cdot y' \ \subseteq \ V_{i,n} V_{i,n} \cdot y' \ \subseteq \ \phi_{i,n}^L(\rZ_{i,n}^L) \cdot y' \subseteq X_{i,n}^L.$$
Therefore
$$\phi_{i,n}^L (15 \ell_{i,n}^L q_{i,n} \cdot A_{i,n}^L) \cdot x \ \subseteq \ \phi_{i,n}^L(16 \ell_{i,n}^L q_{i,n} \cdot A_{i,n}^L) \cdot y' \ \subseteq \ X_{i,n}^L,$$
which implies that $i \leq r_n(x) = k$. So $i \leq k$ as claimed. It follows that
$$y_{k,n} \ \in \ V_{k,n} \cdot y \ \subseteq \ V_{k,n} V_{k,n} \cdot x.$$
We are assuming $x \in X_{k,n}^H$ and thus clause (b$'$) gives
$$y_{k,n} \ \in \ \phi_{k,n}^H(\rZ_{k,n}^H) \cdot x.$$
Since, like $y$, $x \in V_{k,n} \cdot x$, an identical argument shows that $x_{k,n}$ (which is defined similarly to $y_{k,n}$) satisfies
$$x_{k,n} \ \in \ \phi_{k,n}^H(\rZ_{k,n}^H) \cdot x.$$
It follows from the definitions of the various equivalence relations in layer $m$ that
$$y \ E_{s,t}^m \ x \ \Longleftrightarrow \ y_{k,n} \ E_{k,n}^m \ x_{k,n}.$$
We conclude that $y \ E_{s,t}^m \ x$ since $\phi_{k,n}^H (\rZ_{k,n}^H) \cdot x \subseteq [x]_{E_{k,n}^m}$.

(v). Fix $1 \leq k \leq n \leq m$. If $n = m$ then there is currently only one equivalence relation, $E_{k,m}^m$, lying in row $k$ and column $n = m$. In this case clause (v) is automatic since $b_{k,n} \geq 1$. So suppose that $n < m$. Fix $H \in S_{k,n}$ and $x \in X_{k,n}^H$. Let $J$ be the set of $n < j \leq m$ such that
$$\phi_{k,n}^H(8 \cdot 2^{14 \ell_{k,n}^H} \cdot A_{k,n}^H) \cdot x \ \not\subseteq \ [x]_{E_{k,n}^j}.$$
Note that we exclude $n$ from $J$ and therefore must only show that $|J| \leq b_{k,n} - 1$. Set $t = r_{n+1}(x)$. Then $k \leq t \leq n+1$ since $x \in X_{k,n}^H$. By definition there is $L \in S_{t,n+1}$ such that
$$\phi_{t,n+1}^L(15 \ell_{t,n+1}^L q_{t,n+1} \cdot A_{t,n+1}^L) \cdot x \ \subseteq \ X_{t, n+1}^L.$$
Note that by the definition of $V_{t,n+1}$,
$$\phi_{k,n}^H(8 \cdot 2^{14 \ell_{k,n}^H} \cdot A_{k,n}^H) \ \subseteq \ V_{t,n+1}.$$
Therefore clause (iv) implies that for every $j \in J$,
$$\phi_{t,n+1}^L(\rZ_{t,n+1}^L) \cdot x \ \not\subseteq \ [x]_{E_{t,n+1}^j}.$$
Lemma \ref{LEM STRONG BND} now implies that for each $j\in J$ there is $1 \leq i \leq \ell_{t,n+1}^L \leq \ell_t$ such that
$$x \ \in \ \phi_{t,n+1}^L\left( 30 \ell_{t,n+1}^L \cdot (q_{t,n+1} \cdot A_{t,n+1}^L)\right) \cdot \partial_i^{\Phi_{t,n+1}^L}\left(E_{t,n+1}^j, q_{t,n+1} \cdot A_{t,n+1}^L\right).$$
Clause (iii) states that the equivalence relations lying in row $t$ and column $n+1$ are pairwise $(\Phi_{t,n+1}^L, q_{t,n+1} \cdot A_{t,n+1}^L)$-orthogonal. Thus, since $x$, $t$, and $L$ are fixed, we must have that $|J| \leq \ell_t \leq b_{k,n} - 1$. This completes the construction of the three-dimensional array of equivalence relations.

Now to complete the proof we argue that for all $x, y \in X$,
$$x\mathrel{E^X_G}y \ \Longleftrightarrow \ (\exists M) \, (\forall m \geq M) \ x \ E_{1,1}^m \ y.$$
Establishing this fact will indeed complete the proof since the relation described on the right is clearly hyperfinite by clause (i). Clause (i) implies that the right-hand side implies the left. So fix $x, y \in X$ with $x\mathrel{E^X_G}y$. Since $G$ is the union of the $G_k$'s, we can find $k$ with $G_k \cdot x = G_k \cdot y$. Let $H$ be the unique element of $T_k$ which is $G_k$-conjugate to $\Stab(x) \cap G_k$. Then $x \in X_k^H$ and from our discussion just after the listing of clauses (b$'$) and (e$'$) we have that there is $n(1)$ with $H \in S_{k,n}$ and $x \in X_{k,n}^H$ for all $n \geq n(1)$. By our comment following the definition of the function $r_n$, we have that $r_n(x) \geq k$ for all $n > n(1)$. Since $G_k = \bigcup_{n \geq k} V_{k,n}$, there is $n(2) > n(1)$ with $y \in V_{k,n} \cdot x$ for all $n \geq n(2)$. Fix any $n \geq n(2)$ and set $t = r_n(x) \geq k$. Let $L \in S_{t,n}$ be such that
$$\phi_{t,n}^L(15 \ell_{t,n}^L q_{t,n} \cdot A_{t,n}^L) \cdot x \ \subseteq \ X_{t,n}^L.$$
If $m \geq n$ and $\phi_{t,n}^L(\rZ_{t,n}^L) \cdot x \not\subseteq [x]_{E_{t,n}^m}$ then
$$x \ \in \ \phi_{t,n}^L\left(30 \ell_{t,n}^L \cdot (q_{t,n} \cdot A_{t,n}^L)\right) \cdot \partial_i^{\Phi_{t,n}^L}\left(E_{t,n}^m, q_{t,n} \cdot A_{t,n}^L\right)$$
for some $1 \leq i \leq \ell_{t,n}^L$ by Lemma \ref{LEM STRONG BND}. By clause (iii) the equivalence relations $E_{t,n}^m$ lying in row $t$ and column $n$ are pairwise $(\Phi_{t,n}^L, q_{t,n} \cdot A_{t,n}^L)$-orthogonal, and so the above scenario can occur at most $\ell_{t,n}^L$-many times. Thus
$$\phi_{t,n}^L(\rZ_{t,n}^L) \cdot x \ \subseteq \ [x]_{E_{t,n}^m}$$
for all but finitely many $m \geq n$. Now clause (iv) implies that
$$V_{t,n} \cdot x \ \subseteq \ [x]_{E_{1,1}^m}$$
for all but finitely many $m \geq n$. This completes the proof as $y \in V_{k,n} \cdot x$ and $V_{k,n} \subseteq V_{t,n}$ since $t \geq k$.
\end{proof}

\end{document}